\documentclass[11pt,a4paper]{amsart}
\usepackage{exscale,amssymb,amsmath,bbm,color,amsmath,amsthm,amsfonts,stmaryrd}
\usepackage{graphicx,fullpage,mathtools}
\usepackage{bbm,exscale,stmaryrd,wasysym}
\usepackage[foot]{amsaddr}
\usepackage{hyperref}
\usepackage{enumitem}

\newcommand{\E}[1]{\ensuremath{\mathbb{E} \left[#1 \right]}}
\newcommand{\Prob}[1]{\ensuremath{\mathbb{P} \left(#1 \right)}}

\newcommand{\var}[1]{\ensuremath{\mathrm{var} \left(#1 \right)}}

\newcommand{\I}[1]{\ensuremath{\mathbbm{1}_{ \{ #1 \} }}}
\newcommand{\R}{\ensuremath{\mathbb{R}}}
\newcommand{\Z}{\ensuremath{\mathbb{Z}}}
\newcommand{\N}{\ensuremath{\mathbb{N}}}
\newcommand{\Q}{\ensuremath{\mathbb{Q}}}

\newcommand{\fl}[1]{\ensuremath{\lfloor #1 \rfloor}}

\renewcommand{\subset}{\subseteq}

\newcommand{\convdist}{\ensuremath{\stackrel{d}{\longrightarrow}}}
\newcommand{\convprob}{\ensuremath{\stackrel{p}{\rightarrow}}}
\newcommand{\equidist}{\ensuremath{\stackrel{d}{=}}}

\usepackage[sc]{mathpazo}
\usepackage[T1]{fontenc}

\newenvironment{itemize*}%
  {\vspace{-0.3cm}%
  \begin{itemize}%
    \setlength{\itemsep}{0pt}%
    \setlength{\parskip}{0pt}}%
  {\end{itemize}}

\newenvironment{enumerate*}%
  {\vspace{-0.3cm}%
  \begin{enumerate}%
    \setlength{\itemsep}{0pt}%
    \setlength{\parskip}{0pt}}%
  {\end{enumerate}}

\newtheorem{thm}{Theorem}[section]
\newtheorem{lem}[thm]{Lemma}

\newtheorem{prop}[thm]{Proposition}
\newtheorem{cor}[thm]{Corollary}

\newtheorem{rem}[thm]{Remark}

\title{The stable graph: the metric space scaling limit of a critical random graph with i.i.d.\ power-law degrees}
\author{Guillaume Conchon-\,\hspace{-0.3mm}-Kerjan} 
\address{Laboratoire de probabilit\'es, statistique et mod\'elisation, Universit\'e Paris-Diderot, B\^atiment Sophie Germain, Case courrier 7012, 75205 Paris Cedex 13, France}
\email{gconchon@lpsm.paris}

\author{Christina Goldschmidt}
\address{Department of Statistics and Lady Margaret Hall, University of Oxford, 24-29 St Giles', Oxford OX1 3LB, UK}
\email{goldschm@stats.ox.ac.uk}

\date{29th July 2021}

\begin{document}

\maketitle

\begin{abstract}
We prove a metric space scaling limit for a critical random graph with independent and identically distributed degrees having power-law tail behaviour with exponent $\alpha+1$, where $\alpha \in (1,2)$.  The limiting components are constructed from random $\R$-trees encoded by the excursions above its running infimum of a process whose law is locally absolutely continuous with respect to that of a spectrally positive $\alpha$-stable L\'evy process. These spanning $\R$-trees are measure-changed $\alpha$-stable trees. In each such $\R$-tree, we make a random number of vertex-identifications, whose locations are determined by an auxiliary Poisson process.  This generalises results which were already known in the case where the degree distribution has a finite third moment (a model which lies in the same universality class as the Erd\H{o}s--R\'enyi random graph) and where the role of the $\alpha$-stable L\'evy process is played by a Brownian motion.
\end{abstract}

\section{Introduction}

\subsection{Overview}
In recent years, a wide variety of random graph models have been introduced and studied.  Many of these models undergo a phase transition of the following type: below some threshold, the connected components are microscopic in size (in the sense that they each contain a negligible proportion of the vertices) and possess few cycles, whereas above the threshold, there is a component which occupies a positive fraction of the vertices and contains many cycles, and all other components are again microscopic.  We are particularly interested in the behaviour exactly at the point of the phase transition, and in a precise description of the sizes and geometric properties of the components, which is typically much more delicate than in the sub- and supercritical cases.  We will first give a brief overview of the setting in which we are interested, and of our main results, deferring a more detailed account with proper definitions, as well as a summary of the pre-existing literature, to the next section.

We consider a uniform random graph on $n$ vertices with a given degree sequence, where the degrees themselves are independent and identically distributed random variables, $D_1, \ldots, D_n$.  (If $\sum_{i=1}^n D_i$ is odd, we replace $D_n$ by $D_n + 1$.)  For simplicity, we impose the condition that $\Prob{D_1 \ge 1} = 1$, so that there are no isolated vertices. We also assume that $\Prob{D_1 = 2} < 1$ (otherwise we have a random 2-regular graph, which contains many cycles of macroscropic size~\cite{ArratiaBarbourTavare}) and that $\var{D_1} < \infty$ (otherwise the graph behaves very differently; see \cite{DharavdHvLNewunivnovariance}).  The phase transition then occurs when the parameter $\theta := \E{D_1(D_1-1)}/\E{D_1}$ passes through 1: if $\theta < 1$ then the proportion of vertices in the largest component tends to 0 in probability as $n \to \infty$, whereas if $\theta > 1$, this proportion instead converges to a strictly positive constant, again in probability as $n \to \infty$.  

At the critical point $\theta = 1$, there is a sequence of components whose sizes are comparable (rather than a single giant component, as in the supercritical case) and which, even after rescaling, retain some randomness in the limit.  The sizes and geometric properties of these components depend on the tail of the distribution of $D_1$.  In particular, 
\begin{itemize}
\item if $\E{D_1^3} < \infty$ then the largest components have sizes on the order of $n^{2/3}$ and diameters on the order of $n^{1/3}$;
\item if $\Prob{D_1 = k} \sim c k^{-(\alpha+2)}$ for some constant $c> 0$ and $\alpha \in (1,2)$ then the largest components have sizes on the order of $n^{\alpha/(\alpha+1)}$ and diameters on the order of $n^{(\alpha-1)/(\alpha+1)}$.
\end{itemize}
(It will be convenient to refer to the first of these as the ``$\alpha = 2$ case''.)

These scaling properties are either proved or conjectured to be \emph{universal}, that is to hold for whole families of random graph models with similar asymptotic degree distributions.  We will discuss the issue of universality in some detail below.

\subsection{Our results}
Let us now state our scaling limit theorem.  Let $G_1^n, G_2^n, \ldots$ be the (vertex-sets of the) components of the critical random graph, listed in decreasing order of size, with ties broken arbitrarily. We think of these as measured metric spaces, by endowing $G_i^n$ with the graph distance, $d_i^n$, and the counting measure on its vertices, $\mu_i^n$.  Formally,  each is an element of the Polish space of isometry-equivalence classes of measured metric spaces endowed with the Gromov--Hausdorff--Prokhorov distance, which we will define properly below.

\begin{thm} \label{thm:main} Fix $\alpha \in (1,2]$.  Then under the conditions above, there exists a sequence of random compact measured metric spaces $(\mathcal{G}_1, d_1, \mu_1), (\mathcal{G}_2, d_2, \mu_2), \ldots$ (whose law depends on $\alpha$) such that, as $n \to \infty$,
\[
\left( \left(G_i^n, n^{-(\alpha-1)/(\alpha+1)} d_i^n, n^{-\alpha/(\alpha+1)} \mu_i^n\right), i \ge 1 \right) \convdist \left( \left(\mathcal{G}_i, d_i, \mu_i \right), i \ge 1 \right)
\]
in the sense of the product Gromov--Hausdorff--Prokhorov topology.
\end{thm}

The same result also holds for a multigraph with the same degree sequence generated according to the configuration model (see Section~\ref{sec:configmodel} for more details).

In the terminology of \cite{ABBrGoMi}, $(\mathcal{G}_i, d_i)$ is a random $\R$-graph for each $i \ge 1$.  For reasons which will shortly become clear, we refer to the whole limiting object as the \emph{$\alpha$-stable graph} if $\alpha \in (1,2)$ or the \emph{Brownian graph} (instances of which have already occurred several times in the literature) if $\alpha = 2$.

This theorem, in particular, implies the scaling properties mentioned above.  The $\alpha = 2$ case may be deduced from a more general theorem due to Bhamidi and Sen~\cite{BhamidiSen}, proved by different methods.  For $\alpha \in (1,2)$, Bhamidi, Dhara, van der Hofstad and Sen~\cite{BhDhvdHSeMetric} considered the setting of critical percolation on a \emph{supercritical} uniform random graph with given degree sequence, having similar tail behaviour to ours, and proved a scaling limit theorem in the sense of the product Gromov-weak topology.  (This has recently been improved to a convergence in the product Gromov--Hausdorff--Prokhorov topology for degree sequences satisfying certain conditions in \cite{BDvdHSinprep}.)  We will describe the results of \cite{BhDhvdHSeMetric} in more detail below and will, for the moment, simply observe that there are situations which are covered by both theorems, and where the limit objects must therefore be the same, but where this is certainly not obvious from their respective constructions.

One of the most striking aspects of our results is the characterisation of the limit spaces which we are able to give, which is completely new for $\alpha \in (1,2)$, and generalises one which was already known for $\alpha=2$.   In order to give this characterisation, we must first introduce some stochastic processes which play a key role.

Let $\mu = \E{D_1}$.  For $\alpha \in (1,2)$, let $L$ be a spectrally positive $\alpha$-stable L\'evy process with Laplace transform
\[
\E{\exp(-\lambda L_t)} = \exp \left(\frac{c \Gamma(2-\alpha)}{\mu \alpha(\alpha-1)} \lambda^{\alpha} t \right), \quad \lambda \ge 0, \quad t \ge 0,
\]
where $c>0$ is the constant such that $\Prob{D_1 = k} \sim c k^{-(\alpha+2)}$. Such a process can be thought of as encoding a forest of continuum trees; the standard way to do this goes via a (somewhat complicated) functional of $L$ called the height process $H$ (we will define this properly below). Let
\[
C_{\alpha} = \frac{c \Gamma(2-\alpha)}{\alpha(\alpha-1)}.
\] 
We will create a new pair $(\widetilde{L}, \widetilde{H})$ of processes via change of measure as follows: for suitable test-functions $f: \mathbb{D}([0,t], \R)^2 \to \R$, let
\begin{equation} \label{eqn:Ltildedefstable}
\E{f(\widetilde{L}_u, \widetilde{H}_u, 0 \le u \le t)} = \E{\exp\left(-\frac{1}{\mu} \int_0^t sdL_s - \frac{C_{\alpha} t^{\alpha+1}}{(\alpha+1) \mu^{\alpha+1}} \right) f(L_u, H_u, 0 \le u \le t)}.
\end{equation}
For $\alpha = 2$, letting $\mu = \E{D_1}$ and $\beta = \E{D_1(D_1-1)(D_1-2)}$, we instead take
\[
L_t = \sqrt{\frac{\beta}{\mu}} B_t \quad \text{and} \quad H_t = 2 \sqrt{\frac{\mu}{\beta}} \left(B_t - \inf_{0 \le s \le t} B_s \right),
\]
where $B$ is a standard Brownian motion (in the Brownian setting, the associated height process has the same distribution as a reflected Brownian motion, up to a scaling constant).  In this case, define
\begin{equation} \label{eqn:LtildedefBrownian}
\E{f(\widetilde{L}_u, \widetilde{H}_u, 0 \le u \le t)} = \E{\exp\left(-\frac{1}{\mu} \int_0^t sdL_s - \frac{\beta t^3}{6 \mu^3} \right) f(L_u, H_u, 0 \le u \le t)}.
\end{equation}
In either case, let
\[
R_t = \widetilde{L}_t - \inf_{0 \le s \le t} \widetilde{L}_s, \quad t \ge 0.
\]

Now write $(\zeta_i, i \ge 1)$ for the ordered sequence of lengths of excursions of $R$ above 0.  These excursions give rise to spanning $\R$-trees for the limiting components.  For $i \ge 1$, let $\widetilde{\varepsilon}_i: [0,\zeta_i] \to \R_+$ be the $i$th longest excursion of $R$ (with its argument translated in the natural way to $[0,\zeta_i]$). For $i \ge 1$, let $\widetilde{h}_i : [0,\zeta_i] \to \R_+$ be the corresponding (continuous) excursion of $H$ above 0 (which has the same length as $\widetilde{\varepsilon}_i$).  Let $(\widetilde{\mathcal T}_1, \widetilde{d}_1, \widetilde{\mu}_1), (\widetilde{\mathcal T}_2, \widetilde{d}_2, \widetilde{\mu}_2), \ldots$ be the measured $\R$-trees encoded by $\widetilde{h}_1, \widetilde{h}_2, \ldots$ respectively, and write $p_i$ for the canonical projection from $[0,\zeta_i]$ to $\widetilde{\mathcal{T}}_i$, for $i \ge 1$.  Conditionally on $R$, now consider a Poisson point process on $\R_+ \times \R_+$ of intensity $\frac{1}{\mu} \I{x \le R_t} dt dx$. (Equivalently, we can think of this as a Poisson point process of intensity $1/\mu$ in the area under the graph of $R$.)  The points tell us how to identify vertices in the $\R$-trees in order to create cycles.  For $i \ge 1$, suppose that a number $M_i \ge 0$ of points fall within the $i$th longest excursion $\widetilde{\varepsilon}_i$.  Given $\widetilde{\varepsilon}_i$, we then have $M_i \sim \mathrm{Poisson}\left(\frac{1}{\mu} \int_0^{\infty} \widetilde{\varepsilon}_i(u) du\right)$.  If $M_i \ge 1$, write 
\[
\left(s_{i,1}, x_{i,1}\right), \left(s_{i,2}, x_{i,2} \right), \ldots, \left(s_{i, M_i}, x_{i, M_i} \right)
\]
for the points themselves (with their first co-ordinates translated to the interval $[0,\zeta_i]$). For $i \ge 1$ and $1 \le k \le M_i$, let
\[
t_{i,k} = \inf\{t \ge s_{i,k}: \widetilde{\varepsilon}_i(t) \le x_{i,k}\}.
\]
Now for $i \ge 1$, let $(\mathcal{G}_i, d_i, \mu_i)$ be the measured metric space obtained from $(\widetilde{\mathcal{T}}_i, \widetilde{d}_i, \widetilde{\mu}_i)$ by making no change if $M_i = 0$ or, if $M_i \ge 1$, by identifying the $M_i$ pairs of points
\[
\left(p_i(s_{i,1}), p_i(t_{i,1})\right), \ldots, \left(p_i(s_{i,M_i}), p_i(t_{i, M_i})\right).
\]
(Formally, this is done by taking the quotient metric space in a standard way which is described in detail, for example, just before Lemma 21 of \cite{ABBrGo1}.) 

Conditionally on the ordered lengths $\zeta_1, \zeta_2, \ldots$ of the excursions and numbers $M_1, M_2, \ldots$ of Poisson points, we may give an attractive alternative description of the excursions encoding the spanning forests of the $\alpha$-stable and Brownian graphs. These are closely related to the canonical family of random $\R$-trees encompassing the Brownian continuum random tree \cite{AldousCRT1,AldousCRT2,AldousCRT3} and $\alpha$-stable trees~\cite{DuquesneLeGall, Duquesne}, which are the scaling limits of critical Galton--Watson trees conditioned to have size $n$ with offspring distribution in the domain of attraction of a Normal or $\alpha$-stable distribution respectively.

First consider $\alpha \in (1,2)$, and let $\mathbbm{e}$ be a normalised (i.e.\ length 1) excursion of the stable process $L$, and let $\mathbbm{h}$ be the associated normalised excursion of the height process, which would encode an $\alpha$-stable tree.  Now for $m \in \Z_+$, define tilted excursions $\widetilde{\mathbbm{e}}^{(m)}$ and $\widetilde{\mathbbm{h}}^{(m)}$ via
\begin{equation} \label{eqn:reln}
\E{g(\widetilde{\mathbbm{e}}^{(m)}, \widetilde{\mathbbm{h}}^{(m)})} = \frac{\E{g(\mathbbm{e}, \mathbbm{h}) \left( \int_0^1 \mathbbm{e}(u) du \right)^m} } {\E{\left( \int_0^1 \mathbbm{e}(u) du \right)^m} },
\end{equation}
for suitable test-functions $g: \mathbb{D}([0,1], \R_+) \times \mathbb{C}([0,1], \R_+) \to \R$. Let $(\widetilde{\mathcal{T}}^{(m)}, \widetilde{d}^{(m)}, \widetilde{\mu}^{(m)})$ be the $\R$-tree $(\widetilde{\mathcal{T}}^{(m)}, \widetilde{d}^{(m)})$ encoded by $\widetilde{\mathbbm{h}}^{(m)}$, along with its natural mass measure $\widetilde{\mu}^{(m)}$.  Write $\widetilde{p}^{(m)}$ for the projection $[0,1] \to \widetilde{\mathcal{T}}^{(m)}$.  If $m \ge 1$, now sample $m$ pairs of points in the tree as follows.  First pick $(s_1, x_1), \ldots, (s_m,x_m)$ independently and uniformly from the area below the excursion $\widetilde{\mathbbm{e}}^{(m)}$ and above the $x$-axis according to the normalised Lebesgue measure.  Define $t_i = \inf\{t \ge s_i: \widetilde{\mathbbm{e}}^{(m)}(t) \le x_i\}$. Finally, identify $\widetilde{p}^{(m)}(s_i)$ and $\widetilde{p}^{(m)}(t_i)$ for $1 \le i \le m$ in order to obtain $(\mathcal{G}^{(m)}, d^{(m)}, \mu^{(m)})$.  Set $(\mathcal{G}^{(0)}, d^{(0)}, \mu^{(0)}) = (\widetilde{\mathcal{T}}^{(0)}, \widetilde{d}^{(0)}, \widetilde{\mu}^{(0)})$.

Something very similar works in the Brownian case.  Here, we take $\mathbbm{e}$ to be a normalised Brownian excursion (which is, in particular, continuous); in this context, $\mathbbm{h} = 2\mathbbm{e}$, so there is no need to consider two different excursions.  The function $2 \mathbbm{e}$ encodes the Brownian continuum random tree (in the normalisation adopted by Aldous~\cite{AldousCRT1}).  Again define $\widetilde{\mathbbm{e}}^{(m)}$ as at (\ref{eqn:reln}) and let $(\widetilde{\mathcal{T}}^{(m)}, \widetilde{d}^{(m)}, \widetilde{\mu}^{(m)})$ be the measured $\R$-tree encoded by $2 \widetilde{\mathbbm{e}}^{(m)}$, and write $\widetilde{p}^{(m)}$ for the projection $[0,1] \to \widetilde{\mathcal{T}}^{(m)}$.  If $m \ge 1$, now sample $m$ pairs of points in the tree as follows.  First pick $(s_1, x_1), \ldots, (s_m,x_m)$ independently and uniformly from the area below the excursion $\widetilde{\mathbbm{e}}^{(m)}$ and above the $x$-axis according to the normalised Lebesgue measure.  Define $t_i = \inf\{t \ge s_i: \widetilde{\mathbbm{e}}^{(m)}(t) \le x_i\}$. Finally, identify $\widetilde{p}^{(m)}(s_i)$ and $\widetilde{p}^{(m)}(t_i)$ for $1 \le i \le m$ in order to obtain $(\mathcal{G}^{(m)}, d^{(m)}, \mu^{(m)})$.  Set $(\mathcal{G}^{(0)}, d^{(0)}, \mu^{(0)}) = (\widetilde{\mathcal{T}}^{(0)}, \widetilde{d}^{(0)}, \widetilde{\mu}^{(0)})$.

\begin{thm} \label{thm:conditionaldescription}
Conditionally on the lengths $\zeta_1, \zeta_2, \ldots$ of the excursions and the numbers $M_1, M_2, \ldots$ of points, the measured $\R$-graphs $(\mathcal{G}_1, d_1, \mu_1)$, $(\mathcal{G}_2, d_2, \mu_2), \ldots$ are independent with
\[
\left(\mathcal{G}_i, d_i, \mu_i\right) \equidist \left(\mathcal{G}^{(M_i)}, \zeta_i^{(\alpha-1)/\alpha} d^{(M_i)}, \zeta_i \mu^{(M_i)} \right)
\]
for each $i \ge 1$.
\end{thm}

This shows that, in order to understand further the geometric properties of our limit object, a key role will be played by the family of random $\R$-graphs $((\mathcal{G}^{(m)}, d^{(m)}, \mu^{(m)}), m \ge 0)$.  These are studied in depth for $\alpha \in (1,2)$ in the companion paper~\cite{GoldschmidtHaasSenizergues}; the Brownian case was the subject of the earlier paper \cite{ABBrGo2}.    From the absolute continuity relation (\ref{eqn:reln}), for any $\alpha \in (1,2]$ it is straightforward to see that $(\mathcal{G}^{(m)}, d^{(m)}, \mu^{(m)})$ has Hausdorff dimension $\alpha/(\alpha-1)$ almost surely, since this is true of the appropriate Brownian/$\alpha$-stable tree and one cannot change the fractal dimension by making finitely many vertex-identifications.  

The branch-points of the $\alpha$-stable tree are almost surely all infinitary (i.e.\ removing any of them breaks the tree into infinitely many connected components), and this property is inherited, via absolute continuity, by $(\widetilde{\mathcal{T}}^{(m)}, \widetilde{d}^{(m)})$ for $\alpha \in (1,2)$.  It follows from the properties of the excursion $\mathbbm{e}^{(m)}$ (see \cite{GoldschmidtHaasSenizergues} for an in-depth discussion) that the vertex-identifications in $(\widetilde{\mathcal{T}}^{(m)}, \widetilde{d}^{(m)})$ are   almost surely all from a leaf to a branch-point of infinite degree.  In contrast, in the $\alpha=2$ case, the vertex-identifications are almost surely all from a leaf to a point of degree 2 (see \cite{ABBrGo1,ABBrGo2}).

In \cite{ABBrGo2,GoldschmidtHaasSenizergues}, it is further shown that one may explicitly determine the law of the \emph{kernel} of $\mathcal{G}^{(m)}$ (that is, the multigraph with edge-lengths which encodes its cycle structure), and that $\mathcal{G}^{(m)}$ may be constructed by gluing together randomly rescaled Brownian/stable trees.  Finally, it is shown in \cite{ABBrGo2,GoldschmidtHaasSenizergues} that $\mathcal{G}^{(m)}$ possesses a \emph{line-breaking construction}, that is, a recursive construction which starts from the kernel and successively glues on line-segments of random lengths to random points, obtaining a convergent sequence of approximations to the final $\R$-graph.

\section{Background} \label{sec:background}

In this section, we give some background material on our random graph model, and discuss the previously known results on its critical behaviour.  We also give a brief account of the scaling limit theory for Galton--Watson trees.  We then give an overview of the proof of Theorem~\ref{thm:main}.  This is followed by a brief summary of some related literature, and some open problems.  Finally, at the end of this section, we give a plan of the rest of the paper.

For a sequence of random variables $(A_n)_{n \ge 0}$ and a sequence $(a_n)_{n \ge 0}$ of real numbers, we write $A_n = O_{\mathbb{P}}(a_n)$ to mean that $(A_n/a_n)_{n \ge 0}$ is tight.  We write $A_n = \Theta_{\mathbb{P}}(a_n)$ to mean that $A_n = O_{\mathbb{P}}(a_n)$ and $A_n^{-1} = O_{\mathbb{P}}(a_n^{-1})$.  We write $A_n = o_{\mathbb{P}}(a_n)$ to mean $A_n/a_n \convprob 0$ as $n \to \infty$.

\subsection{The configuration model} \label{sec:configmodel}
We wish to sample a graph uniformly at random from among the graphs with the given degrees $D_1, D_2, \ldots, D_n$.  There is a standard method for doing this, which originated (in varying degrees of generality) in the work of Bender and Canfield~\cite{BenderCanfield}, Bollob\'as~\cite{Bollobas} and Wormald~\cite{Wormald}, called the \emph{configuration model}.  (We refer the reader to the recent book of van der Hofstad~\cite{RemcoBook} for a full account of the configuration model and for proofs of the results quoted below.)  We begin by first describing the setting where the vertex degrees are deterministic.  More precisely, suppose that we have vertices labelled $1, 2, \ldots, n$ where vertex $i$ has degree $d_i$, for $1 \le i \le n$.  Suppose that $d_i \ge 1$ for all $1 \le i \le n$ and that $\sum_{i=1}^n d_i$ is even.  To vertex $i$, attach $d_i$ \emph{stubs} or \emph{half-edges}.  Label the $\sum_{i=1}^n d_i$ half-edges in some arbitrary way, and then choose a pairing of them, uniformly at random.  Join the paired half-edges together to make full edges, and then forget the labelling of the half-edges.  In general, this procedure yields a multigraph (i.e.\ with self-loops, or multiple edges).  However, if there exist one or more simple graphs with the given degree sequence (that is, if the degree sequence is \emph{graphical}) then, conditionally on the event that the configuration model yields a simple graph, that graph is \emph{uniform} among the possibilities.

We are concerned with the setting where the degrees themselves are independent and identically distributed random variables $D_1, D_2, \ldots, D_n$.  An immediate issue is that we cannot guarantee that $\sum_{i=1}^n D_i$ is even.  We get around this problem by always assuming that if $\sum_{i=1}^n D_i$ is odd, then we in fact give vertex $n$ degree $D_n + 1$.  For the regime and properties in which we are interested, this makes only a negligible difference, and we will ignore it in the sequel.  Write $\nu = (\nu_k)_{k \ge 1}$ for the probability mass function of $D_1$, that is $\nu_k = \Prob{D_1 = k}$, $k \ge 1$.  Let $\mathbf{M}_n(\nu)$ be the multigraph resulting from the configuration model with these degrees.  It remains to resolve the issue that the degree sequence may, in principle, be non-graphical.  However, it is possible to show that if $D_1$ has finite variance and $\theta = \theta(\nu) =\E{D_1(D_1-1)}/\E{D_1}$ then
\[
\lim_{n \to \infty} \Prob{\mathbf{M}_n(\nu) \text{ is simple}} = \exp(- \theta/2 - \theta^2/4),
\]
and the right-hand side is strictly positive (see for instance Proposition 7.13 of \cite{RemcoBook}).  Let $\mathbf{G}_n(\nu)$ be a graph with the distribution of $\mathbf{M}_n(\nu)$ conditioned to be simple; this is our uniform random graph with i.i.d.\ $\nu$-distributed degrees, and is the main object of study in this paper. 

 If $\nu_k \sim c k^{-(\alpha+2)}$ for some $\alpha \in (1,2)$ as $k \to \infty$, we will have that $\max_{1 \le i \le n} D_i = \Theta_{\mathbb{P}}(n^{1/(\alpha+1)})$.  We will see in the sequel that vertices of degree $\Theta(n^{1/{(\alpha+1)}})$ play an important role in the structure of the graph, and ``show up'' in the scaling limit as vertices of infinite degree (often known as \emph{hubs}).  However, since $\alpha > 1$, with high probability we will not observe edges directly joining two vertices of degree $\Theta(n^{1/(\alpha+1)})$ and, indeed, the vertices of highest degree will be typically well-separated.  If $\E{D_1^3} < \infty$, on the other hand, then $\max_{1 \le i \le n} D_i = o_{\mathbb{P}}(n^{1/3})$ and there are no hubs in the limit.

An important property of the configuration model is that the pairing of the edges may be generated in a progressive manner.  This makes possible the use of an exploration process in order to capture properties of the (multi-)graph.  We do this in a depth-first manner, conditionally on the vertex-degrees, and making use of the arbitrary labelling we gave the half-edges, as follows.  Start from a vertex $v$ chosen with probability proportional to its degree $D_v$.  We will maintain a \emph{stack}, namely an ordered list of half-edges which we have seen but not yet explored.  Put the $D_v$ half-edges attached to $v$ onto this stack, in increasing order of label, so that the lowest labelled half-edge is on top of the stack.  At every subsequent step, if the stack is non-empty, take the top half-edge and sample its pair uniformly at random from those available (i.e.\ the others on the stack and those which we have not yet observed in our exploration).  If the pair half-edge belongs to a vertex $w$ which has not yet been observed (i.e.\ if the pair half-edge does not lie in the stack), remove the paired half-edges from the system, and add the remaining $D_w - 1$ half-edges attached to $w$ to the top of the stack, again in increasing order of label.  If ever the stack becomes empty, select a new vertex with probability proportional to its degree, and put all of its half-edges onto the stack.  Repeat until the whole graph has been exhausted.  Notice that the stack is empty at the end of a step if and only if we have reached the end of a component, and that in each step except the one at the start of a component, we pair two half-edges.  Let $R^n(k)$ be the size of the stack at step $k$.  Then, for example, we may read off the numbers of edges in the successive components as the lengths minus 1 of the excursions above 0 of the process $(R^n(k), k \ge 0)$.  It turns out that this process, as we shall explain below, in fact encodes much more information about the multigraph.

Write $|G|$ for the size of the vertex set of a graph $G$.  For a connected graph $G$, write $s(G)$ for its \emph{surplus}, that is how many more edges it has than any of its spanning trees (which necessarily have $|G|-1$ edges).  Write $G_1^n, G_2^n, \ldots$ for the connected components of $\mathbf{G}_n(\nu)$, in decreasing order of size, with ties broken arbitrarily.  Similarly, write $M_1^n, M_2^n, \ldots$ for the ordered connected components of $\mathbf{M}_n(\nu)$. 

\subsection{The phase transition and critical behaviour of the component sizes}

As we have already mentioned, $\mathbf{G}_n(\nu)$ undergoes a phase transition in its component sizes depending on its parameters~\cite{MolloyReed1, MolloyReed2, JansonLuczak}.  Indeed, if $\theta(\nu) \le 1$, then the largest connected component $G_1^n$ of $\mathbf{G}_n(\nu)$ is such that $|G_1^n|/n \convprob 0$.  On the other hand, if $\theta(\nu) > 1$ then $|G_1^n|/n \convprob \rho(\nu)$, where $\rho(\nu)$ is some strictly positive constant.  These results also hold for $\mathbf{M}_n(\nu)$.  To give an idea of why the quantity $\theta(\nu)$ is important, imagine performing the depth-first exploration outlined above, but ignoring any edges which create cycles.  Then it is not hard to see that, at each step which is not the start of a component, the degree of the vertex to which the half-edge on the top of the stack connects (as long as it does not connect to something on the stack and thus create a cycle) is a \emph{size-biased pick} from among the remaining possibilities.  So, at least close to the beginning, the exploration process should look approximately like a branching process with offspring distribution given by $D^* - 1$, where $\Prob{D^* = k} = k \nu_k/\E{D_1}$.  But then $\theta(\nu) = \E{D^* - 1}$, and so the critical point for the approximating branching process is indeed $\theta(\nu) = 1$.  Our interest is in this precisely critical case, and a significant part of this paper is devoted to making the heuristic argument just outlined precise.

The following theorem, due to Joseph~\cite{Joseph}, summarises some of the possible behaviours for the component sizes in the case $\theta(\nu) = 1$. A version of part (i) was proved independently by Riordan~\cite{Riordan} (see below for further discussion).  Let
\[
\ell_{\downarrow}^2 = \left\{(x_1, x_2, \ldots) \in \R^{\N}: x_1 \ge x_2 \ge \ldots \ge 0, \sum_{i \ge 1} x_i^2 < \infty\right\}.
\]

\begin{thm} \label{thm:JosephRiordan}
\begin{enumerate}
\item[(i)] Suppose that $\Prob{D_1 = 2} < 1$, $\E{D_1} = \mu$ and \linebreak $\E{D_1(D_1-1)(D_1-2)} = \beta$.  Let $B$ be a standard Brownian motion, and let
\begin{equation} \label{eqn:RBrownian}
\widetilde{L}_t = \sqrt{\frac{\beta}{\mu}} B_t - \frac{\beta}{2 \mu^2} t^2, \quad t \ge0 \qquad \text{and} \qquad R_t = \widetilde{L}_t - \inf_{0 \le s \le t} \widetilde{L}_s, \quad t \ge 0.
\end{equation}
Then 
\[
n^{-2/3} (|G_1^n|, |G_2^n|, \ldots) \convdist (\zeta_1, \zeta_2, \ldots)
\]
as $n \to \infty$ in $\ell_{\downarrow}^2$, where $(\zeta_1, \zeta_2, \ldots)$ are the lengths of the excursions above 0 of the process $(R_t)_{t \ge 0}$. The same result holds with $(|G_1^n|, |G_2^n|, \ldots)$ replaced by $(|M_1^n|, |M_2^n|, \ldots)$.
\item[(ii)] Suppose that $\lim_{k \to \infty} k^{\alpha+2} \Prob{D_1 = k} = c$ for some constant $c > 0$ and some $\alpha \in (1,2)$, and that $\E{D_1} = \mu$.  Let $X$ be the process with independent increments whose law is specified by its Laplace transform
\begin{equation*}
\E{\exp(-\lambda X_t)}  = \exp \left( \int_0^t ds \int_0^{\infty} dx (e^{-\lambda x}-1+ \lambda x) \frac{c}{\mu}\frac{1}{x^{\alpha+1}} e^{-xs/\mu} \right), \ \lambda \ge 0, \ t \ge 0,
\end{equation*}
and let
\begin{equation} \label{eqn:Rstable}
\widetilde{L}_t = X_t - \frac{c \Gamma(2-\alpha)}{\alpha(\alpha-1) \mu^{\alpha}} t^{\alpha}, \quad t \ge 0 \qquad \text{and} \qquad  R_t = \widetilde{L}_t - \inf_{0 \le s \le t} \widetilde{L}_s \quad t \ge 0.
\end{equation}
Then 
\[
n^{-\alpha/(\alpha+1)}(|M_1^n|, |M_2^n|, \ldots) \convdist (\zeta_1, \zeta_2, \ldots)
\]
as $n \to \infty$ in $\ell^2_{\downarrow}$, where $(\zeta_1, \zeta_2, \ldots)$ are the lengths of the excursions above 0 of the process $(R_t)_{t \ge 0}$.
\end{enumerate}
\end{thm}

The sequences $(\zeta_1, \zeta_2, \ldots)$ appearing in Theorem~\ref{thm:JosephRiordan} must, of course, have the same distributions as the lengths of the excursions above 0 of the processes $(R_t, t \ge 0)$ from the Introduction.  Indeed, the processes $\tilde{L}$ defined in (\ref{eqn:RBrownian}) and (\ref{eqn:Rstable}) have the same distributions as those defined at (\ref{eqn:LtildedefBrownian}) and (\ref{eqn:Ltildedefstable}), respectively.  We prove this in Proposition~\ref{prop:stablemeasch} below.

Joseph~\cite{Joseph} conjectures that Theorem~\ref{thm:JosephRiordan} (ii) should also hold with $\mathbf{M}_n(\nu)$ replaced by $\mathbf{G}_n(\nu)$.  We show in the sequel (Section~\ref{subsec:backedges}) that this is indeed true (this has been proved independently by Dhara, van der Hofstad, van Leeuwaarden and Sen~\cite{DharavdHvLSen2}).  In consequence, all of our scaling limit results hold interchangeably for $\mathbf{G}_n(\nu)$ and $\mathbf{M}_n(\nu)$.

The common structure exhibited by the two parts of Theorem~\ref{thm:JosephRiordan} is no coincidence.  In both cases, the proof proceeds via an exploration of the graph similar to the one described earlier.  As outlined above, locally, the components resemble critical branching processes.  Since the components have small surplus, the lengths of the excursions of the stack-size process above 0 approximately encode the component sizes.  Moreover, the stack-size process behaves approximately like a reflected random walk.  A weak convergence result for the stack-size process then yields the convergence of the component sizes.

Riordan~\cite{Riordan}, in fact, proves a more refined version of Theorem~\ref{thm:JosephRiordan} (i), but under the (non-optimal) assumption that the degrees are bounded.  Firstly, his results are stated for a uniform random graph with a given $n$-dependent deterministic degree sequence $(d_i^{(n)})_{i \ge 1}$, where the moment conditions on $D_1$ are replaced by appropriate convergence results for the moments of the degree of a uniformly chosen vertex.  In particular, he is able to consider the components anywhere in the critical window, rather than precisely at $\theta = 1$.  (We refer the reader to \cite{Riordan} for the details.)  Secondly, he takes account also of the surplus of each component.  Jointly with the convergence of the rescaled component sizes, he shows that
\[
(s(G_1^n), s(G_2^n), \ldots) \convdist (M_1, M_2, \ldots)
\]
for a non-trivial random sequence $(M_1, M_2, \ldots) \in \Z_+^{\N}$.  The sequence $(M_1, M_2, \ldots)$ is again obtained using the process $R$ in (\ref{eqn:RBrownian}): on top of the graph of the random function $R$, superpose a Poisson point process of intensity $1/\mu$ in the plane.  Then $M_i$ is the number of points falling in the area beneath the excursion $\widetilde{\varepsilon}_i$ and above the $x$-axis, for $i \ge 1$.

The first result of this kind was proved by Aldous~\cite{AldousCritRG} for the Erd\H{o}s-R\'enyi random graph, $G(n,p)$ at its critical point.  More precisely, consider the graph $\mathbf{G}^{\text{ER}}_n$ obtained by taking $n$ vertices and connecting any pair of them by an edge independently with probability $p=1/n$.  Write $G_1^{\text{ER}, n}, G_2^{\text{ER}, n}, \ldots$ for the  components listed in decreasing order of size.  Define $\widetilde{L}$ and $R$ as in (\ref{eqn:RBrownian}) with $\beta=\mu=1$, let $\zeta_1, \zeta_2, \ldots$ be the lengths of the excursions of $R$ and let $M_1, M_2, \ldots$ be the numbers of points of a Poisson process of intensity 1 falling in each excursion.

\begin{thm}[Aldous~\cite{AldousCritRG}] \label{thm:Aldous}
As $n \to \infty$,
\[
\left( n^{-2/3}(|G_1^{\mathrm{ER}, n}|, |G^{\mathrm{ER}, n}_2|, \ldots), (s(G^{\mathrm{ER}, n}_1), s(G_1^{\mathrm{ER}, n}), \ldots) \right) \convdist ((\zeta_1, \zeta_2, \ldots), (M_1, M_2, \ldots))
\]
where the convergence is in $\ell^2_{\downarrow}$ for the component sizes and in the sense of the product topology for the surpluses.
\end{thm}

The limit is the same as in Theorem~\ref{thm:JosephRiordan} (i) in the case $\beta = \mu = 1$. 
This should be intuitively unsurprising, since the vertex degrees in the Erd\H{o}s--R\'enyi random graph approximately behave like i.i.d.\ Poisson(1) random variables, for which Theorem~\ref{thm:JosephRiordan} would apply with $\beta=\mu=1$. 
(Aldous' theorem in fact treats the whole critical window, i.e.\ $G(n,1/n + \lambda n^{-4/3})$ for $\lambda \in \R$.  The effect is to introduce an extra drift of $\lambda t$ into the process $\widetilde{L}$; we omit the very similar details for the sake of brevity.)

\subsection{Branching processes and their metric space scaling limits} \label{sec:branching}
As alluded to above, the components of our critical random graphs behave approximately like critical branching process trees.  It will be useful to spend a little time now exploring what happens in the true branching process setting, since what we do later will be analogous.  Suppose that we take a sequence of i.i.d.\ Galton--Watson trees, with offspring distribution represented by some non-negative random variable $Y$ with $\E{Y}=1$ and $\Prob{Y = 1} < 1$. (This entails that each of the trees has finite size almost surely.)  We use the standard encoding of this forest in terms of its \emph{{\L}ukasiewicz path} or \emph{depth-first walk}, given by $S(0) = 0$ and $S(k) = \sum_{i=1}^k (Y_i - 1)$ for $k \ge 1$ (see Le Gall~\cite{LeGall} or Duquesne and Le Gall~\cite{DuquesneLeGall} for more details).   Here, as usual, we explore the vertices of the forest in depth-first order, and $Y_i$ is the number of children of the $i$th vertex that we visit; these get added to the stack to await processing.  The stack-size process is essentially a reflected version of $S$, given by $(1 + S(k) - \min_{0 \le j \le k} S(j))_{k \ge 0}$.  It is  straightforward to see that the individual trees correspond to excursions above the running minimum of $(S(k))_{k \ge 0}$; it is technically easier to work with the depth-first walk than with the stack-size process, since $S$ it is an \emph{unreflected} random walk.  An even more convenient encoding of the forest is given by the height process, which tracks the generation of the successive vertices listed in depth-first order.  (It is, however, considerably harder to understand its distribution.)  In terms of the depth-first walk, the height process $(G(n))_{n \ge 0}$ is defined by $G(0) = 0$ and
\begin{equation}\label{def:height}
G(n) := \#\left\{j \in \{0,1,\ldots,n-1\}: S(j) = \min_{j \le k \le n} S(k) \right\}.
\end{equation}
The different trees now correspond to excursions above 0 of $G$.

The following generalised functional central limit theorem indicates some of the possible scaling limits for $S$ in this setting (see, for example, Theorem 3.7.2 of Durrett~\cite{Durrett}).

\begin{thm} \label{thm:FCLT}
\begin{enumerate}
\item[(i)] Suppose that $\E{Y} = \sigma^2 < \infty$.  Then
\[
n^{-1/2} (S(\fl{n t}), t \ge 0) \convdist \sigma (B_t, t \ge 0)
\]
as $n \to \infty$, in $\mathbb{D}(\R_+, \R)$, where $B$ is a standard Brownian motion.
\item[(ii)] Suppose that $\lim_{k \to \infty} k^{\alpha+1} \Prob{Y = k} = c$ for some constant $c > 0$ and some $\alpha \in (1,2)$.  Then
\[
n^{-1/\alpha}(S(\fl{n t}), t \ge 0) \convdist (L_t, t \ge 0)
\]
as $n \to \infty$, in $\mathbb{D}(\R_+, \R)$, where $L$ is a spectrally positive $\alpha$-stable L\'evy process, with Laplace transform
\[
\E{\exp(-\lambda L_t)} = \exp \left(\frac{c \Gamma(2-\alpha)}{\alpha(\alpha-1)} \lambda^{\alpha} t \right), \quad \lambda \ge 0, \quad t \ge 0.
\]
\end{enumerate}
\end{thm}

We now turn to the behaviour of the height process.  In the Brownian case, this turns out to be asymptotically the same as that of the reflected depth-first walk, up to a scaling constant.  In the stable case, however, matters are a little more complicated.  Consider the $\alpha$-stable L\'evy process $L$.  Chapter 1 of Duquesne \& Le Gall~\cite{DuquesneLeGall} shows that it is possible to make sense of a corresponding continuous height process, defined as follows.  First, for $0 \le s \le t$, let $\hat{L}^{(t)}_s = L_t - L_{(t-s)-}$ and let $\hat{M}^{(t)}_s = \sup_{0 \le r \le s} \hat{L}_r^{(t)}$. Then define $H_t$ to be the local time at level 0 of the process $\hat{L}^{(t)} - \hat{M}^{(t)}$.  We may choose the normalization in such a way that
\begin{equation} \label{eqn:heightprocess}
H_t = \lim_{\epsilon \downarrow 0} \frac{1}{\epsilon} \int_0^t \I{\hat{M}^{(t)}_s -\hat{L}^{(t)}_s \le \epsilon} ds 
\end{equation}
in probability. Theorem 1.4.3 of \cite{DuquesneLeGall} shows that $H$ has continuous sample paths with probability 1, and so we may (and will) work with a continuous version in the sequel.  Corollary 2.5.1 of \cite{DuquesneLeGall} entails the following joint convergences.

\begin{thm}[Duquesne \& Le Gall~\cite{DuquesneLeGall}] \label{thm:DLGconvergence}
\begin{enumerate}
\item[(i)] Suppose that $\E{Y} = \sigma^2 < \infty$.  Then
\[
(n^{-1/2} S(\fl{n t}), n^{-1/2} G(\fl{n t}), t \ge 0) \convdist  \left(\sigma B_t, \frac{2}{\sigma} \left(B_t - \inf_{0 \le s \le t}B_s\right),  t \ge 0 \right).
\]
as $n\to \infty$ in $\mathbb{D}(\R_+,\R^2)$.
\item[(ii)] Suppose that $\lim_{k \to \infty} k^{\alpha+1} \Prob{Y = k} = c$ for some constant $c > 0$ and some $\alpha \in (1,2)$.  Then we have
\[
\left(n^{-1/\alpha} S(\fl{n t}), n^{-(\alpha - 1)/\alpha} G(\fl{n t}), t \ge 0 \right) \convdist (L_t, H_t, t \ge 0)
\]
as $n\to \infty$ in $\mathbb{D}(\R_+,\R^2)$.
\end{enumerate}
\end{thm}

There is also a conditional version of Theorem~\ref{thm:DLGconvergence} for the depth-first walk $S^n$ and height process $G^n$ of a \emph{single} Galton--Watson tree, conditioned to have total progeny $n$.  (Let us assume that $\Prob{Y=k} > 0$ for all $k \ge 0$, so that the event of having total progeny $n$ has positive probability for all $n$; this is not really necessary, but will facilitate the statement of the theorem.)

\begin{thm} \label{thm:singletree}
\begin{enumerate}
\item[(i)] \emph{(Marckert \& Mokkadem~\cite{MarckertMokkadem}).} Suppose that $\E{Y} = \sigma^2 < \infty$.  Then
\[
(n^{-1/2} S^n(\fl{n t}), n^{-1/2} G^n(\fl{n t}), 0 \le t \le 1) \convdist  \left(\sigma \mathbbm{e}(t), \frac{2}{\sigma} \mathbbm{e}(t),  0 \le t \le 1\right),
\]
as $n\to \infty$ in $\mathbb{D}([0,1],\R^2)$, where $\mathbbm{e}$ is a standard Brownian excursion.
\item[(ii)] \emph{(Duquesne~\cite{Duquesne}).} Suppose that $\lim_{k \to \infty} k^{\alpha+1} \Prob{Y = k} = c$ for some constant $c > 0$ and some $\alpha \in (1,2)$.  Then we have
\[
\left(n^{-1/\alpha} S^n(\fl{n t}), n^{-(\alpha - 1)/\alpha} G^n(\fl{n t}), t \ge 0 \right) \convdist (\mathbbm{e}(t), \mathbbm{h}(t), 0 \le t \le 1)
\]
as $n\to \infty$ in $\mathbb{D}([0,1],\R^2)$, where $\mathbbm{e}$ is a normalised excursion of $L$ and $\mathbbm{h}$ is the corresponding normalised excursion of $H$.
\end{enumerate}
\end{thm}

We now describe briefly how a limiting height process excursion may be used to define a limit $\R$-tree (and the reader to the survey paper of Le Gall~\cite{LeGall} for more details).  Suppose first that $h: [0,\zeta_h] \to \R_+$ is any continuous function such that $h(0) = h(\zeta_h) = 0$.  Define a pseudo-metric on $[0,1]$ via
\[
d_h(x,y) = h(x) + h(y) - 2 \min_{x \wedge y \le z \le x \vee y} h(z), \quad x,y \in [0,\zeta_h].
\]
Define an equivalence relation $\sim$ on $[0,\zeta_h]$ by declaring $x \sim y$ if $d_h(x,y) = 0$.  Now let $\mathcal{T}_h = [0,\zeta_h]/{\sim}$ and endow it with the distance $d_h$ in order to obtain a metric space. This metric space is compact: one can easily check that it is a Hausdorff space, and that it is sequentially compact. Then $(\mathcal{T}_h, d_h)$ is the $\R$-tree encoded by $h$. Write $p_h: [0,\zeta_h] \to \mathcal{T}_h$ for the canonical projection.  We may additionally endow $(\mathcal{T}_h, d_h)$ with a natural ``uniform'' measure $\mu_h$ having total mass $\zeta_h$, obtained as the push-forward of the Lebesgue measure on $[0,\zeta_h]$ onto the tree.  Write $\mathbb{M}$ for the space of compact metric spaces each endowed with a finite (non-negative) Borel measure, up to measure-preserving isometry. We equip $\mathbb{M}$ with the Gromov--Hausdorff--Prokhorov distance $\mathrm{d_{GHP}}$, defined as follows.  (See Section 2.1 of \cite{ABBrGoMi} for more details and proofs of the results claimed below, as well as further references to the literature.) Let $(X,d,\mu)$ and $(X',d',\mu')$ be elements of $\mathbb{M}$.  We say that $C$ is a \emph{correspondence} between $X$ and $X'$ if $C \subset X \times X'$ and, whenever $x \in X$, there exists $x' \in X'$ such that $(x,x') \in C$ and vice versa.  The \emph{distortion} of the correspondence $C$ is
\[
\mathrm{dist}(C) := \sup\{|d(x_1, x_2) - d'(x_1', x_2')|: (x_1, x_1'), (x_2, x_2') \in C\}.
\]
Write $\mathcal{C}(X,X')$ for the set of correspondences between $X$ and $X'$.  Write $M(X,X')$ for the set of non-negative Borel measures on $X \times X'$.  Write $p$ and $p'$ for the canonical projections from $X \times X'$ to $X$ and $X'$ respectively.  We define the \emph{discrepancy} of $\pi \in M(X,X')$ with respect to $\mu$ and $\mu'$ to be
\[
\mathrm{disc}(\pi;\mu, \mu') = \|\mu - p_* \pi\| + \|\mu' - p'_*\pi \|,
\]
where $\|\nu\|$ is the total variation of the signed measure $\nu$. We define the Gromov--Hausdorff--Prokhorov distance by
\[
\mathrm{d_{GHP}}((X,d,\mu),(X',d',\mu')) :=  \inf_{C \in \mathcal{C}(X,X'), \ \pi \in M(X,X')} \left\{ \frac{1}{2}\mathrm{dist}(C) \vee \mathrm{disc}(\pi; \mu,\mu') \vee \pi(C^c) \right\}.
\]
Then $(\mathbb{M}, \mathrm{d_{GHP}})$ is a Polish space.  We observe a very useful upper bound for the Gromov--Hausdorff--Prokhorov distance between $\R$-trees encoded by continuous excursions:
\begin{align}
& \mathrm{d_{GHP}}((\mathcal{T}_h, d_h, \mu_h), (\mathcal{T}_g, d_g, \mu_g)) \notag \\
& \quad \le 2 \max \left\{ \sup_{0 \le x \le \zeta_h \wedge \zeta_g} |h(x) - g(x)|,  \sup_{\zeta_h\wedge \zeta_g < x \le \zeta_h} h(x) + \sup_{\zeta_h\wedge \zeta_g < x \le \zeta_g} g(x)  , \frac{1}{2} |\zeta_h-\zeta_g|\right\},  \label{eqn:GHbound}
\end{align}
The \emph{random} $\R$-trees encoded by $2 \mathbbm{e}$ for $\alpha = 2$ and $\mathbbm{h}$ for $\alpha \in (1,2)$ are known as the \emph{Brownian continuum random tree}, for which we will write $(\mathcal{T}^{(2)}, d^{(2)})$ (with mass measure $\mu^{(2)}$), and the \emph{$\alpha$-stable tree}, for which we will write $(\mathcal{T}^{(\alpha)}, d^{(\alpha)})$ (with mass measure $\mu^{(\alpha)}$), respectively.  (Note that because of our choice of Laplace exponent, this is a constant multiple of the usual $\alpha$-stable tree.) We note the important point that, for any $\alpha \in (1,2]$, the mass measure $\mu^{(\alpha)}$ is concentrated on the leaves of $\mathcal{T}^{(\alpha)}$.

Let $T_n$ be our Galton--Watson tree conditioned to have size $n$.  The natural way to take a scaling limit of the tree itself is to consider it as a metric space using the graph distance $d_n$.  Create a (probability) measure $\mu_n$ by assigning mass $1/n$ to each vertex of $T_n$.  An important consequence of Theorem~\ref{thm:singletree} and the bound (\ref{eqn:GHbound}) is the following.

\begin{thm}
\begin{enumerate}
\item[(i)] Suppose that $\E{Y} = \sigma^2 < \infty$.  Then
\[
\left(T_n, \frac{\sigma}{\sqrt{n}} d_n, \mu_n \right) \convdist \left(\mathcal{T}^{(2)}, d^{(2)}, \mu^{(2)} \right).
\]
as $n \to \infty$ for the Gromov--Hausdorff--Prokhorov topology.
\item[(ii)] Suppose that $\lim_{k \to \infty} k^{\alpha+1} \Prob{Y = k} = c$ for some constant $c > 0$ and some $\alpha \in (1,2)$.  Then
\[
\left(T_n, n^{-(\alpha-1)/\alpha} d_n, \mu_n \right)  \convdist \left( \mathcal{T}^{(\alpha)}, d^{(\alpha)}, \mu^{(\alpha)} \right),
\]
as $n \to \infty$ for the Gromov--Hausdorff--Prokhorov topology.
\end{enumerate}
\end{thm}

Returning now to the setting of Theorem~\ref{thm:DLGconvergence}, the excursions of the limiting height process can heuristically be thought of as defining a forest of random $\R$-trees.  (Since there is neither a shortest nor a longest excursion, there is no sensible way to list these trees.  For definiteness, let us instead think of restricting to an interval $[0,t]$ in time, for which there \emph{is} a longest excursion, and then list the trees in decreasing order of size.)  Using the scaling properties of the underlying L\'evy processes, these consist of randomly rescaled copies of the Brownian continuum random tree in case (i) or $\alpha$-stable trees in case (ii), respectively.  We refer to these as the \emph{Brownian} and \emph{stable forests}.

\subsection{Our method}
Our approach to proving Theorem~\ref{thm:main} is as follows.  Firstly, we show that the law of the depth-first walk of the graph is (up to a small error) absolutely continuous with respect to that of a centred random walk which is in the domain of attraction of the spectrally positive $\alpha$-stable L\'evy process $L$.  This enables us to give an alternative (and perhaps more ``conceptual'') proof of Theorem~\ref{thm:JosephRiordan}.  We also show that the convergence of the depth-first walk can be boosted to a joint convergence with the corresponding height process.  The joint convergence of this pair of coding functions in the setting of a sequence of i.i.d.\ Galton--Watson trees, Theorem~\ref{thm:DLGconvergence}, is a highly non-trivial result.  The corresponding result in our setting, however, follows relatively straightforwardly from Theorem~\ref{thm:DLGconvergence} via absolute continuity and some integrability lemmas.

The height process is the key ingredient in proving a metric space convergence for these graphs, and allows us to show the convergence of a spanning forest of our graph.  In order to obtain the full metric space convergence, we must also control the edges which form cycles.  We call these \emph{back-edges}.  We prove that the number of back-edges edges in the ``large components'' is a tight quantity.  This firstly allows us to resolve Conjecture 8.5 of Joseph~\cite{Joseph}, by showing that all of the above results extend to the case where the multigraph is conditioned to be simple.  Secondly, we are able to capture the full graph structure by tracking also the locations of these back-edges in the spanning forest.  We finally show that all of these quantities can be passed through to the limit in such a way that we get convergence to the stable graph.

\subsection{Related work on scaling limits of critical random graphs, universality, and open problems}

This paper is a contribution to a now extensive literature on scaling limits of critical random graphs. In this section, we will place our work in context by giving a summary of related results.

As mentioned above, the first critical random graph to be studied from the perspective of scaling limits was the Erd\H{o}s-R\'enyi random graph, in the work of Aldous~\cite{AldousCritRG}, who considered both component sizes and surpluses.  Addario-Berry, Broutin and Goldschmidt~\cite{ABBrGo1, ABBrGo2} built on Aldous' work in order to prove convergence to the $\beta=\mu=1$ Brownian graph, in the sense of an $\ell^4$ version of the Gromov--Hausdorff distance.  (It is straightforward to improve this to a convergence in an $\ell^4$ version of the Gromov--Hausdorff--Prokhorov distance, which appears as Theorem 4.1 of Addario-Berry, Broutin, Goldschmidt \& Miermont~\cite{ABBrGoMi}.)

Several models have been proved to lie in the same universality class as the Erd\H{o}s--R\'enyi random graph, which is roughly characterised by the property that the degree of a uniformly chosen vertex converges to a limit with finite third moment.  Already in \cite{AldousCritRG}, Aldous had, in fact, also considered another model: a rank-one inhomogeneous random graph in which, for each $n \ge 1$, we are given a sequence of weights $\mathbf{w}^{(n)} = (w^{(n)}_1, w^{(n)}_2, \ldots, w^{(n)}_n)$ and each pair of vertices $\{i,j\}$ is connected independently with probability $1 - \exp(-q^{(n)} w^{(n)}_i w^{(n)}_j)$, for $1\le i,j \le n$.  Such graphs may be constructed dynamically by assigning an exponential clock to each potential edge and including the edge when the clock rings.  It is straightforward to see that, in consequence, the component sizes then evolve according to the multiplicative coalescent. In his Proposition 4, Aldous gave conditions on sequences $(\mathbf{w}^{(n)}, q^{(n)})_{n \ge 1}$ for which one gets convergence of the rescaled component \emph{weights} to the same limit as for the component sizes in the Erd\H{o}s--R\'enyi case. These results were generalised by Bhamidi, van der Hofstad and van Leeuwaarden~\cite{BhamidivdHvLScalingLimitsCriticalInhomogeneousFinite3rdMoment} (following on from work of van der Hofstad~\cite{vdHCriticalInhomogeneous}) to give convergence of the rescaled component \emph{sizes} in the Norros--Reittu model~\cite{NorrosReittu} (for which $q^{(n)}$ above is replaced by $1/\sum_{i=1}^n w^{(n)}_i$) to the sequence $(\zeta_1, \zeta_2, \ldots)$ appearing in Theorem~\ref{thm:JosephRiordan} (i), with a general $\beta$ and $\mu$.  The convergence of the component sizes was also treated in a similar setting but with i.i.d.\ vertex weights by Turova~\cite{Turova}.  

Nachmias and Peres~\cite{NachmiasPeres} proved the convergence of the rescaled component sizes for critical percolation on a random $d$-regular graph, for $d \ge 3$, to the excursion lengths of the reflected Brownian motion with parabolic drift for appropriate $\beta$ and $\mu$. As we have already detailed above, Riordan~\cite{Riordan} and Joseph~\cite{Joseph} proved analogous results for the critical configuration model with asymptotic degree distribution possessing finite third moment, with Riordan treating the surpluses as well as the component sizes.  Dhara, van der Hofstad, van Leeuwaarden and Sen~\cite{DharavdHvLSen1} improved these results to give the scaling limit of the sizes and the surpluses under a minimal set of conditions on the (deterministic) vertex degrees, which essentially amount to the convergence in distribution of the degree of a uniform vertex, along with the convergence of its third moment.  In a somewhat different direction, Bhamidi, Budhiraja and Wang~\cite{BhamidiEtAlAugmentedMultCoal} considered critical random graphs generated by Achlioptas processes \cite{BohmanFrieze} with bounded size rules.  They again proved convergence of the rescaled component sizes, along with the surpluses, as a process evolving in the critical window, building on results for the barely subcritical regime proved in \cite{BhamidiEtAlBoundedSize}.   Federico~\cite{Federico} has recently proved a scaling limit for the component sizes of the \emph{random intersection graph} which is related to that of the Erd\H{o}s--R\'enyi model.

Turning now to the metric structure, very general results concerning the domain of attraction of the Brownian graph have been proved by Bhamidi, Broutin, Sen and Wang~\cite{BhamidiBroutinSenWang}, building on earlier work for the Norros--Reittu model by Bhamidi, Sen and Wang~\cite{BhamidiSenWang}. In particular, \cite{BhamidiBroutinSenWang} gives a set of sufficient conditions under which one obtains convergence in the Gromov--Hausdorff--Prokhorov sense to the Brownian graph.  It is also demonstrated in that paper that these conditions are fulfilled for certain critical inhomogeneous random graphs (of the stochastic block model variety), and for critical percolation on a supercritical configuration model with finite third moment degree distribution.  A crucial role is played by dynamical constructions of the graphs in question, and by the idea that some pertinent statistic of the evolving graph may be well approximated by the multiplicative coalescent.  Bhamidi and Sen~\cite{BhamidiSen} later proved convergence to the Brownian graph for the critical configuration model (rather than for percolation on the supercritical case) in the Gromov--Hausdorff--Prokhorov sense, under the same set of minimal conditions as in \cite{DharavdHvLSen1}, and used it to deduce geometric properties of the \emph{vacant set} left by a random walk on various models of graph.

We have mentioned a few examples of critical percolation on graphs for which the resulting cluster sizes lie in the universality class of the Brownian graph.  This is expected to be true in much greater generality: for a wide variety of finite base graphs which are sufficiently ``high dimensional'', although the percolation critical point will be model-dependent, the behaviour in the vicinity of that critical point should essentially be the same as in the mean-field case of percolation on the complete graph, i.e.\ the Erd\H{o}s--R\'enyi model.  We refer the reader to the book of Heydenreich and van der Hofstad~\cite{HeydenreichvdH} for an in-depth discussion of this universality conjecture.

The results of the present paper primarily concern cases where the degree of a uniformly chosen vertex has infinite third moment and a power-law tail with exponent $\alpha+1 \in (2,3)$, and in this context the picture is more complicated.  As in the Brownian case, it is to be expected that, as long as the degree of a uniformly chosen vertex has the right properties, we should get the same scaling limit irrespective of precisely which model we consider.  It is technically more straightforward to consider rank-one inhomogeneous random graphs than the configuration model.  In the context of component sizes, this was first done by Aldous and Limic~\cite{AldousLimic} for the rank-one model treated by Aldous in \cite{AldousCritRG} but with appropriately altered conditions on the weight sequence.  These different conditions correspond to different extremal entrance laws for the multiplicative coalescent.  Aldous and Limic obtained the analogue of Theorem~\ref{thm:JosephRiordan} for the component \emph{weights}, where the limit is now given by the ordered lengths of the excursions above the running infimum of the \emph{thinned L\'evy process},
\begin{equation} \label{eqn:thinnedLevy}
\left(\kappa B_t + \lambda t + \sum_{i \ge 1} \vartheta_i(\I{E_i \le t} - \vartheta_i t) \right)_{t \ge 0},
\end{equation}
where $E_i \sim \text{Exp}(\vartheta_i)$ for each $i \ge 1$, $\kappa \ge 0$, $\lambda \in \R$, and $\vartheta_1 \ge \vartheta_2 \ge \ldots \ge 0$ is a sequence such that $\sum_{i \ge 1} \vartheta_i^3 < \infty$ and, if $\kappa = 0$, also $\sum_{i \ge 1} \vartheta_i^2 = \infty$.  This was extended in the $\kappa = 0$ case by Bhamidi, van der Hofstad and van Leeuwaarden~\cite{BhamidivdHvLNovelScalingLimits} to give the convergence of the component \emph{sizes} for the Norros--Reittu model with a specific weight sequence.  Heuristically, the choice of entrance law for the multiplicative coalescent is determined by the properties of the barely subcritical graph.  For the configuration model, the first work in the power-law setting was that of Joseph~\cite{Joseph} detailed above for the case of i.i.d.\ degrees.  The convergence of the sizes and surpluses for much more general (deterministic) degree sequences were treated by Dhara, van der Hofstad, van Leeuwaarden and Sen~\cite{DharavdHvLSen2}, with the possible scaling limits being driven by the same $\kappa = 0$ thinned L\'evy processes as in the Norros--Reittu model. 

A significant challenge in obtaining a metric space convergence in the power-law setting is that one often does not have direct access to a scaling limit result for the height process of the spanning forest discovered by a depth-first exploration.  (That we have such a result in the case of i.i.d.\ degrees is of considerable help to us.) The first metric space scaling limit in the power-law setting was obtained by Bhamidi, van der Hofstad and Sen~\cite{BhamidiVdHSenCoalesc} for the Norros--Reittu model with the specific weight sequence used in \cite{BhamidivdHvLNovelScalingLimits}.  Here, the convergence is in the product Gromov--Hausdorff--Prokhorov sense, and the limit object is constructed by making vertex identifications in tilted inhomogeneous continuum random trees (of the sort introduced by Aldous and Pitman in \cite{AldousPitmanICRT}). 

Broutin, Duquesne and Wang~\cite{BroutinDuquesneWang1,BroutinDuquesneWang2} use a very different approach in order to prove a unified metric space scaling limit for the Norros--Reittu model with very general weight sequences.  They are able to treat situations where the scaling limit of the depth-first walk is a thinned L\'evy process for any $\kappa \ge 0$, $\lambda \in \R$ and sequence $(\vartheta_1, \vartheta_2, \ldots)$, recovering the generality of Aldous and Limic's paper \cite{AldousLimic}.  They embed spanning subtrees of the components of the graph inside a forest of Galton--Watson trees, and exploit the convergence of this (bigger) forest on rescaling to the sequence of $\R$-trees encoded by a L\'evy process (as in Duquesne and Le Gall~\cite{DuquesneLeGall}), whose height process also converges.  This enables them to obtain the convergence of the height process of the true spanning forest in the Gromov--Hausdorff--Prokhorov sense; the surplus edges can also be tracked, in order to obtain a product Gromov--Hausdorff--Prokhorov convergence of the whole ordered sequence of graph components.

Let us finally turn to the work of Bhamidi, Dhara, van der Hofstad and Sen~\cite{BhDhvdHSeMetric}, who proved a metric space scaling limit analogous to that of \cite{BhamidiVdHSenCoalesc} for critical percolation on a supercritical configuration model. Among the settings studied so far, theirs is the closest to ours, although the technical content is rather different.  Their work relies on a connection to the multiplicative coalescent that, in particular, requires the careful analysis of susceptibility functions in the barely subcritical regime; we do not need to study the latter at all. We will describe their setting precisely, in order to provide a comparison with Theorem~\ref{thm:main}.  They take a (deterministic) degree sequence $\mathfrak{d}^n_1, \mathfrak{d}^n_2, \ldots, \mathfrak{d}^n_n$ such that $\sum_{i=1}^n \mathfrak{d}_i^n$ is even and, if $\mathfrak{D}_n$ is the degree of a typical vertex, then
\begin{enumerate}
\item[(i)] $n^{-1/(\alpha+1)} \mathfrak{d}^n_i \to \vartheta_i$ as $n \to \infty$ for each $i \ge 1$, where $\vartheta_1 \ge \vartheta_2 \ge \ldots \ge 0$ is such that $\sum_{i \ge 1} \vartheta_i^3 < \infty$ but $\sum_{i \ge 1} \vartheta_i^2 = \infty$;
\item[(ii)] $\mathfrak{D}_n \convdist \mathfrak{D}$ as $n \to \infty$, along with the convergence of its first two moments, for some random variable $\mathfrak{D}$ with $\Prob{\mathfrak{D} = 1} > 0$, $\E{\mathfrak{D}} = \mu$ and $\E{\mathfrak{D}(\mathfrak{D} -1)}/\E{\mathfrak{D}} = \theta > 1$,  and 
\[
\lim_{K \to \infty} \limsup_{n \to \infty} n^{-3/(\alpha+1)} \sum_{i \ge K+1} (\mathfrak{d}_i^n)^3 = 0.
\]
\end{enumerate}
Let $\theta_n = \E{\mathfrak{D}_n(\mathfrak{D}_n -1)}/\E{\mathfrak{D}_n}$ (which, by (ii), converges to $\theta > 1$).  They then perform percolation at parameter 
\[
p_n(\lambda) = \frac{1}{\theta_n} + \lambda n^{-(\alpha-1)/(\alpha+1)},
\]
for some $\lambda \in \R$, which yields a graph in the critical window.  In this setting, their Theorem 2.2 is the precise analogue of our Theorem~\ref{thm:main} but with the convergence in the product Gromov-weak topology and with the limit object $((\mathcal{G}_i, d_i, \mu_i), i \ge 1)$ constructed by making vertex identifications in the tilted inhomogeneous continuum random trees mentioned above.  Subsequent to the first version of our paper, this convergence has been improved to a product Gromov--Hausdorff--Prokhorov convergence under an extra technical condition in \cite{BDvdHSinprep}.  A precise description of the limit object would be too lengthy to undertake here, but it is instructive to compare the scaling limit of the depth-first walk in the two settings.  For us, this is the measure-changed stable L\'evy process $\widetilde{L}$; for Bhamidi, Dhara, van der Hofstad and Sen it is the thinned L\'evy process in (\ref{eqn:thinnedLevy}) with $\kappa = 0$.  To make the connection between the results, suppose now we take $\mathfrak{D}$ such that $\Prob{\mathfrak{D} = 1} > 0$, $\E{\mathfrak{D}} = \mu$, $\E{\mathfrak{D}(\mathfrak{D} -1)}/\E{\mathfrak{D}} = \theta > 1$ and $\Prob{\mathfrak{D} = k} \sim c k^{-\alpha-2}$.  Let $\mathfrak{d}_1^n, \ldots, \mathfrak{d}_n^n$ be an ordered sample of i.i.d.\ random variables $\mathfrak{D}_1, \ldots, \mathfrak{D}_n$ with the same distribution as $\mathfrak{D}$.  Then conditions (i) and (ii) above are satisfied almost surely for some sequence of random variables $\vartheta_1, \vartheta_2, \ldots$ (see Section 2.2 of \cite{DharavdHvLSen2}). Perform percolation at parameter $p = 1/\theta$ to obtain new degrees $D_1, D_2, \ldots, D_n$ which are mildly dependent but whose ordered version behaves very similarly to the order statistics of a i.i.d.\ sample which satisfy the conditions of our theorem. (In particular, using results of Janson~\cite{JansonPerc}, such mild dependence can be shown to have a negligible effect on the properties of the graph.) Then it should be the case that Bhamidi, Dhara, van der Hofstad and Sen's limit object is the same as the stable graph.  In particular, the process defined at (\ref{eqn:thinnedLevy}) with $\kappa = 0$ and $\lambda = 0$ should, for this particular random sequence $(\vartheta_1, \vartheta_2, \ldots)$, have the same law as $\widetilde{L}$.  Similarly, if it is the case that the scaling limit is the same as for the analogous inhomogeneous random graph setting, then our limit object should also coincide with a particular annealed version of that of Broutin, Duquesne and Wang~\cite{BroutinDuquesneWang1,BroutinDuquesneWang2}.

It is perhaps worth emphasising that, in contrast to the bulk of the other papers cited here, the multiplicative coalescent (and its relationship to percolation) appears nowhere in our proofs, and is conceptually absent from our approach.

Let us now give a list of open problems and conjectures arising from our work.

\begin{enumerate}
\item[(i)] Prove that the stable graph is, indeed, an annealed version of the limit object from \cite{BhDhvdHSeMetric} or \cite{BroutinDuquesneWang1,BroutinDuquesneWang2}.
\item[(ii)] The convergence in our main theorem occurs with respect to the product Gromov--Hausdorff--Prokhorov topology.  For sequences $\mathbf{A} = (A_1, A_2, \ldots)$ and $\mathbf{B}=(B_1, B_2, \ldots)$ of compact measured metric spaces, we may obtain stronger topologies using the distances
\begin{equation} \label{eqn:distdefn}
\mathrm{dist}_{p}(\mathbf{A},\mathbf{B}) = \left( \sum_{i \ge 1} d_{\mathrm{GHP}}(A_i, B_i)^p \right)^{1/p}
\end{equation}
for $p \ge 1$. For the Erd\H{o}s--R\'enyi random graph, the analogous convergence to the Brownian graph holds in the sense of $\mathrm{dist}_{4}$.  We conjecture that it should be possible to improve our main result for $\alpha \in (1,2]$ to a convergence in the sense of $\mathrm{dist}_{2\alpha/(\alpha-1)}$.  
\item[(iii)] One reason for wanting to prove such a result is that it would imply the convergence in distribution of the diameter of the whole graph (i.e.\ the largest distance between any two vertices in the same component).  In order to prove convergence in $\mathrm{dist}_{2\alpha/(\alpha-1)}$, we would need bounds on the component diameters in terms of powers of their sizes for the whole graph (we can do this for the parts explored up to time $O(n^{\alpha/(\alpha+1)})$ using the methods of this paper, but that is not sufficient).  A finer understanding of the \emph{barely subcritical regime} for the configuration model would presumably help to resolve this issue.
\item[(iv)] As shown in Proposition~\ref{prop:Levymeasch}, the measure change used in this paper makes sense for a large family of spectrally positive L\'evy processes (see Section~\ref{sec:measurechangeLevy} for the precise conditions).  Any such L\'evy process may be intuitively thought of as encoding a forest of continuum trees, although the analogue of Theorem~\ref{thm:DLGconvergence} holds only with the imposition of extra regularity conditions (see Theorem 2.3.1 of \cite{DuquesneLeGall}). Is it possible to find a sequence of degree distributions $(\nu^n)_{n \ge 1}$, depending now on $n$ and such that the regularity conditions are satisfied, so that if we take $D_1^{(n)}, \ldots, D_n^{(n)}$ to be i.i.d.\ random variables with distribution $\nu^n$ then we get convergence of our discrete measure change to its continuum analogue?  Or does the self-similarity inherent in the Brownian and stable settings play a key role?  If a generalisation to the L\'evy case is possible, what is the connection to thinned L\'evy processes, or to the approach of Broutin, Duquesne and Wang~\cite{BroutinDuquesneWang1,BroutinDuquesneWang2}?
\end{enumerate}

For simplicity we have restricted our attention in this paper to the case where $\theta(\nu) = 1$.  The critical window is obtained by considering the situation where the degrees $D_1^n, \ldots, D_n^n$ are i.i.d.\ but now with some $n$-dependent degree distribution $\nu^n$, such that $\E{D_1^n} \to \mu$ for some $\mu$ as $n \to \infty$,  $\theta(\nu^n) = 1 + \lambda n^{-(\alpha-1)/(\alpha+1)}$ and $\Prob{D_1^n = k} \sim c k^{-\alpha-2}$ as $k \to \infty$, for some fixed $\lambda \in \R$.  An extension of our approach to this regime has very recently been made by Donderwinkel~\cite{Donderwinkel}.

\subsection{Plan of the rest of the paper}
In Section~\ref{sec:stable}, we study the process $\widetilde{L}$ which gives rise to the stable graph.  In particular, we establish the local absolute continuity relation between $\widetilde{L}$ and $L$, and present some results in excursion theory.  The section concludes with the proof of Theorem~\ref{thm:conditionaldescription}. In Section~\ref{sec:forest}, we study a forest which is closely related to $\mathbf{M}_n(\nu)$. We show that the absolute continuity relation (\ref{eqn:Ltildedefstable}), (\ref{eqn:LtildedefBrownian}) may be seen as the limit of a discrete measure change between the degrees in the order we observe them when we explore this forest in a depth-first manner and an i.i.d.\ sequence of random variables whose law is the size-biased version of $\nu$. The main result of this section is Theorem~\ref{thm:jointdfwheight}, which gives the joint convergence of the depth-first walk and height process of the discrete forest to their continuum counterparts.  In Section~\ref{sec:graph}, we explore the multigraph $\mathbf{M}_n(\nu)$ in a depth-first manner, and record its structure via coding functions close to those of the forest in Section~\ref{sec:forest}, and show their convergence in law.  We also deal with the occurrence of the back-edges, and prove that $\mathbf{M}_n(\nu)$ and $\mathbf{G}_n(\nu)$ cannot have different scaling limits.  We must then extract the individual components of the graph in decreasing order of size, and prove that their individual coding functions converge. We adapt an approach of Aldous~\cite{AldousCritRG} using size-biased point processes; this is perhaps the most technical part of the paper.  Section~\ref{sec:graph} culminates in the proof of Theorem~\ref{thm:main}. The \hyperref[appn]{Appendix} contains various technical results.  In particular, in Section~\ref{sec:measurechangeLevy} we give a formulation of the measure change in (\ref{eqn:Ltildedefstable}) and (\ref{eqn:LtildedefBrownian}) for a general class of L\'evy processes, which may be of independent interest. In Section~\ref{sec:singlecompo}, we show the natural result that a single component of $\mathbf{G}_n(\nu)$ or $\mathbf{M}_n(\nu)$ \emph{conditioned to have size $\fl{x n^{\alpha/(\alpha+1)}}$} has a component of the $\alpha$-stable graph of size $x$ as its scaling limit.

\section{The limit object: the stable graph} \label{sec:stable}

\subsection{An absolute continuity relation for spectrally positive $\alpha$-stable L\'evy processes}
We begin by discussing the coding function $R$ discovered by Joseph~\cite{Joseph}, which was defined in (\ref{eqn:Rstable}). 
Fix $\alpha \in (1,2)$, $\mu \in (1,2)$ and $c > 0$. Recall that $L$ is the spectrally positive $\alpha$-stable L\'evy process having L\'evy measure $\pi(dx) = \frac{c}{\mu} x^{-(\alpha+1)}dx$.  This process has Laplace transform 
\[
\E{\exp(-\lambda L_t)} = \exp(t \Psi(\lambda)), \qquad \lambda \ge 0, \quad t \ge 0,
\]
where
\[
\Psi(\lambda) = \int_0^{\infty} \frac{c}{\mu} x^{-(\alpha+1)} dx (e^{-\lambda x} - 1 + \lambda x) = \frac{C_{\alpha} }{\mu} \lambda^{\alpha},
\]
with
\[
C_{\alpha} = \frac{c \Gamma(2-\alpha)}{\alpha(\alpha-1)}.
\]
Recall also that $X$ is the unique process with independent increments such that
\[
\E{\exp(-\lambda X_t)}  = \exp \left( \int_0^t ds \int_0^{\infty} dx (e^{-\lambda x}-1+ \lambda x) \frac{c}{\mu}\frac{1}{x^{\alpha+1}} e^{-xs/\mu} \right), \quad \lambda \ge 0, \quad t \ge 0.
\]
Let
\[
A_t = - C_{\alpha} \frac{t^{\alpha}}{\mu^{\alpha}}
\]
and define
\[
\widetilde{L}_t = X_t + A_t.
\]
We observe that $X$ is a martingale and $A$ is a finite-variation process, so this is, in fact, the Doob--Meyer decomposition of the process $\widetilde{L}$.

\begin{prop} \label{prop:transient}
We have
\[
\widetilde{L}_t \to -\infty \quad \text{ a.s.}
\]
as $t \to \infty$.
\end{prop}

\begin{proof}
Lemma B.3 of Joseph~\cite{Joseph} gives the convergence in probability; we adapt his argument.  Since $A_t$ is deterministic and tends to $-\infty$, it will be sufficient to prove that
\[
\limsup_{t \to \infty} t^{-\alpha} X_t = 0 \quad \text{a.s.}
\]
Consider a Poisson point process on $\R_+ \times \R_+$ of intensity 
\[
\frac{c}{\mu} x^{-(\alpha+1)} e^{-xs/\mu} \ ds \ dx,
\]
with points $\{(s, \Delta_s)\}$.  Let $X^{(1)}$ be the martingale arising from compensating the jumps of magnitude at most 1, formally defined as the $\epsilon \to 0$ limit of the family $\{X^{(1,\epsilon)}, \epsilon > 0\}$ of processes given by
\[
X^{(1,\epsilon)}_t = \sum_{s \le t} \Delta_s \I{\epsilon < \Delta_s <1)} - \int_0^t ds \int_{\epsilon}^{1} dx \frac{c}{\mu} x^{-\alpha} e^{-xs/\mu},
\]
and let
\[
X^{(2)}_t = \sum_{s \le t} \Delta_s \I{\Delta_s \ge 1} - \int_0^t ds \int_1^{\infty} dx \frac{c}{\mu} x^{-\alpha} e^{-xs/\mu}.
\]
Then $X_t = X^{(1)}_t + X^{(2)}_t$.  By Doob's $L^2$-inequality, we have
\[
\E{\left(\sup_{0 \le s \le t} \left|X^{(1)}_s \right|\right)^2} \le 4 \E{(X_t^{(1)})^2} =  4 \int_0^t ds \int_0^1 dx \frac{c}{\mu} x^{-(\alpha-1)} e^{-xs/\mu}.
\]
For every $s>0$, 
\[
\int_0^1 dx \frac{c}{\mu} x^{-(\alpha-1)} e^{-xs/\mu}= s^{\alpha -2}\int_{0}^{s/\mu}\frac{c}{\mu^{\alpha-1}}u^{-(\alpha-1)}e^{-u}du\leq \frac{c\Gamma(2-\alpha)}{\mu^{\alpha-1}} s^{\alpha -2},
\]  
and so
\[
\E{\left(\sup_{0 \le s \le t} \left|X^{(1)}_s \right|\right)^2} \le C t^{\alpha-1}
\]
for some constant $C > 0$. Hence, applying Markov's inequality, we get
\[
\Prob{ \sup_{n-1 < s \le n} \left| X^{(1)}_s \right| > n^{(\alpha+1)/2}} \le \frac{C}{n^2}.
\]
As this is summable in $n$, the Borel--Cantelli lemma gives that
\[
\Prob{ \sup_{n-1 < s \le n} \left| X^{(1)}_s \right| > n^{(\alpha+1)/2} \text{ i.o.}} = 0.
\]
Since $\alpha > 1$, it follows that
\[
\limsup_{t \to \infty} t^{-\alpha} X^{(1)}_t = 0 \quad \text{a.s.}
\]
Turning now to $X^{(2)}$, for all $t \ge 0$ we have the straightforward bound
\[
\sup_{t\geq 0}X_s^{(2)} \le \sum_{s \ge 0} \Delta_s \I{\Delta_s \ge 1}.
\]
The right-hand side has expectation
\[
\int_0^{\infty} ds \int_1^{\infty} dx \frac{c}{\mu} x^{-\alpha} e^{-xs/\mu} = c \int_1^{\infty} dx \ x^{-(\alpha+1)} \int_0^{\infty} ds \frac{x}{\mu} e^{-xs/\mu} = c \int_1^{\infty} x^{-(\alpha+1)} dx = \frac{c}{\alpha}.
\]
Since the jumps are all non-negative, this computation also entails that
\[
\inf_{t\geq 0} X^{(2)}_t \geq -\frac{c}{\alpha}.
\]
Hence by Markov's inequality, we have for all $n\geq 1$:
\[
\Prob{ \sup_{n-1 < s \le n} \left| X^{(2)}_s \right| >  n^{(\alpha+1)/2}}\le \frac{c}{\alpha  n^{(\alpha+1)/2}}.
\]
As for $X^{(1)}$, the Borel--Cantelli lemma 
gives that
\[
\limsup_{t \to \infty} t^{-\alpha} X^{(2)}_t = 0 \quad \text{a.s.}
\]
The result follows.
\end{proof}

The main purpose of this section is to expand considerably our understanding of the processes $\widetilde{L}$ and $R$.    Our first new result says that the law of the process $\widetilde{L}$ is absolutely continuous with respect to the law of the L\'evy process $L$ on compact time-intervals.

\begin{prop} \label{prop:stablemeasch}
For every $t \ge 0$, we have the following absolute continuity relation: for every non-negative integrable functional $f: \mathbb{D}([0,t], \R) \to \R_+$,
\begin{align*}
\E{f(\widetilde{L}_s, 0 \le s \le t)} & = \E{\exp\left( - \frac{1}{\mu} \int_0^{t} s dL_s - C_{\alpha}\frac{t^{\alpha+1}}{(\alpha+1) \mu^{\alpha+1}} \right) f(L_s, 0 \le s \le t)}.
\end{align*}
\end{prop}

This proposition is a consequence of a more general change of measure for spectrally positive L\'evy processes, Proposition~\ref{prop:Levymeasch}, which is proved in the \hyperref[appn]{Appendix} below.

In the Brownian case, we instead have $L_t = \sqrt{\frac{\beta}{\mu}} B_t$, where $B$ is a standard Brownian motion,
\begin{equation} \label{eqn:paradrift}
\widetilde{L}_t = \sqrt{\frac{\beta}{\mu}} B_t - \frac{\beta}{2 \mu^2} t^2,
\end{equation}
and Proposition~\ref{prop:Levymeasch} gives
\[
\E{ f(\widetilde{L}_s, 0 \le s \le t)} = \E{\exp \left( - \frac{1}{\mu} \int_0^t s dL_s -  \frac{\beta}{6 \mu^3} t^3 \right) f \left(L_s, 0 \le s \le t\right)}.
\]
In order to harmonise notation, let us define $C_2 := \beta/2$, so that Proposition~\ref{prop:stablemeasch} is valid as stated for all $\alpha \in (1,2]$.

\begin{rem}
The absolute continuity cannot be extended to $t = \infty$: the process $(L_t, t \ge 0)$ is recurrent whereas, by Proposition~\ref{prop:transient} for $\alpha \in (1,2)$ or (\ref{eqn:paradrift}) for $\alpha = 2$, we have $\widetilde{L}_t \to -\infty$ a.s.\ as $t \to \infty$.  (In particular, $\left(\exp\left( - \frac{1}{\mu} \int_0^{t} s dL_s - C_{\alpha}\frac{t^{\alpha+1}}{(\alpha+1) \mu^{\alpha+1}} \right), t \ge 0 \right)$ is a martingale which is not uniformly integrable.)
\end{rem}

Recall that $H$ is the height process which corresponds to $L$.  Then for any $\alpha \in (1,2]$, we may define a pair $(\widetilde{L}, \widetilde{H})$ of processes via change of measure as follows: for suitable test-functions $f: \mathbb{D}([0,t], \R)^2 \to \R$,
\[
\E{f(\widetilde{L}_u, \widetilde{H}_u, 0 \le u \le t)} = \E{\exp\left(-\frac{1}{\mu} \int_0^t sdL_s - \frac{C_{\alpha} t^{\alpha+1}}{(\alpha+1) \mu^{\alpha+1}} \right) f(L_u, H_u, 0 \le u \le t)}.
\]

\subsection{Excursion theory}
We begin with some notation.  Write $\mathbb{D}_+(\R_+, \R_+)$ for the space of c\`adl\`ag functions $f: \R_+ \to \R_+$ with only positive jumps. We write $\mathcal{E}$ for the space of excursions, that is 
\[
\mathcal{E} = \{ \varepsilon \in \mathbb{D}_+(\R_+, \R_+): \exists t>0 \text{ s.t.\ $\varepsilon(s) > 0$ for $s \in (0,t)$ and $\varepsilon(s) = 0$ for $s \ge t$}\}. 
\]
For $\varepsilon \in \mathcal{E}$, let $\zeta(\varepsilon)$ be the lifetime of $\varepsilon$, that is the smallest $t$ such that $\varepsilon(s) = 0$ for $s \ge t$.  Let $\mathcal{E}^* = \mathcal{E}\cup\{\partial\}$, where the extra state, $\partial$, represents the empty excursion, with $\zeta(\partial) = 0$.  

Let $I_t = \inf_{0 \le s \le t}L_s$.  It is standard that the process $-I$ acts as a local time at 0 for the reflected L\'evy process $(L_t - \inf_{0 \le s \le t} L_s, t \ge 0)$ (see Chapter VII of Bertoin~\cite{BertoinBook} or Section 1.1.2 of Duquesne and Le Gall~\cite{DuquesneLeGall}).  Indeed, we may decompose the path of the reflected process into excursions above 0.    Now write 
\[
\sigma_{\ell} = \inf\left\{t \ge 0: I_t < -\ell\right\},
\]
so that $(\sigma_{\ell}, \ell \ge 0)$ is the inverse local time.  We observe that $\sigma_{\ell}$ is a stopping time for the (usual augmentation of the) natural filtration of $(L_t)_{t \ge 0}$.  For $\ell \ge 0$, write
\[
\varepsilon^{(\ell)} = \begin{cases} (\ell + L_{\sigma_{\ell-} + u}, 0 \le u \le \sigma_{\ell} - \sigma_{\ell-}) & \text{ if } \sigma_{\ell} - \sigma_{\ell-} > 0 \\
\partial & \text{ otherwise.}
\end{cases}
\]
Then the following theorem is standard (see Theorem VII.1.1 of Bertoin~\cite{BertoinBook}, converting from the spectrally negative case, or Miermont~\cite{Mier03} for a convenient reference).

\begin{thm} \label{thm:localtimeexc}
\begin{enumerate}
\item[(a)] The inverse local time process $(\sigma_{\ell}, \ell \ge 0)$ is a stable subordinator of index $1/\alpha$ and, more specifically, with L\'evy measure
\[
\frac{\mu^{1/\alpha}}{C_{\alpha}^{1/\alpha} \alpha \Gamma(1-1/\alpha)} x^{-1 - 1/\alpha} dx.
\]
\item[(b)] There exists a $\sigma$-finite measure $\N$ on $\mathcal{E}$ such that the point measure on $\R_+ \times \mathcal{E}$ given by
\begin{equation} \label{eqn:pointmeas}
\sum_{s \ge 0: \sigma_s - \sigma_{s-} > 0} \delta_{(s,\varepsilon^{(s)})}
\end{equation}
is a Poisson random measure of intensity $d\ell \otimes \mathbb{N}(d \mathbbm{e})$. Moreover, the excursion measure $\N$ satisfies
\[
\N(\zeta(\mathbbm{e}) \in dx) = \frac{\mu^{1/\alpha}}{C_{\alpha}^{1/\alpha} \alpha \Gamma(1-1/\alpha)} x^{-1 - 1/\alpha} dx.
\]
\end{enumerate}
\end{thm}

Consider the excursions occurring before time $\sigma_\ell$.  With probability 1, only finitely many of these are longer than $\eta$ in duration for any $\eta > 0$.  So, in particular, they may be listed in decreasing order of length as $(\varepsilon_i^{(\ell)}, i \ge 1)$.

Since $L$ is self-similar, it is possible to make sense of normalised versions of $\N$ i.e.\ $\N^{(x)}(\cdot) = \N(\cdot | \zeta(\mathbbm{e}) = x)$, which are probability measures.  (Again see Miermont~\cite{Mier03} for more details.)  For example, the law of $\mathbbm{e}$ under $\N^{(x)}$ is the same as the law of
\[
\left( (x/\zeta(\mathbbm{e}))^{1/\alpha} \mathbbm{e}(\zeta(\mathbbm{e})s/x), 0 \le s \le x\right)
\]
under $\mathbb{N}(\cdot | \zeta(\mathbbm{e}) > \eta)$ for any fixed $\eta > 0$.  In particular, we have that under $\N^{(x)}$, the rescaled excursion $(x^{-1/\alpha} \mathbbm{e}(xu), 0 \le u \le 1)$ has the same law as $\mathbbm{e}$ under $\N^{(1)}$.  It follows that the excursions $\varepsilon^{(s)}$ appearing in (\ref{eqn:pointmeas}) may be thought of in two parts: as their lengths $\zeta(\varepsilon^{(s)}) = \sigma_s - \sigma_{s-}$ and their normalised ``shapes'' $e^{(s)}:= \left(\zeta(\varepsilon^{(s)})^{-1/\alpha} \varepsilon^{(s)}(\zeta(\varepsilon^{(s)}) u), 0 \le u \le 1\right)$ where, crucially, the collection of shapes $(e^{(s)}, s \ge 0)$ is \emph{independent} of the collection of excursion lengths $(\zeta(\varepsilon^{(s)}), s \ge 0)$.  We will write $\mathbb{E}_{\mathbb{N}^{(x)}}$ for the expectation with respect to $\mathbb{N}^{(x)}$.

We observe that the excursions of the L\'evy process $L$ above its running infimum and the excursions of the height process $H$ are in one-to-one correspondence and have the same lengths.  In particular, we can make sense of an excursion of the height process $\mathbbm{h}$ derived from $\mathbbm{e}$, under $\N$ or its conditioned versions.  The scaling relation for the height process is that under $\N^{(x)}$ the rescaled excursion $(x^{-(\alpha-1)/\alpha} \mathbbm{h}(xu), 0 \le u \le 1)$ has the same law as $\mathbbm{h}$ under $\N^{(1)}$.  The usual \emph{stable tree} is encoded by (a scalar multiple of) a height process with the distribution of $\mathbbm{h}$ under $\N^{(1)}$.

Much of this structure can be transferred into our setting, by absolute continuity. Recall that
\[
R_t = \widetilde{L}_t - \inf_{0 \le s \le t} \widetilde{L}_s, \quad t \ge 0.
\]
We will make use of the following properties.

\begin{lem} \label{lem:Josephprops} The following statements hold almost surely.
\begin{enumerate}
\item[(i)] For each $\epsilon > 0$, $R$ has only finitely many excursions of length greater than or equal to $\epsilon$.
\item[(ii)] The set $\{t: R_t = 0\}$ has Lebesgue measure 0.
\item[(iii)] If $(l_1,r_1)$ and $(l_2,r_2)$ are excursion-intervals of $R$ and $l_1<l_2$, then $\widetilde{L}_{l_1}> \widetilde{L}_{l_2}$.
\item[(iv)] For $a \ge 0$, let $\mathcal{B}_a = \{b > a: \widetilde{L}_{b-} = \inf_{a \le s \le b} \widetilde{L}_s\}$.  Then $\mathcal{B}_a$ does not intersect the set of jump times of $\widetilde{L}$. 
\end{enumerate}
\end{lem}

\begin{proof} 

Part (i) is a consequence of Lemma B.3 of Joseph~\cite{Joseph}.  For parts (ii), (iii) and (iv), we first argue that the claimed properties are almost surely true for the L\'evy process $L$ and then use absolute continuity to deduce them for $\widetilde{L}$. 

The analogues of both (ii) and (iii) are standard for $L$ (see, for example, Chapter VII of Bertoin~\cite{BertoinBook}; indeed, these properties are necessary for Theorem~\ref{thm:localtimeexc} to work).  It follows by absolute continuity that $\Prob{\text{Leb}(\{s \le t: R_s = 0\}) = 0}=1$ and 
\[
\Prob{\text{$\widetilde{L}_{l_1}> \widetilde{L}_{l_2}$ for all $(l_1,r_1), (l_2,r_2)$ excursion-intervals of $R$ with $l_1<l_2 \le t$}} = 1,
\]
for fixed $t>0$.  But then (ii) and (iii) follow by monotone convergence.

 By the stationarity and independence of increments of $L$, it is sufficient to prove (iv) for $a=0$.  But this then follows from Corollary 1 of Rogers~\cite{Rogers}.  In particular, if we let $\mathcal{J}$ be the set of jump-times of $\tilde{L}$, by absolute continuity we get $\Prob{\mathcal{B}_a \cap [0,t] \cap \mathcal{J} \neq \emptyset} = 0$ for any $t > 0$. By monotone convergence again, we obtain $\Prob{\mathcal{B}_a \cap \mathcal{J} \neq \emptyset} = 0$.
\end{proof}

Let $\widetilde{I}_t = \inf_{0 \le s \le t}\widetilde{L}_s$.  As for the reflected stable process, we have that $-\widetilde{I}$ acts as a local time at 0 for $R$.  We write $(\widetilde{\sigma}_{\ell}, \ell \ge 0)$ with $\widetilde{\sigma}_{\ell} = \inf \{t > 0: \widetilde{I}_t < -\ell\}$ for  the inverse local time, $(\widetilde{\varepsilon}^{(\ell)}, \ell \ge 0)$ for the collection of excursions above 0, indexed by local time (with $\widetilde{\varepsilon}^{(\ell)} = \partial$ if $\widetilde{\sigma}_{\ell} - \widetilde{\sigma}_{\ell-} = 0$), and $(\widetilde{e}^{(\ell)}, \ell \ge 0)$ for their shapes.  In order to understand the laws of these quantities, we first need to prove two preliminary results, Lemma~\ref{lem:expbound} and Proposition~\ref{prop:stoppedmart}. 

\begin{lem} \label{lem:expbound}
Let $\alpha \in (1,2]$.  Then for any $\theta > 0$,
\[
\mathbb{E}_{\N^{(1)}}\left[\exp \left( \theta \int_0^1 \mathbbm{e}(t) dt \right) \right]  < \infty.
\]
\end{lem}

\begin{proof}
This is well known in the $\alpha = 2$ case; see, for example, Section 13 of Janson~\cite{Janson}.
For $\alpha \in (1,2)$, observe that 
\[
\int_0^1 \mathbbm{e}(t) dt \le \sup_{t \in [0,1]} \mathbbm{e}(t).
\]

By Theorem 9 of Kortchemski~\cite{KortchemskiSubExpTailBounds} (see also the discussion at the top of the 12th page), for any $\delta \in (0,\frac{\alpha}{\alpha-1})$, there exist constants $C_1, C_2 > 0$ such that
\[
\N^{(1)} \left(\sup_{t \in [0,1]} \mathbbm{e}(t) \ge u \right) \le C_1 \exp(-C_2 u^{\delta}),
\]
for every $u \ge 0$. (Note that since Kortchemski works with the L\'evy process having Laplace exponent $\lambda^{\alpha}$, his normalised excursions are a constant scaling factor different from ours.  But this changes the bound only by a constant.) Since we may take $\delta > 1$, the result follows.
\end{proof}

For $t \ge 0$ write 
\[
\Phi(t) := \exp \left( - \frac{1}{\mu} \int_0^t sdL_s - \frac{C_{\alpha} t^{\alpha+1}}{(\alpha+1) \mu^{\alpha + 1}} \right).
\]
In Lemma~\ref{lem:laplint} in the \hyperref[appn]{Appendix} we show that $\Phi$ is a particular instance of a family of exponential martingales for spectrally positive L\'evy processes.

\begin{prop} \label{prop:stoppedmart}
Let $\alpha \in (1,2]$.  For any $\ell \ge 0$ we have that  $(\Phi(t \wedge \sigma_{\ell}), t \ge 0)$ is a uniformly integrable martingale and, thus, $\E{\Phi(\sigma_{\ell})} = 1$.
\end{prop}

\begin{proof}
By Lemma~\ref{lem:laplint} with $\theta =\frac{1}{\mu}$, $\gamma=\delta=0$ and $\pi(dx)=\frac{c}{\mu}x^{-(\alpha+1)}dx$, the process $(\Phi(t), t \ge 0)$ is a non-negative martingale of mean 1. Since for any $\ell\ge 0$, $\sigma_{\ell}$ is a stopping time for $L$ and since $\Phi$ has right-continuous trajectories, $(\Phi(t \wedge \sigma_{\ell}), t \ge 0)$ is a martingale with respect to the natural filtration of $L$. So by the almost sure martingale convergence theorem, we must have $\Phi(t \wedge \sigma_{\ell}) \to \Phi(\sigma_{\ell})$ almost surely as $t \to \infty$.  Then $(\Phi(t \wedge \sigma_{\ell}), t \ge 0)$ is uniformly integrable if and only if this convergence also holds in $L^1$.  By Fatou's lemma, we get
$\E{\Phi(\sigma_{\ell})} \le 1$, so that $\Phi(\sigma_{\ell})$ is integrable.  Now, for any $t > 0$,
\[
\E{|\Phi(\sigma_{\ell}) - \Phi(t \wedge \sigma_{\ell})|} = \E{|\Phi(\sigma_{\ell} ) - \Phi(t)| \I{\sigma_{\ell} > t}} \le \E{\Phi(\sigma_{\ell}) \I{\sigma_{\ell} > t}} + \E{\Phi(t) \I{\sigma_{\ell} > t}}.
\]
Observe that by the definition of the measure-changed process, we have $\E{\Phi(t) \I{\sigma_{\ell} > t}} = \Prob{\widetilde{\sigma}_{\ell} > t}$.  So
\[
\E{|\Phi(\sigma_{\ell}) - \Phi(t \wedge \sigma_{\ell})|} \le \Prob{\widetilde{\sigma}_{\ell} > t} +  \E{\Phi(\sigma_{\ell}) \I{\sigma_{\ell} > t}}.
\]
Since $\Phi(\sigma_{\ell})$ is integrable (and hence uniformly integrable) and since $\sigma_{\ell} < \infty$ almost surely, we have $\lim_{t \to \infty} \E{\Phi(\sigma_{\ell}) \I{\sigma_{\ell} > t}} = 0$. By Proposition~\ref{prop:transient}, we have that $\widetilde{L}_t \to -\infty$ almost surely as $t \to \infty$, and so $\widetilde{\sigma}_{\ell} < \infty$ almost surely.  So 
$ \lim_{t \to \infty} \Prob{\widetilde{\sigma}_{\ell} > t} = 0$ and we get
\[
\E{|\Phi(\sigma_{\ell}) - \Phi(t \wedge \sigma_{\ell})|}  \to 0
\]
as $t \to \infty$. Hence, $(\Phi(t \wedge \sigma_{\ell}), t \ge 0)$ is uniformly integrable and, in particular, we may deduce that $\E{\Phi(\sigma_{\ell})} = 1$.
\end{proof}

We are now in a position to characterise the joint law of $(\widetilde{\sigma}_s, 0 \le s \le \ell)$ and $(\widetilde{\varepsilon}^{(s)}, s \le \ell)$.  We will find it convenient to list the excursions occurring before local time $\ell$ has been accumulated in decreasing order of length, as $(\widetilde{\varepsilon}^{(\ell)}_i, i \ge 1)$.  Proposition~\ref{prop:stoppedmart} implies that we may use the Radon--Nikodym derivative $\Phi(t)$ to change measure at the \emph{random} times $\sigma_{\ell}$.  As earlier, we write $(\varepsilon^{(\ell)}_i, i \ge 1)$ for the excursions of $L$ occurring before time $\sigma_{\ell}$ in decreasing order of length. For an excursion $\varepsilon \in \mathcal{E}^*= \mathcal{E}\cup\{\partial\}$, write $a(\varepsilon) = \int_0^{\zeta(\varepsilon)} \varepsilon(u) du$ for its area.

\begin{prop} \label{prop:exclaw}
For suitable test functions $f$ and $g_1, g_2, g_3, \ldots$, and any $n \ge 1$, we have
\begin{align*}
& \E{f(\widetilde{\sigma}_s, 0 \le s \le \ell) \prod_{i=1}^n g_i \left(\widetilde{\varepsilon}_i^{(\ell)} \right)} \\
& = \mathbb{E} \left[ \exp \left(  \frac{1}{\mu} \int_0^{\ell} \sigma_r dr  -  \frac{C_{\alpha} \sigma_{\ell}^{\alpha+1}}{(\alpha+1) \mu^{\alpha+1}} \right) f(\sigma_s, 0 \le s \le \ell) \right. \\
& \quad   \left. \times\E{\exp \left( \frac{1}{\mu} \sum_{j > n} a(\varepsilon^{(\ell)}_j) \right) \Bigg| \zeta(\varepsilon_k^{(\ell)}), k > n}   \prod_{i=1}^n \E{\exp \left( \frac{1}{\mu} a(\varepsilon^{(\ell)}_i) \right) g_i \left(\varepsilon_i^{(\ell)}\right) \Bigg| \zeta(\varepsilon_i^{(\ell)})} \right] .
\end{align*}
In particular, the excursions $(\widetilde{\varepsilon}_i^{(\ell)}, i \ge 1)$ are conditionally independent given their lengths.  Moreover, for any $i \ge 1$ and any suitable test function $g$,
\begin{align*}
\E{g \left(\widetilde{\varepsilon}_i^{(\ell)} \right) \ \Big| \ \zeta(\widetilde{\varepsilon}_i^{(\ell)}) = x}  & = \frac{\mathbb{E}_{\N^{(x)}}  \left[ \exp \left( \frac{1}{\mu} \int_0^x \mathbbm{e}(t) dt \right) g(\mathbbm{e}) \right]}{\mathbb{E}_{\N^{(x)}}  \left[ \exp \left( \frac{1}{\mu} \int_0^x \mathbbm{e}(t) dt \right) \right]} \\
& = \frac{\mathbb{E}_{\N^{(1)}}  \left[ \exp \left( \frac{x^{1+1/\alpha}}{\mu} \int_0^1 \mathbbm{e}(t) dt \right) g(x^{1/\alpha}\mathbbm{e}(\cdot/x)) \right]}{\mathbb{E}_{\N^{(1)}} \left[ \exp \left( \frac{x^{1+1/\alpha}}{\mu} \int_0^1 \mathbbm{e}(t) dt \right) \right]}.
\end{align*}
\end{prop}

\begin{proof}
By integration by parts and writing $L_s = I_s + (L_s - I_s)$, noting that $L_{\sigma_{\ell}} = -\ell$, we get
\[
-\frac{1}{\mu} \int_0^{\sigma_{\ell}} s dL_s  
 = \frac{\ell \sigma_{\ell}}{\mu} + \frac{1}{\mu} \int_0^{\sigma_{\ell}} L_s ds 
 =  \frac{\ell \sigma_{\ell} }{\mu} + \frac{1}{\mu} \int_0^{\sigma_{\ell}} I_s ds + \frac{1}{\mu} \int_0^{\sigma_{\ell}} (L_s - I_s) ds. 
\]
Changing variable in the middle term, and using the fact that $I_{\sigma_s} = -s$, we obtain
\[
 \frac{\ell \sigma_{\ell} }{\mu} + \frac{1}{\mu} \int_0^{\ell} I_{\sigma_s} d\sigma_s + \frac{1}{\mu} \int_0^{\sigma_{\ell}} (L_s - I_s) ds = \frac{\ell \sigma_{\ell}}{\mu} - \frac{1}{\mu} \int_0^{\ell} s d\sigma_s +  \frac{1}{\mu} \int_0^{\sigma_{\ell}} (L_s - I_s) ds.
\]
Another integration by parts yields that this is equal to
\[
\frac{1}{\mu} \int_0^{\ell} \sigma_s ds +  \frac{1}{\mu} \int_0^{\sigma_{\ell}} (L_s - I_s) ds.
\]
Finally, we can integrate the excursions of $L-I$ separately to obtain that this is equal to
\[
\frac{1}{\mu} \int_0^{\ell} \sigma_s ds +  \frac{1}{\mu} \sum_{s \le \ell} a(\varepsilon^{(s)}).
\]
Hence,
\begin{equation} \label{eqn:stoppedmart}
\Phi(\sigma_{\ell}) = \exp \left( \frac{1}{\mu} \int_0^{\ell} \sigma_r dr +  \frac{1}{\mu} \sum_{s \le \ell}  a(\varepsilon^{(s)}) - C_{\alpha} \frac{ \sigma_{\ell}^{\alpha+1}}{(\alpha+1) \mu^{\alpha+1}} \right).
\end{equation}
Now,
\begin{align*}
& \E{f(\widetilde{\sigma}_s, 0 \le s \le \ell) \prod_{i=1}^n g_i \left(\widetilde{\varepsilon}_i^{(\ell)} \right)} \\
& = \E{\exp \left( \frac{1}{\mu} \int_0^{\ell} \sigma_r dr +  \frac{1}{\mu} \sum_{s \le \ell}  a(\varepsilon^{(s)}) - C_{\alpha} \frac{ \sigma_{\ell}^{\alpha+1}}{(\alpha+1) \mu^{\alpha+1}} \right)  f(\sigma_s, 0 \le s \le \ell) \prod_{i=1}^n g_i \left(\varepsilon_i^{(\ell)} \right)} \\
& = \mathbb{E} \left[ \exp \left( \frac{1}{\mu} \int_0^{\ell} \sigma_r dr - C_{\alpha} \frac{ \sigma_{\ell}^{\alpha+1}}{(\alpha+1) \mu^{\alpha+1}} \right) f(\sigma_s, 0 \le s \le \ell)  \right. \\
& \qquad \qquad \left. \times \E{\exp \left( \frac{1}{\mu} \sum_{s \le \ell}  a(\varepsilon^{(s)})\right) \prod_{i=1}^n f_i \left(\varepsilon_i^{(\ell)} \right) \Big| (\sigma_s, 0 \le s \le \ell)} \right],
\end{align*}
As discussed below Theorem~\ref{thm:localtimeexc}, the excursions of the stable L\'evy process are conditionally independent given their lengths, which yields the first expression in the statement of the proposition.  The final statement is an immediate consequence of the scaling property for stable excursions; we observe that this change of measure for the excursions is well-defined by Lemma~\ref{lem:expbound}.
\end{proof}

Lemma~\ref{lem:Josephprops} (i) implies that we can list \emph{all} the excursions of $R$ in decreasing order of length: write $(\widetilde{\varepsilon}_i, i \ge 1)$ for this list.  Write $(\widetilde{h}_i, i \ge 1)$ for the corresponding height process excursions.

\begin{prop} \label{prop:condindep}
The pairs of excursions $(\widetilde{\varepsilon}_i, \widetilde{h}_i, i \ge 1)$ are conditionally independent given their lengths $(\zeta(\widetilde{\varepsilon}_i), i \ge 1)$, with law specified by
\begin{align*}
\E{g \left(\widetilde{\varepsilon}_i, \widetilde{h}_i\right) \ \Big| \ \zeta(\widetilde{\varepsilon}_i) = x}  & = \frac{\mathbb{E}_{\N^{(x)}} \left[ \exp \left( \frac{1}{\mu} \int_0^x \mathbbm{e}(t) dt \right) g(\mathbbm{e}, \mathbbm{h}) \right]}{\mathbb{E}_{\N^{(x)}}  \left[ \exp \left( \frac{1}{\mu} \int_0^x \mathbbm{e}(t) dt \right) \right]} \\
& = \frac{\mathbb{E}_{\N^{(1)}}  \left[ \exp \left( \frac{x^{1+1/\alpha}}{\mu} \int_0^1 \mathbbm{e}(t) dt \right) g(x^{1/\alpha}\mathbbm{e}(\cdot/x), x^{(\alpha-1)/\alpha} \mathbbm{h}(\cdot/x)) \right]}{\mathbb{E}_{\N^{(1)}}  \left[ \exp \left( \frac{x^{1+1/\alpha}}{\mu} \int_0^1 \mathbbm{e}(t) dt \right) \right]}.
\end{align*}
\end{prop}

\begin{proof}
The excursions of $R$ occurring before local time $\ell$ has been accumulated are a strict subset of all the excursions that ever occur.  By Lemma~\ref{lem:Josephprops}, we have that
\[
\sup \{\zeta(\varepsilon): \text{ $\varepsilon$ is an excursion of $R$ starting after time $t$}\} \convprob 0
\]
as $t \to \infty$ and $-\widetilde{I}_t \to \infty$ as $t \to \infty$. The latter implies that $\widetilde{\sigma}_{\ell} < \infty$ a.s., and since $-\widetilde{I}_t < \infty$ for each $t > 0$, we also have $\widetilde{\sigma}_{\ell} \to \infty$ as $\ell \to \infty$.  Hence,
\[
\sup \{\zeta(\varepsilon): \text{ $\varepsilon$ is an excursion of $R$ starting after time $\widetilde{\sigma}_{\ell}$}\} \convprob 0
\]
as $\ell \to \infty$. It follows that
\[
(\zeta(\widetilde{\varepsilon}_i^{(\ell)}), i \ge 1) \to (\zeta(\widetilde{\varepsilon}_i), i \ge 1) \quad \text{ a.s.}
\] 
in the product topology, as $\ell \to \infty$. The result then follows from Proposition~\ref{prop:exclaw} since the expressions there do not depend on the value of $\ell$.
\end{proof}

This enables us to give the proof of Theorem~\ref{thm:conditionaldescription} assuming the definition of $((\mathcal{G}_i, d_i, \mu_i), i \ge 1)$ from $\widetilde{L}$ given following Theorem~\ref{thm:main}.

\begin{proof}[Proof of Theorem~\ref{thm:conditionaldescription}]
With Proposition~\ref{prop:condindep} in hand, it remains to deal with the Poisson points which give rise to the vertex-identifications.  We have straightforwardly that, given $\widetilde{\varepsilon}_i$, the number $M_i$ of points falling under the excursion is conditionally independent of the other excursions and has a Poisson distribution with parameter $\frac{1}{\mu} \int_0^{\infty} \widetilde{\varepsilon}_i(u) du$.  Moreover, conditionally on the number of points, their locations are i.i.d.\ uniform random variables in the area under the excursion.  For any suitable test function $g$,
\begin{align*}
& \E{g \left(\widetilde{\varepsilon}_i, \widetilde{h}_i\right) \I{M_i=m} \ \Big| \ \zeta(\widetilde{\varepsilon}_i) = x} \\
& \qquad = \E{g \left(\widetilde{\varepsilon}_i, \widetilde{h}_i\right) \exp \left(-\frac{1}{\mu} \int_0^{\infty} \widetilde{\varepsilon}_i(u) du \right) \frac{1}{m!} \left(\frac{1}{\mu} \int_0^{\infty} \widetilde{\varepsilon}_i(u) du \right)^m \ \Bigg| \ \zeta(\widetilde{\varepsilon}_i) = x} \\
& \qquad = \frac{\mathbb{E}_{\N^{(x)}} \left[ \exp \left( \frac{1}{\mu} \int_0^x \mathbbm{e}(t) dt \right) g(\mathbbm{e}, \mathbbm{h}) \exp \left(- \frac{1}{\mu} \int_0^x \mathbbm{e}(t) dt \right) \frac{1}{m!} \left( \frac{1}{\mu} \int_0^x \mathbbm{e}(t) dt \right)^m  \right]}{\mathbb{E}_{\N^{(x)}}  \left[ \exp \left( \frac{1}{\mu} \int_0^x \mathbbm{e}(t) dt \right) \right]}
\end{align*}
and so
\begin{align*}
\E{g \left(\widetilde{\varepsilon}_i, \widetilde{h}_i\right) \ \Big| \ \zeta(\widetilde{\varepsilon}_i) = x, M_i = m}  
& =  \frac{\mathbb{E}_{\N^{(x)}}  \left[ \left( \frac{1}{\mu} \int_0^x \mathbbm{e}(t) dt \right)^m g(\mathbbm{e}, \mathbbm{h})   \right]}{\mathbb{E}_{\N^{(x)}}  \left[\left( \frac{1}{\mu} \int_0^x \mathbbm{e}(t) dt \right)^m \right]} \\
& = \frac{\mathbb{E}_{\N^{(1)}}  \left[ \left(  \int_0^1 \mathbbm{e}(t) dt \right)^m g(x^{1/\alpha}\mathbbm{e}(\cdot/x), x^{(\alpha-1)/\alpha} \mathbbm{h}(\cdot/x)) \right]}{\mathbb{E}_{\N^{(1)}}  \left[ \left( \int_0^1 \mathbbm{e}(t) dt \right)^m \right]}.
\end{align*}
The claimed result follows.
\end{proof}

\section{Convergence of a discrete forest} \label{sec:forest}

The multigraph $\mathbf{M}_n(\nu)$ contains cycles with probability tending to 1 as $n \to \infty$.  However, its components will turn out to be tree-like, in that they each have a finite surplus, with probability 1. In this section, we study an idealised version of the depth-first walk of the multigraph, ignoring cycles.

Let $(\hat{D}_1^n, \hat{D}_2^n, \ldots, \hat{D}_k^n)$ be $D_1, D_2, \ldots, D_n$ arranged in size-biased random order.  More precisely, let $\Sigma$ be a random permutation of $\{1,2,\ldots,n\}$ such that
\[
\Prob{\Sigma = \sigma| D_1, \ldots, D_n} = \frac{D_{\sigma(1)}}{\sum_{j=1}^n D_{\sigma(j)}} \frac{D_{\sigma(2)}}{\sum_{j=2}^n D_{\sigma(j)}} \cdots \frac{D_{\sigma(n)}}{D_{\sigma(n)}}
\]
and define
\[
(\hat{D}_1^n, \hat{D}_2^n, \ldots, \hat{D}_n^n) = (D_{\Sigma(1)}, D_{\Sigma(2)}, \ldots, D_{\Sigma(n)}).
\]
Now let $\widetilde{S}^n(0) = 0$ and, for $k \ge 1$,
\[
\widetilde{S}^n(k) = \sum_{i=1}^k (\hat{D}_i^n - 2).
\]
Then $\widetilde{S}^n$ is the depth-first walk of a forest of trees in which the $i$th vertex visited in depth-first order has $\hat{D}_i^n - 1 \ge 0$ children.  Define the corresponding height process,
\[
\widetilde{G}^n(k) = \#\left\{j \in \{0,1,\ldots,k-1\}: \widetilde{S}^n(j) = \inf_{j \le \ell \le k} \widetilde{S}^n(\ell) \right\}.
\] 
The purpose of this section is to recover Theorem 8.1 of Joseph~\cite{Joseph} and, indeed, to strengthen it by adding the convergence of the height process to that of the depth-first walk. We will prove the following.

\begin{thm} \label{thm:jointdfwheight}
We have
\[
\left( n^{-\frac{1}{\alpha+1}} \widetilde{S}^n(\fl{n^{\frac{\alpha}{\alpha+1}} t}),  n^{-\frac{\alpha-1}{\alpha+1}} \widetilde{G}^n(\fl{n^{\frac{\alpha}{\alpha+1}} t}), t \ge 0\right)
\convdist (\widetilde{L}_t, \widetilde{H}_t, t \ge 0)
\]
as $n \to \infty$ in $\mathbb{D}(\R_+, \R)^2$.
\end{thm}

In order to prove this theorem, we will begin by showing that there is an analogue in the discrete setting of the change of measure used to define $\widetilde{L}$.

Write $Z_1, Z_2, \ldots, Z_n$ for i.i.d.\ random variables with the size-biased degree distribution, i.e.\
\[
\Prob{Z_1 = k} = \frac{k \nu_k}{\mu}, \quad k \ge 1.
\]
Observe that $\mu \in (1,2)$ since, firstly, $D_1 \ge 1$ and, secondly, $\E{D_1^2} = 2\mu$ and we must have $\var{D_1} = \mu(2-\mu) > 0$.
Then we have $\E{Z_1} = 2$, $\Prob{Z_1 \ge 1} = 1$ and $\Prob{Z_1 = k} \sim \frac{c}{\mu}k^{-(\alpha+1)}$ as $k \to \infty$ if $\alpha \in (1,2)$, or $\var{Z_1} = \beta/\mu$ if $\alpha = 2$.

\begin{prop} \label{prop:changemeas}
For any $k_1, k_2, \ldots, k_n \ge 1$, we have
\[
\Prob{\hat{D}_1^n = k_1, \hat{D}_2^n = k_2, \ldots, \hat{D}_n^n = k_n}
= k_1 \nu_{k_1} k_2 \nu_{k_2} \ldots k_n \nu_{k_n} \prod_{i=1}^n \frac{(n-i+1)}{\sum_{j=i}^n k_j}.
\]
Moreover, for $0 \le m \le n$ and $k_1, k_2, \ldots, k_m \ge 1$, let
\[
\phi_m^n(k_1, k_2, \ldots, k_m) := \E{\prod_{i=1}^m \frac{(n-i+1)\mu}{\sum_{j=i}^m k_j + \Xi_{n-m}}},
\]
where $\Xi_{n-m}$ has the same law as $D_{m+1}+D_{m+2} + \cdots + D_{n}$. Then for any suitable test-function $g: Z_+^m \to \R_+$,
\begin{equation} \label{eqn:discretechmeas}
\E{g(\hat{D}_1^n, \hat{D}_2^n, \ldots, \hat{D}_m^n)} = \E{\phi_m^n(Z_1,Z_2,\ldots,Z_m) g(Z_1, Z_2, \ldots, Z_m)}.
\end{equation}
\end{prop}

We have not found a precise reference for the contents of Proposition~\ref{prop:changemeas}.  The analogue of (\ref{eqn:discretechmeas}) for continuous random variables is equation (1) of Barouch \& Kaufman~\cite{BarouchKaufman}; see Proposition 1 of Pitman \& Tran~\cite{PitmanTran} for a proof.  The proof of Proposition~\ref{prop:changemeas} is elementary and may be found in the \hyperref[appn]{Appendix}.

We now show that the Radon--Nikodym derivative in the change-of-measure formula converges in distribution under appropriate conditions. Until the end of this section, we restrict our attention to the case $\alpha \in (1,2)$; the proof for the Brownian case is similar but a little more involved, so we defer it to Section \ref{subsec:changemeasbrown} in the \hyperref[appn]{Appendix}.

\begin{prop} \label{prop:discretechangeofmeas}
Let
\[
\Phi(n,m) := \phi_m^n(Z_1, Z_2, \ldots, Z_m) 
\]
and recall that
\[
\Phi(t) = \exp\left( - \frac{1}{\mu} \int_0^t s dL_s -\frac{C_{\alpha} t^{\alpha+1}}{(\alpha+1)\mu^{\alpha+1}} \right).
\]
Then for fixed $t > 0$, $\Phi(n,\fl{t n^{\frac{\alpha}{\alpha+1}}}) \convdist \Phi(t)$ as $n \to \infty$.  Moreover, the sequence of random variables $(\Phi(n,\fl{t n^{\frac{\alpha}{\alpha+1}}}))_{n \ge 1}$  is uniformly integrable.
\end{prop}

In order to prove this, we will need some technical lemmas. 

First, we consider the asymptotics of $S(k) = \sum_{i=1}^k (Z_i-2)$.  The generalised functional central limit theorem, Theorem~\ref{thm:FCLT} (ii), entails that
\begin{equation} \label{eqn:Levyconvergence}
n^{-1/(\alpha+1)} \left(S\left(\fl{tn^{\alpha/(\alpha+1)}}\right), t \ge 0\right) \convdist (L_t, t \ge 0)
\end{equation}
as $n \to \infty$ in $\mathbb{D}(\R_+,\R)$, where $L$ is the spectrally positive $\alpha$-stable L\'evy process introduced in the previous section. We will need to deal with functionals of $S$ converging, which we will do via the continuous mapping theorem (Theorem 3.2.4 of Durrett~\cite{Durrett}).  We give here the details for the functional which will arise most frequently in the sequel.

\begin{lem} \label{lem:Levyintegral}
For any $t \ge 0$, 
\[
\frac{1}{n} \sum_{k=0}^{\fl{tn^{\frac{\alpha}{\alpha+1}}}-1} S(k) \convdist \int_0^t L_s ds
\]
as $n \to \infty$.

\begin{proof}
We have
\begin{equation*}
\frac{1}{n} \sum_{k=0}^{\fl{tn^{\frac{\alpha}{\alpha+1}}}-1} S(k) = \frac{1}{n} \int_0^{\fl{t n^{\frac{\alpha}{\alpha+1}}}} S(\fl{v}) dv = \int_0^{n^{-\frac{\alpha}{\alpha+1}} \fl{tn^{\frac{\alpha}{\alpha+1}}}} n^{-\frac{1}{\alpha+1}}S(\fl{u n^{\frac{\alpha}{\alpha+1}}}) du,
\end{equation*}
by changing variable. Then the convergence in law follows from the fact that $n^{-\frac{\alpha}{\alpha+1}} \fl{t n^{\frac{\alpha}{\alpha+1}}} \to t$ and the continuous mapping theorem.
\end{proof}
\end{lem}

\begin{lem} Let $\mathcal{L}(\lambda) := \E{\exp(-\lambda D_1)}$.  Then as $\lambda \to 0$, 
\begin{equation} \label{eqn:Lasymp}
\mathcal{L}(\lambda) = \exp \left(-\mu \lambda + \frac{\mu(2-\mu)}{2} \lambda^2 - \frac{C_{\alpha}\lambda^{\alpha+1}}{(\alpha+1)} + o(\lambda^{\alpha+1}) \right).
\end{equation}

\end{lem}

\begin{proof}
First observe that
\[
\mathcal{L}'''(\lambda) = - \E{D_1^3 \exp(-\lambda D_1)} = - \sum_{k=1}^{\infty} k^3 e^{-\lambda k} \nu_k.
\]
Since $\nu_k \sim c k^{-(\alpha+2)}$, the right-hand side is finite and, by the Euler--Maclaurin formula, asymptotically equivalent to
\[
\int_0^{\infty} c x^{1-\alpha} e^{-\lambda x} dx = c\lambda^{\alpha-2} \Gamma(2-\alpha),
\]
as $\lambda \to 0$.  In other words,
\[
\mathcal{L}'''(\lambda) = - c \lambda^{\alpha-2} \Gamma(2-\alpha) + o(\lambda^{\alpha-2}),
\]
as $\lambda \to 0$. We also have $\E{D_1} = \mu$ and $\E{D_1^2} = 2\mu$.  So integrating three times, we obtain
\[
\mathcal{L}(\lambda) = 1 - \mu \lambda + \mu \lambda^2 -  \frac{c \Gamma(2-\alpha)}{(\alpha-1)\alpha(\alpha+1)} \lambda^{\alpha+1} + o(\lambda^{\alpha+1}),
\]
and it is straightforward to see that this implies
\[
\mathcal{L}(\lambda) = \exp \left(-\mu \lambda + \frac{\mu(2-\mu)}{2} \lambda^2 - \frac{C_{\alpha}\lambda^{\alpha+1}}{(\alpha+1)} + o(\lambda^{\alpha+1}) \right). \qedhere
\]
\end{proof}

\begin{lem} \label{lem:Laplace}
For $m = O(n^{\alpha/(\alpha+1)})$, we have
\[
\exp\left(m - \frac{(2+\mu)}{2\mu} \frac{m^2}{n}\right) \left[ \mathcal{L}\left(\frac{m}{n\mu} \right) \right]^{n-m} = (1+o(1)) \exp \left(- \frac{C_{\alpha} m^{\alpha+1}}{(\alpha+1)\mu^{\alpha+1}n^{\alpha}} \right).
\]
\end{lem}

\begin{proof}
By (\ref{eqn:Lasymp}), it is sufficient to show that 
\begin{align*}
& m - \frac{(2+\mu)}{2\mu} \frac{m^2}{n} + (n-m) \left(-\frac{\mu m}{n\mu} + \frac{\mu(2-\mu)}{2} \frac{m^2}{n^2\mu^2}- \frac{C_{\alpha}}{(\alpha+1)} \frac{m^{\alpha+1}}{n^{\alpha+1}\mu^{\alpha+1}}\right)\\
& \qquad = - \frac{C_{\alpha}m^{\alpha+1}}{(\alpha+1)\mu^{\alpha+1}n^{\alpha}} + o(1),
\end{align*}
as $n \to \infty$. But this is now easily seen to be true on cancellation and using $m = O(n^{\alpha/(\alpha+1)})$. The result follows.
\end{proof}

\begin{lem} \label{lem:lowerbound}
Let $s(0) = 0$ and $s(i) = \sum_{j= 1}^i (k_j-2)$ for $i \ge 1$. Then if $m = O(n^{\alpha/(\alpha+1)})$, we have
\[
\phi_n^m(k_1,k_2, \ldots,k_m) \ge \exp\left(\frac{1}{n \mu}\sum_{i=0}^{m} (s(i) - s(m)) - \frac{C_{\alpha}m^{\alpha+1}}{(\alpha+1)\mu^{\alpha+1} n^{\alpha}} \right) (1 + o(1)),
\]
where the $o(1)$ term is independent of $k_1, \ldots, k_m \ge 1$.
\end{lem}

\begin{proof}
First rewrite
\[
\prod_{i=1}^m (n-i+1) = n^m \prod_{i=1}^{m-1} \left(1 - \frac{i}{n} \right).
\]
Then
\begin{align*}
 \phi_n^m(k_1,k_2,\ldots,k_m) & =  \prod_{i=1}^{m-1} \left(1 - \frac{i}{n} \right)\E{ \prod_{i=1}^m\left(\frac{n \mu}{\sum_{j=i}^m k_j + \Xi_{n-m}} \right)} \\
& = \E{\exp \left( \sum_{i=1}^{m-1} \log\left(1 - \frac{i}{n} \right) - \sum_{i=1}^m \log \left(\frac{\Xi_{n-m}}{n \mu} + \frac{1}{n \mu} \sum_{j=i}^m  k_j\right) \right)}.
\end{align*}
Now note that for any $x \in (-1,\infty)$, we have $\log(1+x) \le x$.  We also have $\log(1-i/n) \ge -i/n -m^2/n^2$ for $1 \le i \le m-1$.  So
\begin{align*}
& \phi_n^m(k_1,k_2,\ldots,k_m) \\
& \ge \E{ \exp \left( -\sum_{i=1}^{m-1} \frac{i}{n} - \frac{m^3}{n^2}- m\left[\frac{\Xi_{n-m}}{n \mu} - 1\right] - \frac{1}{n \mu}\sum_{i=1}^m  \sum_{j=i}^m  k_j \right)} \\
& = \exp \left( -\frac{m(m-1)}{2n} - \frac{m^3}{n^2}  + m + \frac{1}{n\mu} \sum_{i=1}^m \Big(s(i) - s(m) - 2(m-i+1)\Big) \right) \\
& \qquad \times \E{\exp \left(- \frac{m}{n\mu} \Xi_{n-m}\right)} \\
& =  \exp \left(- \frac{m(m-1)}{2n} - \frac{m^3}{n^2} + m + \frac{1}{n \mu}\sum_{i=0}^{m} (s(i) - s(m)) - \frac{m(m+1)}{n \mu} \right) \\
& \qquad \times \E{\exp \left( -\frac{m}{n \mu}\Xi_{n-m} \right)}  \\
& = \exp \left(\frac{1}{n \mu}\sum_{i=0}^{m} (s(i) - s(m)) \right) \exp\left(m - \frac{(2+\mu)}{2\mu} \frac{m^2}{n}\right) \left[ \mathcal{L}\left(\frac{m}{n\mu} \right) \right]^{n-m} \\ 
& \qquad \times \exp\left(\frac{(\mu-2)m}{2\mu n} - \frac{m^3}{n^2}\right).
\end{align*}
We have $m^3/n^2  = O(n^{\frac{\alpha-2}{\alpha+1}}) = o(1)$ and so the final exponential tends to 1 as $n \to \infty$. The desired result then follows from Lemma~\ref{lem:Laplace}.
\end{proof}

\begin{lem} \label{lem:James}
Let $(X_n)_{n \ge 1}$, $(Y_n)_{n \ge 1}$ be two sequences of non-negative random variables such that $X_n \ge Y_n$ and $\E{X_n} = 1$ for all $n$.  Suppose that $X$ is another non-negative random variable such that $\E{X} = 1$ and $Y_n \convdist X$.  Then $X_n \convdist X$ and $(X_n)_{n \ge 1}$ is uniformly integrable.
\end{lem}

\begin{proof}
Since $X_n \ge Y_n$, we have that
\[
\limsup_{n \to \infty} \E{|X_n - Y_n|} = \limsup_{n \to \infty} \E{X_n - Y_n} = 1 - \liminf_{n \to \infty} \E{Y_n}.
\]
But by the Portmanteau theorem,
\[
\liminf_{n \to \infty} \E{|Y_n|} \ge \E{X} = 1
\]
and so $X_n - Y_n \to 0$ in $L^1$. It follows that $X_n - Y_n \convdist 0$.  But then by Slutsky's lemma we have $X_n \convdist X$.
Now, we prove that $X_n \to X$ in $L^1$, which will imply the desired uniform integrability. By Skorokhod's representation theorem, we may work on a probability space where $X_n \to X$ almost surely. But then we also have $\E{X_n} = \E{X} = 1$ for all $n$ and so $\E{|X_n - X|} \to 0$ by Scheff\'e's lemma.
\end{proof}

\begin{proof}[Proof of Proposition~\ref{prop:discretechangeofmeas}]
Recall that $S(k) = \sum_{i=1}^k (Z_i - 2)$.  By Lemma~\ref{lem:lowerbound}, for $m=\fl{t n^{\frac{\alpha}{\alpha+1}}}$,
\[
\Phi(n,m) \ge \underline{\Phi}(n,m) := \exp\left(\frac{1}{n \mu}\sum_{i=0}^{m} (S(i) - S(m)) - \frac{C_{\alpha} m^{\alpha+1}}{(\alpha+1) \mu^{\alpha+1}n^{\alpha}} \right) (1+o(1)).
\]
By (\ref{eqn:Levyconvergence}) and Lemma~\ref{lem:Levyintegral}, we have
\[
\qquad \frac{1}{n \mu} \sum_{i=0}^{m} (S(i) - S(m)) \convdist \frac{1}{\mu} \int_0^t (L_s -L_t)ds.
\]
Hence, by the continuous mapping theorem,
\begin{align*}
\underline{\Phi}(n,\fl{t n^{\frac{\alpha}{\alpha+1}}}) & \convdist \exp\left(\frac{1}{\mu} \int_0^t (L_s -L_t) ds - \frac{C_{\alpha} t^{\alpha+1}}{(\alpha+1) \mu^{\alpha+1}} \right)\\
& \qquad = \exp \left( - \frac{1}{\mu} \int_0^t s dL_s - \frac{C_{\alpha} t^{\alpha+1}}{(\alpha+1) \mu^{\alpha+1}} \right),
\end{align*}
and the right-hand side is, of course, $\Phi(t)$, which has mean 1. We also have \linebreak $\E{\Phi(n,\fl{t n^{\frac{\alpha}{\alpha+1}}})} = 1$ for all $n$.  So by Lemma~\ref{lem:James}, we must have
\[
\Phi(n,\fl{t n^{\frac{\alpha}{\alpha+1}}}) \convdist \Phi(t)
\]
as $n \to \infty$, as well as the claimed uniform integrability.
\end{proof}

We are now ready to prove Theorem~\ref{thm:jointdfwheight}.

\begin{proof}[Proof of Theorem~\ref{thm:jointdfwheight}]
It is sufficient to show that for any $t \ge 0$ and any bounded continuous test-function $f: \mathbb{D}([0,t],\R)^2 \to \R$,
\[
\E{f\left( n^{-\frac{1}{\alpha+1}} \widetilde{S}^n(\fl{n^{\frac{\alpha}{\alpha+1}} u}),  n^{-\frac{\alpha-1}{\alpha+1}} \widetilde{G}^n(\fl{n^{\frac{\alpha}{\alpha+1}} u}), 0 \le u \le t\right)} \to \E{f(\widetilde{L}_u, \widetilde{H}_u, 0 \le u \le t)},
\]
as $n \to \infty$.  Let us write $\overline{S}^n(u) = n^{-\frac{1}{\alpha+1}} S(\fl{n^{\frac{\alpha}{\alpha+1}} u})$ and, similarly, \linebreak $\overline{G}^n(u)= n^{-\frac{\alpha-1}{\alpha+1}} G(\fl{n^{\frac{\alpha}{\alpha+1}} u})$. Then, by changing measure, we wish to show that for any $t \ge 0$ and any bounded continuous test-function $f: \mathbb{D}([0,t],\R)^2 \to \R$,
\[
\E{\Phi(n,\fl{t n^{\frac{\alpha}{\alpha+1}}}) f(\overline{S}^n(u), \overline{G}^n(u), 0 \le u \le t)} \to \E{\Phi(t) f(L_u, H_u, 0 \le u \le t)},
\]
as $n \to \infty$. From the proof of Proposition~\ref{prop:discretechangeofmeas}, we have that
\[
\E{ \Big|\Phi(n,\fl{t n^{\frac{\alpha}{\alpha+1}}}) - \underline{\Phi}(n,\fl{t n^{\frac{\alpha}{\alpha+1}}}) \Big|} \to 0
\]
as $n \to \infty$, and so it will suffice to show that
\[
\E{\underline{\Phi}(n,\fl{t n^{\frac{\alpha}{\alpha+1}}}) f(\overline{S}^n(u), \overline{G}^n(u), 0 \le u \le t)} \to \E{\Phi(t) f(L_u, H_u, 0 \le u \le t)}.
\]
But
\[
\underline{\Phi}(n,\fl{t n^{\frac{\alpha}{\alpha+1}}}) = \exp \left( \frac{1}{\mu} \int_0^t (\overline{S}^n(u) -\overline{S}^n(t))du - \frac{C_{\alpha}\fl{tn^{\frac{\alpha}{\alpha+1}}}^{\alpha+1}} {(\alpha+1)\mu^{\alpha+1}n^{\alpha}} \right).
\]
In particular, for a path $x \in \mathbb{D}([0,t],\R)$, let
\[
\Theta(x,t) = \exp \left(  \frac{1}{\mu} \int_0^t (x(u) -x(t)) du - \frac{C_{\alpha}t^{\alpha+1}}{(\alpha+1)\mu^{\alpha+1}} \right)
\]
and observe that $\Theta$ is a continuous functional of its first argument.
Then we have
\[
\E{ \Big| \underline{\Phi}(n,\fl{t n^{\frac{\alpha}{\alpha+1}}}) - \Theta(\overline{S}^n,t) \Big| } \to 0.
\]
So it suffices to show that
\[
\E{\Theta(\overline{S}^n,t) f(\overline{S}^n(u), \overline{G}^n(u), 0 \le u \le t)} \to \E{\Theta(L, t) f(L_u, H_u, 0 \le u \le t)}.
\]
But this now follows from Theorem~\ref{thm:DLGconvergence} and uniform integrability.
\end{proof}

\section{The configuration multigraph} \label{sec:graph}
The processes $\widetilde{S}^n$ and $\widetilde{G}^n$ encode a forest of trees where the numbers of children of the vertices, visited in depth-first order, are $\hat{D}^n_i-1$, $i \ge 1$.  Let us write $\widetilde{\mathbf{F}}_n(\nu)$ for this forest. In this section, we wish to encode similarly the multigraph $\mathbf{M}_n(\nu)$. Let us first describe the organisation of this section.

In Section~\ref{subsec:explo}, we simultaneously generate and explore $\mathbf{M}_n(\nu)$ using the depth-first approach outlined in the Introduction: we view each connected component of the graph as a spanning tree explored in a depth-first manner plus some additional edges, creating cycles, that we call \textit{back-edges}. In Section~\ref{subsec:cv-dfw}, we prove that the exploration process is close enough to $\widetilde{S}^n$ in order to have the same scaling limit, and add the joint convergence of the locations of the back-edges in Section~\ref{subsec:backedges} (Theorem \ref{thm:Coxprocess}). In Section~\ref{subsec:cmpts}, we split the multigraph into its components by showing that the (rescaled) ordered sequence of component sizes converges to the sequence of ordered excursion lengths of the continuous process $R$ (Proposition~\ref{prop:ordered}). We improve this result in Section~\ref{subsec:markedexc} by adding in the locations of the back-edges under each excursion (Proposition \ref{prop:orderedexmarks}). In Section~\ref{subsec:heightp}, we study the height process and show in Proposition~\ref{prop:joint} the joint convergence of the height process excursions and the locations of the back-edges. This finally allows us to prove Theorem \ref{thm:main} in Section~\ref{subsec:metricstructure}. 

\subsection{Exploration of the multigraph}\label{subsec:explo}
We work conditionally on the sequence \linebreak $(\hat{D}_1^n, \hat{D}_2^n, \ldots, \hat{D}_n^n)$.  Let us declare that the vertex of degree $\hat{D}_i^n$ is called $v_i$.  This means that we have already determined the (size-biased by degree) order in which we will observe new vertices.  We will couple $\widetilde{\mathbf{F}}_n(\nu)$ and $\mathbf{M}_n(\nu)$ by using the same ordering on the new vertices we explore.

Recall that we start from vertex $v_1$ with degree $\hat{D}_1^n$.  We maintain a \emph{stack}, namely an ordered list of half-edges which we have seen but not yet explored (remember that the half-edges come with an arbitrary labelling for this purpose).  We put the $\hat{D}_1^n$ half-edges of $v_1$ onto this stack, in increasing order of label, so that the lowest labelled half-edge is on top of the stack.  At a subsequent step, suppose we have already seen the vertices $v_1, v_2, \ldots, v_k$.  If the stack is non-empty, take the top half-edge and sample its pair.  This lies on the stack with probability proportional to the height of the stack minus 1 or belongs to $v_{k+1}$ with probability proportional to $\sum_{i=k+1}^n \hat{D}_i^n$.  In the first case we simply remove both half-edges from the stack.  In the second, we remove the half-edge at the top of the stack (which has just been paired) and replace it by the remaining half-edges (if any) of $v_{k+1}$.  If the stack is empty, we start a new component at $v_{k+1}$.

Let us now describe the forest $\mathbf{F}_n(\nu)$ from which we will recover $\mathbf{M}_n(\nu)$.  Whenever there is a back-edge in $\mathbf{M}_n(\nu)$, say from vertex $v_k$ to vertex $v_i$ with $i \le k$, remove the back-edge and replace it by two edges, one from $v_k$ to a new leaf and the other from $v_i$ to a new leaf.  To recover $\mathbf{M}_n(\nu)$ it is then sufficient to remove the edges to the new leaves and put in a new edge from $v_i$ to $v_k$.

Our aim is to encode the forest $\mathbf{F}_n(\nu)$, firstly via its depth-first walk and then by its height process. We will simultaneously keep track of marks which tell us which vertices we should identify in order to recover the multigraph.

Our first observation is that the vertex-sets of \emph{pairs} of components in $\widetilde{\mathbf{F}}_n(\nu)$ correspond precisely to the vertex-sets of subcollections of components in $\mathbf{M}_n(\nu)$.  (This is illustrated in Figure~\ref{fig:unplugging} to which the reader is referred in the following argument.) More precisely, without loss of generality, let the pair of components of $\widetilde{\mathbf{F}}_n(\nu)$ be the first two, on vertices $v_1, v_2, \ldots, v_{m_1}$ and $v_{m_1+1}, v_{m_1+2}, \ldots, v_{m_1+m_2}$.  Suppose that the same vertices in $\mathbf{M}_n(\nu)$ are adjacent to $b$ back-edges.  Now vertices $v_1$ and $v_{m_1+1}$ each possess one more half-edge in $\mathbf{M}_n(\nu)$ than they do in $\widetilde{\mathbf{F}}_n(\nu)$ (since in $\widetilde{\mathbf{F}}_n(\nu)$, if $v_i$ is the first vertex of a component, it has $\hat{D}_i-1$ edges and not $\hat{D}_i$).  In particular, adding an edge between $v_1$ and $v_{m_1+1}$ clearly produces a tree with $m_1+m_2$ vertices and $\frac{1}{2} \sum_{i=1}^{m_1+m_2} \hat{D}_i^n = m_1 + m_2 - 1$ edges.  We now ``rewire'' this tree to obtain the relevant components of $\mathbf{M}_n(\nu)$. The effect of adding a back-edge is to shunt all of the subsequent subtrees along in the depth-first order.  (See Figure~\ref{fig:unplugging}.)  The overall effect is that each back-edge causes a new component to come into existence.  Each time we observe a back-edge, it occupies two half-edges, so there are two subtrees which get pushed out of the component.  The earlier of these subtrees in the depth-first order becomes the basis for the next component of $\mathbf{M}_n(\nu)$.  The root of this component has one more child than it had in the original tree.  This allows the absorption of the second subtree, whose root gets attached by its free half-edge to the root of the component.  Subsequent back-edges similarly each generate one new component. Following this through, we see that we end up with $b+1$ components of $\mathbf{M}_n(\nu)$.  (For the purposes of intuition, note that because the vast majority of vertices lie in components of size $o(n^{\alpha/(\alpha+1)})$, with high probability at most one of them will be of size $\Theta(n^{\alpha/(\alpha+1)})$ and thus show up in the limit. So, at least heuristically, this rewiring process cannot affect what we see in the limit.)

\begin{figure}
\begin{center}
\includegraphics[width=9cm]{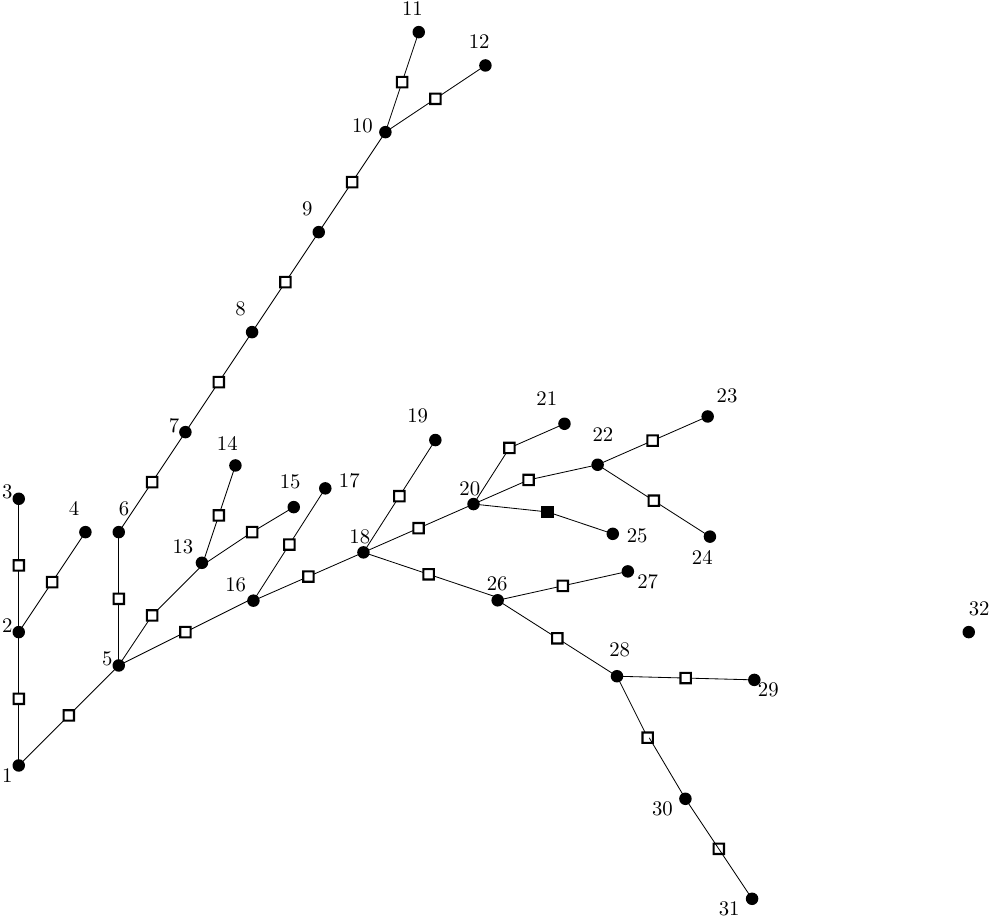} 

\includegraphics[width=9cm]{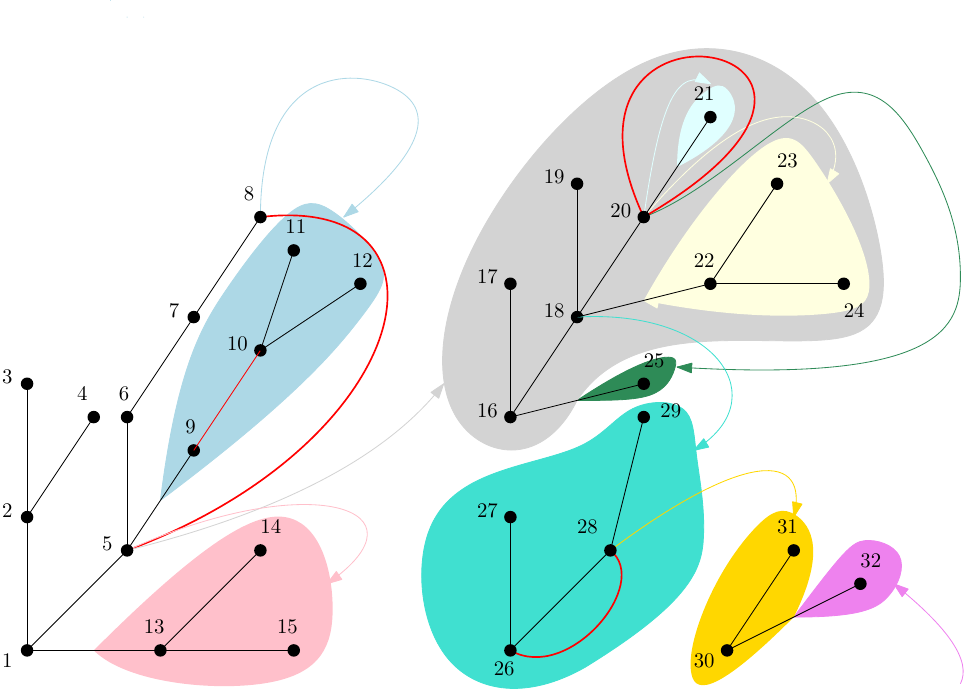} 

\includegraphics[width=9cm]{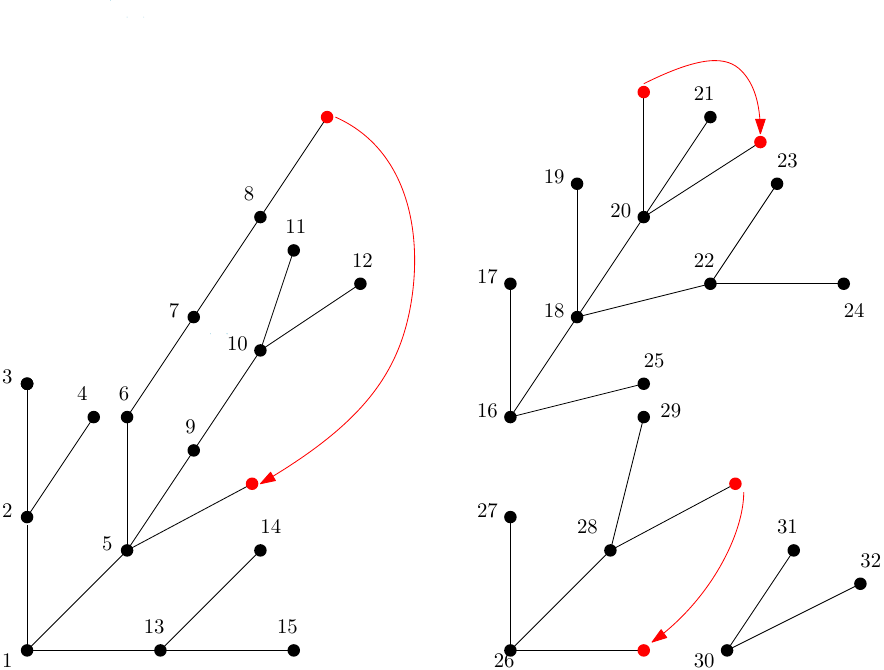}
\end{center}
\caption{For simplicity, the labels given are those corresponding to the depth-first order. Top: the first two components of the forest $\widetilde{\mathbf{F}}_n(\nu)$.  Middle: the first four components of $\mathbf{M}_n(\nu)$, on the same vertices.  Three back-edges are marked in red. The subtree surgery required to get from $\widetilde{\mathbf{F}}_n(\nu)$ to $\mathbf{M}_n(\nu)$ is indicated. The back-edge from $8$ to $5$ moves the subtree rooted at $9$ to the next available half-edge, also belonging to $5$. This shifts further the subtrees rooted at $13$ and $16$: $13$ gets connected to $1$, which has one more edge in $\mathbf{M}_n(\nu)$ than in $\widetilde{\mathbf{F}}_n(\nu)$ ($\hat{D}_1$ instead of $\hat{D}_1-1$), as any first vertex of a connected component of $\mathbf{M}_n(\nu)$. Vertex $16$ starts a new component. 
Bottom: the first four components of the forest $\mathbf{F}_n(\nu)$ with marks.}

\label{fig:unplugging}
\end{figure}

It is clear that the effects of adding back-edges are relatively local and so it is at least intuitively clear that the depth-first walk of the forest $\mathbf{F}_n(\nu)$ should be similar to that of $\widetilde{\mathbf{F}}_n(\nu)$, as long as there are not too many back-edges.  Let $X^n$ denote the depth-first walk of $\mathbf{F}_n(\nu)$.  We will now describe how to construct $X^n$ from $\widetilde{S}^n$, and also how to keep track of the back-edges. We will write $R^n(k)$ for the number of half-edges on the stack at step $k$.  We will let $N^n$ count the occurrences of back-edges, and $U^n$ the positions of their targets on the stack. We will write $\mathcal{M}^n(k)$ for a set of marks (in $\mathbb{N}$) at step $k$, indicating back-edges which have not yet been closed, and $\tau_n(k)$ for the number of vertices already seen at step $k$ (note that we see a new vertex if and only if the current step does not involve a back-edge). Finally, let $C^n(k)$ be the number of components of $\mathbf{F}_n(\nu)$ we have fully explored by time $k$. So we will have that for all $k\ge 1$, $C^n(k)=-\min_{0\le \ell\le k}X^n(\ell)$, and that $R^n(k) = X^n(k) + C^n(k)$.

We start from $X^n(0) =  N^n(0)  = 0$, $\mathcal{M}^n(0) = \emptyset$ and $\tau_n(0) = 0$.  For $k \ge 0$, we might encounter the following three situations.
\begin{itemize}
\item  \textbf{New component.}\\
If $X^n(k) = \min_{0 \le i \le k-1} X^n(i)-1$ or $k=0$, let $\tau_n(k+1) = \tau_n(k)+1$, $N^n(k+1) = N^n(k)$, $\mathcal{M}^n(k+1) = \mathcal{M}^n(k)$ and
\[
X^n(k+1) = X^n(k) + \widetilde{S}^n(\tau_n(k)+1) - \widetilde{S}^n(\tau_n(k)) + 1.
\]
\item \textbf{Start a back-edge or not.}
\\
If $X^n(k) > \min_{0 \le i \le k-1} X^n(i)-1$ and $X^n(k) \notin \mathcal{M}^n(k)$,
\begin{itemize}
\item With probability
\[
\frac{X^n(k) - \min_{0 \le i \le k} X^n(k)- |\mathcal{M}^n(k)|}{X^n(k) - \min_{0 \le i \le k} X^n(k) - |\mathcal{M}^n(k)|+ \sum_{j = \tau_n(k)+1}^{n} \hat{D}_j^n},
\]
let $\tau_n(k+1) = \tau_n(k)$ and $X^n(k+1) = X^n(k) - 1$.  Let $N^n(k+1) = N^n(k)+1$, sample $U^n(k+1)$ uniformly from 
\[
\left\{\min_{0 \le i \le k}X^n(i), \ldots, X^n(k)-1\right\} \setminus \mathcal{M}^n(k),
\]
and let $\mathcal{M}^n(k+1) = \mathcal{M}^n(k) \cup \{U^n(k+1)\}$.
\item With probability
\[
\frac{\sum_{j = \tau_n(k)+1}^{n} \hat{D}_j^n}{X^n(k) - \min_{0 \le i \le k} X^n(k) - |\mathcal{M}^n(k)| + \sum_{j = \tau_n(k)+1}^{n} \hat{D}_j^n}
\]
let $\tau_n(k+1) = \tau_n(k) + 1$, 
\[
X^n(k+1) = X^n(k) + \widetilde{S}^n(\tau_n(k)+1) - \widetilde{S}^n(\tau_n(k)),
\]
$N^n(k+1) = N^n(k)$ and $\mathcal{M}^n(k+1) = \mathcal{M}^n(k)$.
\end{itemize}
\item \textbf{Close a back-edge?}
\\ If $X^n(k) > \min_{0 \le i \le k-1} X^n(i)-1$ and $X^n(k) \in \mathcal{M}^n(k)$ then let $\tau_n(k+1) = \tau_n(k)$, $X^n(k+1) = X^n(k) - 1$, $N^n(k+1) = N^n(k)$, and $\mathcal{M}^n(k+1) = \mathcal{M}^n(k) \setminus \{X^n(k)\}$.
\end{itemize}

It is straightforward to check that this is the depth-first walk of the forest $\mathbf{F}^n(\nu)$.  We observe that for $k \ge 1$ we have
\[
X^n(k)+ \min_{0 \le i \le k} X^n(i) = \widetilde{S}^n(\tau_n(k)) +1  - N^n(k) - \#\{i \le k: |\mathcal{M}^n(i)| < |\mathcal{M}^n(i-1)|\}.
\]
Hence, $X^n(k)+\min_{0 \le i \le k} X^n(i)$ is the number of half-edges seen but not yet paired or reserved for back-edges, in the currently explored connected component (if $-\min_{0 \le i \le k} X^n(i)=j$, this is the $(j+1)$-th component). Indeed,
\[
\widetilde{S}^n(\tau_n(k)) +1=(\hat{D}_1-1)-1 +\hat{D}_2-2+\ldots +\hat{D}_{\tau_n(k)-1} -2+\hat{D}_{\tau_n(k)} -1 
\]
is the number of half-edges seen and not paired in the spanning forest of the components explored so far. $N^n(k) + \#\{i \le k: |\mathcal{M}^n(i)| < |\mathcal{M}^n(i-1)|\}$ is the number of half-edges that have been paired in or reserved for back-edges. Note that $\vert \mathcal{M}^n(i)\vert$ increases (resp. decreases) by 1 exactly when a back-edge is initiated (resp. closed). $N^n(k)$ counts the number of back-edges started.

In particular,
\begin{equation} \label{eqn:close}
\widetilde{S}^n(\tau_n(k)) - 2N^n(k) \le X^n(k) + \min_{0 \le i \le k} X^n(i) -1\le \widetilde{S}^n(\tau_n(k)).
\end{equation}

\subsection{Convergence of the depth-first walk and marks}\label{subsec:cv-dfw}
Let us first prove a bound on the number $N^n(k)$ of back-edges which have occurred by step $k$. 

\begin{lem} \label{lem:back}
For every $t>0$, the sequences of random variables $\left( N^n(\fl{t n^{\alpha/(\alpha+1)}})\right)_{n\geq 1}$ and \\
$\left( \sup_{0 \le k \le \fl{tn^{\alpha/(\alpha+1)}}} |\tau_n(k) - k|\right)_{n\geq 1}$ are tight. 
\end{lem}

%

\begin{proof}
Fix $t>0$. We observe that $k - 2N^n(k) \le \tau_n(k) \le k$ so that it is enough to prove that $\left( N^n(\fl{t n^{\alpha/(\alpha+1)}})\right)_{n \ge 1}$ is tight. As underlined at the beginning of Section~\ref{subsec:explo}, we can realize $\mathbf{M}_n(\nu)$ by first generating the sequence $(\hat{D}_k^n)_{1\leq k \leq n}$, and then proceeding to pair the half-edges as we perform the exploration. At time $i$, the number of half-edges on the stack is $R^n(i)$, and the total number of unpaired half-edges is $\sum_{j=\tau_n(i)+1}^n\hat{D}_j^n$, so that the conditional probability of starting a back-edge is 
\[
\frac{R^n(i)}{\sum_{j=\tau_n(i)+1}^n\hat{D}_j^n}
\]
Note that if the component we are exploring at step $i$ began at step $j$, then $R^n(i)\le 2+\sum_{k=\tau_n(j)}^{\tau_n(i)} (\hat{D}^n_k-2)$, since at most $\hat{D}^n_{\tau_n(j)}+\ldots + \hat{D}^n_{\tau_n(i)}$ half-edges have been seen, and at least $2(\tau_n(i)-\tau_n(j))$ of them have been used to connect the first $\tau_n(i)-\tau_n(j)+1$ vertices of the component. Hence, 
\begin{align*}
R^n(i)& \le  2 + \max_{0\le j \le i}\widetilde{S}^n(\tau_n(j))-\min_{0\leq j\leq i}\widetilde{S}^n(\tau_n(j)) \le 2 + \max_{0\le j \le i}\widetilde{S}^n(j)-\min_{0\leq j\leq i}\widetilde{S}^n(j),
\end{align*}
since $\tau_n(\{1, \ldots, i\})\subset \{1, \ldots, i\}$.  Moreover, $\tau_n(i) +1\leq \fl{t n^{\alpha/(\alpha+1)}} +1 $ for every $i\leq \fl{t n^{\alpha/(\alpha+1)}}$.  Therefore, conditionally on $(\hat{D}_k^n)_{1\leq k \leq n}$, for every $1 \le i \le \fl{t n^{\alpha/(\alpha+1)}}$, irrespective of what happened in the first $i-1$ steps of the exploration, the probability of creating a back-edge at time $i$ is at most 
\[
\frac{2 + \max_{0\le j \le \fl{t n^{\alpha/(\alpha+1)}}} \widetilde{S}^n(j)-\min_{0\leq j\leq \fl{t n^{\alpha/(\alpha+1)}}}\widetilde{S}^n(j)}{\sum_{j=\fl{t n^{\alpha/(\alpha+1)}} +1}^n\hat{D}_j^n},
\] 
which is a measurable function of $(\hat{D}_i^n)_{1\leq i \leq n}$ only.  Therefore, conditionally on $(\hat{D}_i^n)_{1\leq i \leq n}$, the random variable $N^n(\fl{t n^{\alpha/(\alpha+1)}})$ is stochastically dominated by a Binomial random variable with parameters $\fl{tn^{\alpha/(\alpha+1)}}$ and
\[
\frac{2 + \max_{0\le j \le \fl{t n^{\alpha/(\alpha+1)}}} \widetilde{S}^n(j)-\min_{0\leq j\leq \fl{t n^{\alpha/(\alpha+1)}}}\widetilde{S}^n(j)}{\sum_{j=\fl{t n^{\alpha/(\alpha+1)}} +1}^n\hat{D}_j^n},
\] 
For $K > 0$, define 
\[
\mathcal{E}_1=\left\lbrace 2 + \max_{1\le i\le \fl{t n^{\alpha/(\alpha+1)}}}\widetilde{S}^n(i)-\min_{0\le i\le  \fl{t n^{\alpha/(\alpha+1)}}}\widetilde{S}^n(i) \le Kn^{1/(\alpha+1)}\right\rbrace
\]
and 
\[
\mathcal{E}_2 = \left\lbrace\sum_{j=\fl{t n^{\alpha/(\alpha+1)}} + 1}^n\hat{D}_j^n \ge n\right\rbrace.
\] 
Fix $\epsilon > 0$. Theorem \ref{thm:jointdfwheight}, Lemma \ref{lem:lln} and the fact that $\mu > 1$ imply that there exists $K>0$ such that for $n$ large enough,
\[
\Prob{\mathcal{E}_1 \cap \mathcal{E}_2}\ge 1-\epsilon.
\]
On the event $\mathcal{E}_1 \cap \mathcal{E}_2$, we have
\[
\frac{2 + \max_{0\le j \le \fl{t n^{\alpha/(\alpha+1)}}} \widetilde{S}^n(j)-\min_{0\leq j\leq \fl{t n^{\alpha/(\alpha+1)}}}\widetilde{S}^n(j)}{\sum_{j=\fl{t n^{\alpha/(\alpha+1)}} +1}^n\hat{D}_j^n} \le \frac{K}{n^{\alpha/(\alpha +1)}}.
\]
Let $Y\sim \mathrm{Bin}(\fl{tn^{\alpha/(\alpha+1)}},2K/n^{\alpha/(\alpha +1)} )$. Then there exists $K'>0$ such that for $n$ large enough,
\[
\Prob{ N^n(\fl{t n^{\alpha/(\alpha+1)}}) \ge K'}\le \Prob{Y\ge K'} +\Prob{(\mathcal{E}_1\cap\mathcal{E}_2)^c}\le 2\epsilon.
\]
Since $\epsilon > 0$ was arbitrary, the result follows.
\end{proof}

In particular, the steps on which back-edges occur are negligible on the timescale in which we are interested.
Write $\widetilde{I}_t = \inf_{0 \le s \le t} \widetilde{L}_s$ for $t\ge 0$ and recall that
\[
R_t = \widetilde{L}_t - \widetilde{I}_t.
\]

\begin{prop} \label{prop:Xconv}
As $n \to \infty$,
\begin{align*}
& \left(n^{-1/(\alpha+1)} \widetilde{S}^n(\fl{t n^{\alpha/(\alpha+1)}}), n^{-1/(\alpha+1)} R^n(\fl{tn^{\alpha/(\alpha+1)}}), n^{-1/(\alpha+1)} C^n(\fl{tn^{\alpha/(\alpha+1)}}), t \ge 0 \right) \\
& \qquad \convdist \left(\widetilde{L}_t, R_t, -\tfrac{1}{2} \widetilde{I}_t, t \ge 0 \right),
\end{align*}
in $\mathbb{D}(\R_+, \R)^3$.
\end{prop}

\begin{proof}
Recall that (\ref{eqn:close}) says that
\[
\widetilde{S}^n(\tau_n(k)) - 2N^n(k) \le X^n(k) + \min_{0 \le i \le k} X^n(i) -1\le \widetilde{S}^n(\tau_n(k)).
\]
Let $1 \le j \le k$ be such that $X^n(j) = \min_{0 \le \ell \le k} X^n(\ell)$.  Then $X^n(j) + \min_{0 \le \ell \le j} X^n(\ell) = 2X^n(j)$ and so 
\[
\min_{0 \le \ell \le k} \left(X^n(\ell) + \min_{0 \le i \le \ell} X^n(i)\right) = 2 X^n(j) = 2 \min_{0 \le \ell \le k} X^n(\ell).
\]
It follows that
\begin{align*}
2 \min_{0 \le \ell \le k} X^n(\ell) \geq & 1+\min_{0 \le \ell \le k}(\widetilde{S}^n(\tau_n(\ell)) - 2 N^n(\ell))
\\
\geq & 1+ \min_{0 \le \ell \le k} \widetilde{S}^n(\tau_n(\ell)) - 2 N^n(k),
\end{align*}
since $(N^n(i))_{i\geq 0}$ is non-increasing. 
Thus, by (\ref{eqn:close}) again, we have
\[
1+ \min_{0 \le \ell \le k} \widetilde{S}^n(\tau_n(\ell)) - 2 N^n(k) \le 2 \min_{0 \le \ell \le k} X^n(\ell) \le 1+\min_{0 \le \ell \le k} \widetilde{S}^n(\tau_n(\ell)).
\]
By Theorem~\ref{thm:jointdfwheight} and the continuous mapping theorem,
\[
\left(n^{-1/(\alpha+1)} \min_{0 \le \ell \le \fl{s n^{\alpha/(\alpha+1)}}} \widetilde{S}^n(\ell), s \ge 0\right) \convdist (\widetilde{I}_s, s\ge 0)
\]
and combining this with Lemma~\ref{lem:back} and recalling that $C^n(k) = - \min_{0 \le \ell \le k} X^n(\ell)$ yields 
\[
\left(n^{-1/(\alpha+1)} C^n(\fl{s n^{\alpha/(\alpha+1)}}), s \ge 0\right) \convdist (-\tfrac{1}{2} \widetilde{I}_s, s\ge 0)
\]
Another application of (\ref{eqn:close}) gives
\[
1+ \widetilde{S}^n(\tau_n(k)) - 2 N^n(k) - 2 \min_{0 \le i \le k} X^n(i) \le R^n(k) \le 1+ \widetilde{S}^n(\tau_n(k)) - 2 \min_{0 \le i \le k} X^n(i)
\]
and since
\[
(n^{-1/(\alpha+1)} \widetilde{S}^n(\tau_n(\fl{t n^{\alpha/(\alpha+1)}})), t \ge 0) \convdist (\widetilde{L}_t, t \ge 0)
\]
we have
\[
\left(n^{-1/(\alpha+1)} R^n(\fl{tn^{\alpha/(\alpha+1)}}), t \ge 0\right) \convdist (R_t, t \ge 0),
\]
jointly with the convergence of the minimum.
\end{proof}

Thus the exploration of $\widetilde{\mathbf{F}}_n(\nu)$ sees approximately twice as many components as that of $\mathbf{F}_n(\nu)$ but the limiting reflected process is the same for both.  In particular, asymptotically the two processes have the same longest excursions.  This fact will play an important role in the sequel.

\subsection{Back-edges} \label{subsec:backedges}

We will now show that the parts of the multigraph we observe up until well beyond the timescale in which we are interested are, with high probability, simple.  To this end, let $A^n(k)$ be the number of loops and edges created parallel to an existing edge, up until step $k$ of the depth-first exploration of $\mathbf{F}_n(\nu)$.  Call these \emph{anomalous} edges.

\begin{prop} \label{prop:backedgesarelate}
Suppose $\frac{\alpha}{\alpha+1} < \beta < \frac{\alpha}{2}$.  Then we have
\[
\Prob{A^n(\fl{n^{\beta}}) > 0} \to 0
\]
as $n \to \infty$.
\end{prop}

\begin{proof}
We adapt the proof of Lemma 7.1 of Joseph~\cite{Joseph} (which applies in the finite third moment setting).  Self-loops are obviously associated with a unique vertex.  We associate extra edges created parallel to an existing edge with their vertex which is discovered first in the depth-first exploration.  Consider a particular vertex of degree $d$ in the exploration before time $\fl{n^{\beta}}$.  Its $k$th half-edge (in the order that we process them) creates a self-loop with probability bounded above by
\[
\frac{d-k}{\sum_{i=\lfloor n^{\beta}\rfloor+1}^n \hat{D}^n_i}.
\]
It creates a multiple edge with probability at most
\[
\frac{k-1}{\sum_{i=\lfloor n^{\beta}\rfloor+1}^n \hat{D}^n_i}.
\]
This vertex therefore possesses an anomalous edge with probability bounded above by
\[
\frac{d(d-1)}{\sum_{i=\lfloor n^{\beta}\rfloor+1}^n \hat{D}^n_i}.
\]
Hence, by the conditional version of Markov's inequality,
\[
\Prob{A^n(\fl{n^{\beta}}) > 0 \ \Big| \ \hat{D}_1^n, \hat{D}_2^n, \ldots, \hat{D}_n^n} \le \left( \frac{\sum_{i=1}^{\fl{n^{\beta}}} (\hat{D}_i^n)^2}{\sum_{i=\lfloor n^{\beta}\rfloor+1}^n \hat{D}^n_i}\right) \wedge 1.
\]
But $\sum_{i=\lfloor n^{\beta}\rfloor+1}^n \hat{D}^n_i= \sum_{i=1}^n \hat{D}^n_i-\sum_{i=1}^{\lfloor n^{\beta}\rfloor} \hat{D}^n_i \ge \sum_{i=1}^n \hat{D}^n_i-\sum_{i=1}^{\lfloor n^{\beta}\rfloor} (\hat{D}^n_i)^2$. By Lemmas~\ref{lem:sumofsquares} and \ref{lem:lln} and the bounded convergence theorem, we obtain that
\[
\Prob{A^n(\fl{n^{\beta}}) > 0} \to 0
\]
as $n \to \infty$.
\end{proof}

Let $\rho(n) = \inf\{k \ge 0: A^n(k) > 0\}$ and note that the event that $\mathbf{M}_n(\nu)$ is simple is equal to $\{\rho(n) = \infty\}$.  The last proposition shows that we observe any anomalous edges long after the timescale in which we explore the largest components of the graphs.  This allows us to conclude that all of the results we prove using only the timescale $n^{\alpha/(\alpha+1)}$ for the multigraph are also true conditionally on $\{\rho(n) = \infty\}$.  In this way, we may give a proof of Conjecture 8.6 of Joseph~\cite{Joseph}.  (See also Theorem 3 of \cite{DharavdHvLSen2}.)

\begin{thm}
Conditional on $\{\rho(n) = \infty\}$, as $n \to \infty$,
\[
\left(n^{-1/(\alpha+1)} R^n(\fl{sn^{\alpha/(\alpha+1)}}, n^{-1/(\alpha+1)} C^n(\fl{sn^{\alpha/(\alpha+1)}}, s \ge 0 \right) \convdist \left(R_s, -\tfrac{1}{2} \widetilde{I}_s, s \ge 0 \right),
\]
in $\mathbb{D}(\R_+, \R)^2$.
\end{thm}

\begin{proof}
Given Propositions~\ref{prop:Xconv} and \ref{prop:backedgesarelate}, this follows in exactly the manner as Joseph's Theorem 3.2 follows from his Theorem 3.1.
\end{proof}

Henceforth, using exactly the same argument, statements about our processes should be understood to hold either unconditionally or conditionally on the event $\{\rho(n) = \infty\}$.

We now turn to the locations of the back-edges that do occur.  Recall that if a back-edge occurs at step $k$, then $U^n(k)$ is its index in the stack.  For steps $k$ such that $N^n(k) - N^n(k-1) = 0$, declare $U^n(k) = \partial$, where $\partial$ denotes that no mark occurs. 

For $s \ge 0$, let $(N_s, s \ge 0)$ be a Cox process driven by $(R_s/\mu, s \ge 0)$, and define marks
\[
U_s  \begin{cases}
= \partial & \text{if $N_s - N_{s-} = 0$} \\
\sim \mathrm{U}[0, R_s] & \text{if $N_s - N_{s-} = 1$}.
\end{cases}
\]
Then, conditionally on $(R_s, s \ge 0)$, we have that
\[
\sum_{\substack{s \ge 0 :\\N_s - N_{s-} = 1}} \delta_{(s, U_s)}
\]
is a Poisson point process on $\{(s,x) \in \R_+ \times \R_+: x \le R_s\}$ of constant intensity $1/\mu$.

\begin{thm} \label{thm:Coxprocess}
We have jointly
\begin{align*}
& \left( n^{-1/(\alpha+1)} R^n(\fl{sn^{\frac{\alpha}{\alpha+1}}}), n^{-1/(\alpha+1)} C^n(\fl{sn^{\alpha/(\alpha+1)}}, N^n(\fl{sn^{\frac{\alpha}{\alpha+1}}}), s \ge 0 \right) \\
& \qquad \convdist (R_s, -\tfrac{1}{2} \widetilde{I}_s, N_s, s \ge 0),
\end{align*}
in $\mathbb{D}(\R_+, \R)^3$ and
\[
\sum_{\substack{k \ge 1:\\N^n(k) - N^n(k-1) = 1}} \delta_{\left(n^{-\frac{\alpha}{\alpha+1}}k, n^{-\frac{1}{\alpha+1}} U^n(k)\right)} \convdist \sum_{\substack{s \ge 0 :\\N_s - N_{s-} = 1}} \delta_{(s, U_s)},
\]
for the topology of vague convergence for counting measures on $\R_+^2$.
\end{thm}

We observe that, in particular, for fixed $t \ge 0$ we have $\sup_{0 \le s \le t} R_s < \infty$ a.s.\ and so $N_t < \infty$ a.s.

\begin{proof}
We refine the argument from the proof of Lemma~\ref{lem:back}. We will find it convenient to build all of our random variables and stochastic processes, for all $n \ge 1$, on the same probability space.  By Proposition~\ref{prop:Xconv} and Skorokhod's representation theorem, we may assume that this probability space is such that the convergence
\begin{equation} \label{eqn:asc}
n^{-1/(\alpha+1)} (R^n(\fl{s n^{\frac{\alpha}{\alpha+1}}}), s \ge 0) \to (R_s, s \ge 0)
\end{equation}
occurs almost surely.  Since $R^n$ is a measurable function of $(\hat{D}_i)_{1\leq i\leq n}$ for each $n \ge 1$, we have that $R$ is a measurable function of the collection of random variables $\{\hat{D}_i^m: 1 \le i \le m, m \ge 1\}$.  For definiteness, let us assume that all \emph{additional} randomness needed to generate the processes $X^n, N^n, \mathcal{M}^n$ is independent between different values of $n$.

Let $\mathcal{F}^n_k = \sigma(\{\hat{D}_i^m: 1 \le i \le m, m \ge 1\}, X^n(i), N^n(i), \mathcal{M}^n(i), 0 \le i \le k)$ and so, in particular, $\mathcal{F}_0^n$ contains $\sigma(\{\hat{D}_i^m: 1 \le i \le m, m \ge 1\})$ for every $n \ge 1$. (Recall that in the construction of $\mathbf{M}_n(\nu)$ at the beginning of Section~\ref{subsec:explo}, we condition on $(\hat{D}_1^n, \hat{D}_2^n, \ldots, \hat{D}_n^n)$ \emph{before} doing the exploration, so that having this information available at time 0 is consistent with our construction.)

At step $k+1$, conditionally on $\mathcal{F}_{k}^n$, the probability of seeing a back-edge is
\[
\frac{R^n(k)- |\mathcal{M}^n(k)|}{R^n(k) - |\mathcal{M}^n(k)|+ \sum_{j = \tau_n(k)+1}^{n} \hat{D}_j^n},
\]
where $|\mathcal{M}^n(k)| \le N^n(k)$. Now
\[
\sum_{i=k+1}^n \hat{D}_i^n \le R^n(k) - |\mathcal{M}^n(k)| + \sum_{i=\tau^n(k) + 1}^{n} \hat{D}_i^n  \le \sum_{i=1}^n \hat{D}_i^n + R^n(k).
\]
But then by Lemma~\ref{lem:back}, the almost sure convergence (\ref{eqn:asc}) and Lemma~\ref{lem:lln} (in the \hyperref[appn]{Appendix}), we obtain
\begin{align}
& n^{\alpha/(\alpha+1)} \left(\frac{R^n(\fl{sn^{\alpha/(\alpha+1)}})- |\mathcal{M}^n(\fl{sn^{\alpha/(\alpha+1)}})|}{R^n(\fl{sn^{\alpha/(\alpha+1)}}) - |\mathcal{M}^n(\fl{sn^{\alpha/(\alpha+1)}})|+ \sum_{j = \tau_n(\fl{sn^{\alpha/(\alpha+1)}})+1}^{n} \hat{D}_j^n}, 0 \le s \le t \right) \notag \\
& \qquad \convprob \frac{1}{\mu} (R_s, 0 \le s \le t)  \label{eqn:intensity}
\end{align}
in $\mathbb{D}([0,t], \R_+)$.  Then $(N^n(k), k \ge 0)$ is a counting process with compensator
\[
N_{\mathrm{comp}}^n(k) = \sum_{j=1}^{k-1}\frac{R^n(j)- |\mathcal{M}^n(j)|}{R^n(j) - |\mathcal{M}^n(j)|+ \sum_{i = \tau_n(j)+1}^{n} \hat{D}_i^n} \I{X^n(j) \notin \mathcal{M}^n(j)} .
\]
Since
\begin{align*}
& \sum_{j=1}^{k-1}\frac{R^n(j)- |\mathcal{M}^n(j)|}{R^n(j) - |\mathcal{M}^n(j)|+ \sum_{i = \tau_n(j)+1}^{n} \hat{D}_i^n} \I{X^n(j) \in \mathcal{M}^n(j)} \\
& \qquad \le N^n(k-1) \max_{0 \le j \le k-1} \frac{R^n(j)- |\mathcal{M}^n(j)|}{R^n(j) - |\mathcal{M}^n(j)|+ \sum_{i = \tau_n(j)+1}^{n} \hat{D}_i^n}
\end{align*}
and $n^{-\alpha/(\alpha+1)} N^n(\fl{t n^{\alpha/(\alpha+1)}}) \convprob 0$ by Lemma~\ref{lem:back}, using (\ref{eqn:intensity}) we get that
\[
E^n(\fl{tn^{\alpha/(\alpha+1)}}) := \sum_{j=0}^{\fl{tn^{\alpha/(\alpha+1)}}-1}\frac{R^n(j)- |\mathcal{M}^n(j)|}{R^n(j) - |\mathcal{M}^n(j)|+ \sum_{i = \tau_n(j)+1}^{n} \hat{D}_i^n} \I{X^n(j) \in \mathcal{M}^n(j)}  \convprob 0.
\]
So by (\ref{eqn:intensity}), we obtain that for each $t \ge 0$,
\begin{align*}
N^n_{\mathrm{comp}}(\fl{t n^{\alpha/(\alpha+1)}}) 
& =  \sum_{j=1}^{k-1}\frac{R^n(j)- |\mathcal{M}^n(j)|}{R^n(j) - |\mathcal{M}^n(j)|+ \sum_{i = \tau_n(j)+1}^{n} \hat{D}_i^n} - E^n(\fl{tn^{\alpha/(\alpha+1)}}) \\
& \convprob \frac{1}{\mu} \int_0^t R_s ds.
\end{align*}
Finally, applying Theorem 14.2.VIII of Daley and Vere-Jones~\cite{DaleyVere-Jones} yields that
\[
 \left(N^n(\fl{sn^{\alpha/(\alpha+1)}}), 0 \le s \le t \right) \convdist (N_s, 0 \le s \le t).
\]
The marks are uniform on the vertices of the stack which do not already carry marks, and so it is straightforward to see that they must be uniform in the limit.
\end{proof}

\subsection{Components of the finite graph}\label{subsec:cmpts}

We now turn to the consideration of the individual components of the multigraph. Let $\sigma^n(0) = 0$ and for $k \ge 1$, write
\[
\sigma^n(k) = \inf\{j \ge 0: C^n(j) \ge k\}.
\]
This is the time at which we finish exploring the $k$th component of the forest $\mathbf{F}_n(\nu)$.  Let
\[
\zeta^n(k) = \sigma^n(k) - \sigma^n(k-1),
\]
the corresponding length of the excursion, which is equal to the total number vertices within the component, since precisely one of these is killed at each step. But then $\zeta^n(k)$ is also equal to the number of vertices in the corresponding component of $\mathbf{M}_n(\nu)$, plus twice the number of back-edges. Let
\[
\varepsilon^n_k(t) = n^{-1/(\alpha+1)} \left(X^n(\sigma^n(k-1) + \fl{t n^{\alpha/(\alpha+1)}}) - X^n(\sigma^n(k-1)) \right)
\]
for $0 \le t \le n^{-\alpha/(\alpha+1)} \zeta^n(k)$ be the $k$th rescaled excursion of $X^n$, with length $\zeta(\varepsilon_k^n)=n^{-\alpha/(\alpha+1)} \zeta^n(k)$ and rescaled left endpoint $g_k^n = n^{-\alpha/(\alpha+1)} \sigma^n(k-1)$.

Recall from Section~\ref{sec:stable} the notation $(\widetilde{\varepsilon}_i, i \ge 1)$ for the ordered excursions of $R$ above 0 and $\zeta(\widetilde{\varepsilon}_i) $ for the lifetime of $\widetilde{\varepsilon}_i$. Denote by $g_i$ the left endpoint of $\widetilde{\varepsilon}_i $.  Recall also that $\ell^2_{\downarrow} = \{(x_1, x_2, \ldots) \in \R^{\N}: x_1 \ge x_2 \ge \ldots \ge 0, \sum_{i \ge 1} x_i^2 < \infty\}$.  Let $\Gamma$ be a countable index set and write $\ell_+^2(\Gamma)$ for the set of non-negative sequences $(x_{\gamma}: \gamma \in \Gamma)$ such that $\sum_{\gamma \in \Gamma} x_{\gamma}^2 < \infty$.  Write $\text{ord}: \ell_+^2(\Gamma) \to \ell^2_{\downarrow}$ for the map which puts the elements of $(x_{\gamma}: \gamma \in \Gamma)$ into decreasing order. For a sequence $(\varepsilon_k, A_k)_{k\geq 1}$ where $\varepsilon_k$ is an excursion of length $\zeta(\varepsilon_k)$ and $A_k$ is some other random variable, write
\[
\text{ord}\left(\zeta(\varepsilon_k), A_k, k \ge 1\right)
\]
for the same sequence put in decreasing order of $\zeta(\varepsilon_k)$. 

This section is devoted to proving the following proposition.

\begin{prop} \label{prop:ordered} We have $(\zeta(\widetilde{\varepsilon}_i), i \ge 1) \in \ell^2_{\downarrow}$,
\[
 \mathrm{ord} \left( \zeta(\varepsilon^n_{k}), k \ge 1 \right) \convdist  (\zeta(\widetilde{\varepsilon}_i), i \ge 1)
\]
as $n \to \infty$ in $\ell^2_{\downarrow}$, and
\[
 \mathrm{ord} \left( \zeta(\varepsilon^n_{k}),  g_k^n, k\ge 1\right) \convdist  (\zeta(\widetilde{\varepsilon}_i), g_i, i \ge 1),
\]
where the convergence is in $\ell^2_{\downarrow}$ for the first coordinate and in the product topology for the second.
\end{prop}

We apply a method outlined in Proposition 15 of Aldous~\cite{AldousCritRG}, which is most conveniently recounted in Section 2.6 of Aldous and Limic~\cite{AldousLimic}.  This is very similar to Theorem 8.3 of Joseph~\cite{Joseph}, who omits many of the details of the proof. We feel that the argument is sufficiently subtle to merit a full account, which we now give.

Essentially, there are two steps to proving the desired convergence. First, we need to show that the longest excursions of $R^n$ and $R$ occurring before some finite time match up for large enough $n$.  Then we need to show that long excursions of $R^n$ cannot ``wander off to time $\infty$''. Proposition~\ref{prop:pointproc} below is designed to deal with these issues.  

Following Aldous, we introduce the concept of a size-biased point process. Suppose we have random variables $\mathbf{Y} = (Y_{\gamma}: \gamma \in \Gamma)$ in $\ell_+^2(\Gamma)$. Given $\mathbf{Y}$, let $E_{\gamma} \sim \text{Exp}(Y_{\gamma})$ independently for different $\gamma \in \Gamma$.  Set
\begin{equation} \label{eqn:sigma}
\Sigma(a) = \sum_{\gamma \in \Gamma} Y_{\gamma} \I{E_{\gamma} < a}
\end{equation}
and note that $\Sigma(a) < \infty$ a.s.  Let $\Sigma_{\gamma} = \Sigma(E_{\gamma})$. Then $\Xi = \{(\Sigma_{\gamma}, Y_{\gamma}): \gamma \in \Gamma\}$ is the \emph{size-biased point process} (SBPP) associated with $\mathbf{Y}$. Write $\pi$ for the projection onto the second co-ordinate, so that $\pi(\{(s_{\gamma}, y_{\gamma})\}) = \{y_{\gamma}\}$.

\begin{prop}[Proposition 15 of \cite{AldousCritRG} and Proposition 17 of \cite{AldousLimic}] \label{prop:pointproc} 
Let $\mathbf{Y}^n \in \ell_+^2(\Gamma^n)$ for each $n > 1$, let $\Sigma^n$ be the analogue of (\ref{eqn:sigma}) and let $\Xi^n$ be the associated SBPP. Suppose that
\[
\Xi^n \convdist \Xi^{\infty},
\]
as $n \to \infty$, for the topology of vague convergence of counting measures on $[0,\infty) \times (0,\infty)$, where $\Xi^{\infty}$ is some point process satisfying
\begin{enumerate}
\item $\sup\{s: (s,y) \in \Xi^{\infty} \text{ for some $y$}\} = \infty$ a.s.
\item if $(s,y) \in \Xi^{\infty}$ then $\sum_{(s',y') \in \Xi^{\infty}: s' < s} y' = s$ a.s.
\item $\max\{y: (s,y) \in \Xi^{\infty} \text{ for some $s > t$}\} \convprob 0$ as $t \to \infty$.
\end{enumerate}
Then $\mathbf{Y}^{\infty} := \mathrm{ord}(\pi(\Xi^{\infty})) \in \ell_{\downarrow}^2$ and $\mathrm{ord}(\mathbf{Y}^n) \convdist \text{ord}(\mathbf{Y}^{\infty})$ in $\ell^2_{\downarrow}$. In addition, 
\[
\mathrm{ord}(Y^n_{\gamma}, \Sigma^n_{\gamma}, \gamma\in \Gamma^n) \convdist \mathrm{ord}(Y_{\gamma},\Sigma_{\gamma},  \gamma\in \Gamma)
\]
as $n \to \infty$, where the convergence is in $\ell^2_{\downarrow}$ for the first coordinate and in the product topology for the second.
\end{prop}

The original statement does not mention the last convergence. It is, in fact, implicitly contained in the proof of Proposition 15 given in \cite{AldousCritRG}, more precisely in the assertion that the tightness of the sequence $(\Sigma^n(a))_{n\geq 1}$ for arbitrary $a>0$ and the convergence $\Xi^n \convdist \Xi^{\infty}$ together imply the convergence $\mathrm{ord}(\mathbf{Y}^n) \convdist \text{ord}(\mathbf{Y}^{\infty})$ in $\ell^2$.

For $k \ge 1$, we let
\[
Y^n_k = n^{-\alpha/(\alpha+1)} \left[\zeta^n(k) - (N^n(\sigma^n(k)) - N^n(\sigma^n(k-1)) +1)\right].
\]

Recall that at the end of the exploration of the $(k-1)$th component of $\mathbf{M}_n(\nu)$ in depth-first order, we choose a new vertex from the unexplored parts of the graph with probability proportional to its degree. So we pick a \emph{component} with probability proportional to the sum of its degrees, which is twice the number of its edges. Since the number of steps it takes to explore a component of $\mathbf{M}_n(\nu)$ is the number of its vertices (which is the number of its non-back-edges plus one) plus twice the number of back-edges (unlike the non-back-edges, a back-edge takes one step to create and another step to close), it follows that $Y^n_k$ is the number of edges of the $k$th component of $\mathbf{M}_n(\nu)$ times $n^{-\alpha/(\alpha+1)}$. 

For $k\geq 1$, let 
\[
\Sigma_k^n = \sum_{i=1}^{k-1} Y^n_i =   n^{-\alpha/(\alpha+1)} \left[\sigma^n(k-1)-N^n(\sigma^n(k-1))-k+1\right],
\]
and put
\[
\Xi^n = \{(\Sigma^n_k, Y_k^n): k \ge 1\}.
\]
It is easy to see that $\Xi^n$ then has the same distribution as the SBPP associated with 
\\
$n^{-\alpha/(\alpha+1)}(\zeta^n(k) -N^n(\sigma^n(k)) + N^n(\sigma^n(k-1))-1, k \ge 1)$.

Recall from Section~\ref{sec:stable} the notation $(\widetilde{\varepsilon}^{(\ell)}, \ell \ge 0)$ for the excursions of the reflected process $(R_t, t \ge 0)$ indexed by local time $\ell$, and $(\widetilde{\sigma}_{\ell}, \ell \ge 0)$ for the inverse local time process. Let $\Xi^{\infty}$ be the point process given by
\[
\Xi^{\infty} = \{(\widetilde{\sigma}_{\ell^-}, \zeta(\widetilde{\varepsilon}^{(\ell)})): \ell \ge 0, \widetilde{\sigma}_{\ell} - \widetilde{\sigma}_{\ell-} > 0\}.
\]
By Proposition~\ref{prop:transient} and Lemma~\ref{lem:Josephprops}, properties (1), (2) and (3) from Proposition~\ref{prop:pointproc} above hold for $\Xi^{\infty}$. In order to apply Proposition~\ref{prop:pointproc}, it thus remains to establish the convergence of $\Xi^n$ to $\Xi^{\infty}$.  We do this by first proving a deterministic result for a suitable class of functions, extending Lemma $7$ of Aldous~\cite{AldousCritRG} from the setting of continuous functions to the setting of c\`adl\`ag functions satisfying certain conditions.

For a c\`adl\`ag function $f: [0,\infty) \to \R$ with only positive jumps, let $E(f)$ be the set of non-empty intervals $e=(l,r)$ such that $f(l)=\inf_{s\leq l} f(s) = f(r)$ and $f(s) >f(l)$ for all $s\in (l,r)$. We say that such intervals are excursions of $f$.
Let $\mathcal{S}$ denote the set of functions $f: [0,\infty) \to \R$ satisfying the following conditions:
\begin{enumerate}
\item $f$ is c\`adl\`ag and has only non-negative jumps.
\item $f(x) \to -\infty$ as $x \to \infty$.
\item If $0 \le a < b$ and $f(b-) = \inf_{a \le s \le b} f(s)$ then $f(b) = f(b-)$. 
\item For each $\epsilon > 0$, $E(f)$ contains only finitely many excursions of length greater than or equal to $\epsilon$.
\item The complement of $\cup_{(l,r) \in E(f)} (l,r)$ has Lebesgue measure 0. 
\item If $(l_1,r_1), (l_2,r_2) \in E(f)$ and $l_1<l_2$, then $f(l_1)>f(l_2)$.
\end{enumerate}
\begin{lem}\label{lem:Aldous7}
Let $f \in \mathcal{S}$ and let $(f_n)_{n\geq 1}$ be a sequence of c\`{a}dl\`{a}g functions such that 
$f_n \to f$ as $n \to \infty$ in the Skorokhod sense. For each $n\in \N$, let $(t_{n,i})_{i\geq 1}$ be a strictly increasing sequence such that  
\begin{enumerate}
\item[(i)] $t_{n,1}=0$ and $\lim_{i\rightarrow \infty} t_{n,i} = \infty$, 

\item[(ii)] $f_n(t_{n,i})=\inf_{s \leq t_{n,i}}f_n(s)$,

\item[(iii)] for each $s < \infty$, $\lim_{n\rightarrow \infty} \max_{i : t_{n,i}\leq s}(f_n(t_{n,i})-f_n(t_{n,i+1}))=0.$
\end{enumerate}
Write $\Xi=\{(l,r-l): (l,r) \in E(f)\}$ and $\Xi_n=\{(t_{n,i},t_{n,i+1}-t_{n,i}): i \ge 1\}$ for $n \ge 1$. Then 
\[
\Xi_n \to \Xi
\]
as $n \to \infty$, where the convergence holds in the topology of vague convergence of counting measures on $[0, \infty) \times (0, \infty)$.
\end{lem}

\begin{proof}
We adapt the proof of Lemma 4.8 of Martin \& R\'ath~\cite{MartinRath}.  Suppose that $(l,r)$ is an excursion of $f$.  Fix $\epsilon \in (0,r-l)$.  Since $f \in \mathcal{S}$, there exists $\delta > 0$ such that
\begin{itemize}
\item[] $f(x) \ge f(l) + \delta$ for all $x \in [0, l - \epsilon/2]$
\item[] $f(x) \ge f(l) + \delta$ for all $x \in [l+\epsilon/2, r - \epsilon/2]$
\item[] $f(x) \le f(l) - \delta$ for some $x \in (r, r+\epsilon/2]$.
\end{itemize}

The first line is a consequence of conditions (1) and (6) in the definition of the set $\mathcal{S}$: if for every $n>0$, there exists $x_n \in [0, l - \epsilon/2]$ such that $f(x_n) < f(l) + 1/n$, then by condition (1) there would be an accumulation point $x_{\infty} \in [0, l - \epsilon/2]$ of the sequence $(x_n)_{n\ge 1}$ such that $f(x_{\infty})\leq f(l)$, and hence there would exist an interval $(l',r')$ with $r'<r$ such that $f(r')\le f(x_{\infty})\le f(r)$. But this contradicts (6).

The second line follows from a similar argument: there cannot be an accumulation point $x_{\infty} \in [l+\epsilon/2, r - \epsilon/2]$ such that $f(x_{\infty})\leq f(l)$ since $(l,r)$ is an excursion interval.  The third line is again a consequence of condition (6) in the definition of the set $\mathcal{S}$.

Since we have $f_n \to f$ in the Skorokhod sense, there exist $n_0$ and a sequence of continuous strictly increasing functions $\lambda_n: [0,\infty) \to [0, \infty)$ such that $\lambda_n(0) = 0$, $\lim_{t \to \infty} \lambda_n(t) = \infty$ and for all $n \ge n_0$,
\[
|f_n(\lambda_n(x)) - f(x)| < \delta/2 \quad \text{for all $x \in [0,r+\epsilon/2]$}
\]
and
\[
|\lambda_n(x) - x| < \epsilon/2 \quad \text{for all $x \in [0,r+\epsilon/2]$.}
\]
Then for $n \ge n_0$,
\begin{itemize}
\item[] $f_n(\lambda_n(x)) \ge f(l) + \delta/2$ for all $x \in [0,l - \epsilon/2]$
\item[] $f_n(\lambda_n(x)) > f(l) - \delta/2$ for all $x \in [l-\epsilon/2,l+\epsilon/2]$
\item[] $f_n(\lambda_n(l)) < f(l) + \delta/2$
\item[] $f_n(\lambda_n(x)) \ge f(l) + \delta/2$ for all $x \in [l + \epsilon/2, r - \epsilon/2]$
\item[] $f_n(\lambda_n(x)) \le f(l) - \delta/2$ for some $x \in [r,r+\epsilon/2]$.
\end{itemize}
For the second point, note that $f_n(\lambda_n(x)) > f(x)-\delta/2\geq f(r)-\delta/2=f(l)-\delta/2$ for every $l+\epsilon/2\le x <r$. Therefore,
\begin{itemize}
\item[] $f_n(x) \ge f(l) + \delta/2$ for all $x \in [0,l - \epsilon]$
\item[] $f_n(x) > f(l) - \delta/2$ for all $x \in [l-\epsilon,l+\epsilon]$
\item[] $f_n(x) < f(l) + \delta/2$ for some $x \in [l-\epsilon/2,l+\epsilon/2]$.
\item[] $f_n(x) \ge f(l) + \delta/2$ for all $x \in [l + \epsilon, r - \epsilon]$
\item[] $f_n(x) \le f(l) - \delta/2$ for some $x \in [r-\epsilon/2,r+\epsilon]$.
\end{itemize}
From the first and third inequalities, the set $I:=\{x\in [l-\epsilon, r+\epsilon], \,f_n(x)\leq \inf_{y\le l-\epsilon}f_n(y)\}$ is not empty. Let $l^{(n)}:=\inf I$ and let $(x_k)_{k\geq 1}$ be a decreasing sequence in $I$ converging to $l^{(n)}$. Since $f_n$ is right-continuous, $f_n(l^{(n)})=\lim_{k\to \infty}f_n(x_k)$, and $f_n(l^{(n)})\leq \inf_{x\le l-\epsilon}f_n(x)$, so that $l^{(n)}=\min I$. Clearly,
\[
f_n(l^{(n)}) = \inf_{x \le l^{(n)}} f_n(x) 
\]
Similarly, from the first, second, fourth and fifth inequalities, 
we obtain the existence of $r^{(n)} \in [r-\epsilon,r+\epsilon]$ such that
\[
f_n(r^{(n)}) = \inf_{x \le r^{(n)}} f_n(x).
\]
Now fix $\eta > 0$.  We can find $t > 0$ such that there are no excursions of length exceeding $\eta$ in $E(f)$ which intersect $[t,\infty)$.  Then there exists a finite collection $\{(l_i,r_i): 1 \le i \le m\}$ of excursions in $E(f)$ with $l_i \le t+\eta$ for $1 \le i \le m$ and such that $\cup_{1 \le i \le m} (l_i,r_i)$ covers all of $[0,t+\eta]$ except for a set of Lebesgue measure at most $\eta/2$.  Set $\epsilon = \eta/4m$ and apply the above argument for each excursion, to see that for $n$ sufficiently large, there exist disjoint intervals $(l_1^{(n)}, r_1^{(n)}), \ldots, (l_m^{(n)}, r_m^{(n)})$ such that
\[
|l_i^{(n)} - l_i| < \eta/4m \quad \text{and} \quad |r_i^{(n)} - r_i| < \eta/4m \quad \text{for $1 \le i \le m$}.
\]
But then the remaining length in $[0,t+\eta]$ is at most $\eta$ and so, in particular, we must have captured all possible intervals $(t_{n,i}, t_{n,i+1})$ with $t_{n,i} \le t+\eta$ and $t_{n,i+1}-t_{n,i} \ge \eta$, up to an error of at most $\eta/4m$ at each end-point.  The required vague convergence follows straightforwardly.
\end{proof}

\begin{lem} \label{lem:pointprocexc}
We have $\widetilde{L} \in \mathcal{S}$ almost surely.  Moreover,
\[
\Xi^n \convdist \Xi^{\infty}
\]
as $n \to \infty$, where the convergence holds in the topology of vague convergence of counting measures on $[0, \infty) \times (0, \infty)$.
\end{lem}

\begin{proof}
$\widetilde{L}$ is clearly c\`adl\`ag with only non-negative jumps almost surely.  The other  conditions required for a function to lie in $\mathcal{S}$ follow from Proposition~\ref{prop:transient} and Lemma~\ref{lem:Josephprops}.

Now let $f=\widetilde{L}$ and $f_n=n^{-1/(\alpha +1)}X^n(\lfloor n^{\alpha/(\alpha +1)}\cdot\rfloor)$. It is clear that $\Xi^{\infty}$ is $\Xi$ of the previous lemma for this $f$.  For $n \ge 1$, $i \ge 0$, let $t_{n,i} = n^{-\alpha/(\alpha+1)}\sigma^n(i)$.  Then $t_{n,0} = 0$ and $\lim_{i \to \infty} t_{n,i} = \infty$.  By construction, the $t_{n,i}$ are times at which new infima of $f_n$ are reached.  Moreover, 
\[
f_n(t_{n,i}) - f_n(t_{n,i+1}) = n^{-1/(\alpha +1)}(X^n(\sigma^n(i)) - X^n(\sigma^n(i+1))) = n^{-1/(\alpha +1)}.
\]
Hence, the $(t_{n,i})_{i \ge 1}$ satisfy the conditions in Lemma~\ref{lem:Aldous7}.  It follows that
\[
\Xi_n = \{(t_{n,i}, t_{n,i+1}-t_{n,i}), i \ge 1\} \convdist \Xi
\]
as $n \to \infty$.  Now we have 
\[
\Sigma_i^n = t_{n,i} - n^{-\alpha/(\alpha+1)} [N^n(\sigma^n(i)) + i]
\]
and 
\[
Y_i^n = t_{n,i}-t_{n,i-1} -n^{-\alpha/(\alpha+1)}[N^n(\sigma^n(i)) - N^n(\sigma^n(i-1)) +1].
\]
Since $n^{-\alpha/(\alpha+1)} N^n(\sigma^n(i)) \to 0$ for each $i \ge 0$, it is straightforward to see that $\Xi^n$ and $\{(t_{n,i}, t_{n,i+1}-t_{n,i}), i \ge 1\}$ can be made arbitrarily close in the vague topology by taking $n$ large. Hence, $\Xi^n \convdist \Xi^{\infty}$ as desired.
\end{proof}

Proposition~\ref{prop:pointproc} tells us that we may now extract the ordered excursion lengths, and that we can add the convergence of the starting points of the excursions. This completes the proof of Proposition~\ref{prop:ordered}. As an aside, we observe that Proposition~\ref{prop:ordered} gives us an analogue of Corollary 16 of \cite{AldousCritRG}, as follows.

\begin{cor}
The point process $\Xi^{\infty} = \{(\widetilde{\sigma}^{(\ell)},\zeta(\widetilde{\varepsilon}^{(\ell)})): \ell \ge 0, \widetilde{\sigma}_{\ell} - \widetilde{\sigma}_{\ell-} > 0\}$ consisting of the left endpoints and the lengths of the excursions of $R$ is distributed as the SBPP associated with the set of excursion lengths of $R$. 
\end{cor}

\subsection{Marked excursions}\label{subsec:markedexc}

We now strengthen the convergence in Proposition~\ref{prop:ordered} to a convergence of ordered marked excursions.  Let us first prove a deterministic analytic result, similar in spirit to Lemma~\ref{lem:Aldous7}, which we shall use to handle the positions of the back-edges (recall that when a back-edge is discovered, its other endpoint is explored at the first time when the corresponding mark in the stack reaches the top of the stack).

\begin{lem} \label{lem:beposition}
Let $f \in \mathcal{S}$ and let $(f_n)_{n \ge 1}$ be a sequence of c\`adl\`ag functions such that $f_n \to f$ as $n \to \infty$ in the Skorokhod sense.  For each $n \in \N$, let $(s_n, y_n)$ be a pair of points such that $s_n \ge 0$ and $y_n \le f_n(s_n)$.  Let $(s,y)$ be such that $s > 0$, $f(s-) = f(s)$ and $0 < y < f(s)$.  Let $t_n = \inf\{u \ge s_n : f_n(u) \le y_n\}$ and $t = \inf\{u \ge s: f(u) \le y\}$.  Suppose that $(s_n,y_n) \to (s,y)$, that $t$ is not a local minimum of $f$ and that $f(t-) = f(t)$.  Then $t_n \to t$ as $n \to \infty$.
\end{lem}

\begin{proof}
Fix $0 < \epsilon < t-s$.  Since $f \in \mathcal{S}$, $y < f(s)$, $f(s-) = f(s)$, $f(t-) = f(t)$ and $t$ is not a local minimum of $f$, there exists $\delta > 0$ such that 
\begin{align*}
& f(x) \ge y + \delta \quad \text{for all $x \in [s-\epsilon/2, t - \epsilon/2]$}\\
& f(x) \le y - \delta \quad \text{for some $x \in (t, t+\epsilon/2]$}.
\end{align*}
As $f_n \to f$, there exist $n_0$ and a sequence of continuous strictly increasing functions $\lambda_n: [0,\infty) \to [0,\infty)$ such that $\lambda_n(0) = 0$, $\lim_{x \to \infty} \lambda_n(x) = \infty$ and, for all $n \ge n_0$,
\[
|f_n(\lambda_n(x)) - f(x)| < \delta/2 \quad \text{for all $x \in [0, t+\epsilon/2]$}
\]
and
\[
|\lambda_n(x) - x| < \epsilon/4 \quad \text{for all $x \in [0,t+\epsilon]$}.
\]
Then
\begin{align*}
& f_n(\lambda_n(x)) \ge y + \delta/2 \quad \text{for all $x \in [s-\epsilon/2,t-\epsilon/2]$} \\
& f_n(\lambda_n(x)) \le y - \delta/2 \quad \text{for some $x \in (t,t+\epsilon/2]$.}
\end{align*}
By taking $n_0$ larger if necessary, we also have
\[
|s - s_n| < \epsilon/4 \quad \text{and} \quad  |y-y_n| < \delta/2
\]
for all $n \ge n_0$.  Then for all $n \ge n_0$, we have $|\lambda_n^{-1}(s_n) - s| < \epsilon/2$ and so
\[
f_n(x) \ge y + \delta/2 \quad \text{for all $x \in [s_n, t-\epsilon]$}.
\]
It follows that
\[
f_n(x) > y_n \quad \text{for all $x \in [s_n, t-\epsilon]$}.
\]
Moreover, it must be the case that $f_n$ goes below $y_n$ in the time-interval $[\lambda_n(t - \epsilon/2), \lambda_n(t + \epsilon/2)]$, i.e.\ we must have $t_n \in [\lambda_n(t - \epsilon/2), \lambda_n(t + \epsilon/2)]$. But then $t_n \in [t-\epsilon, t + \epsilon]$ for all $n \ge n_0$.  As $\epsilon$ was arbitrary, the result follows.
\end{proof}

Let
\[
M^n(k) = N^{n}(\sigma^n(k)) - N^{n}(\sigma^n(k-1)),
\]
be the number of back-edges falling in the excursion $\varepsilon^n_{k}$. Suppose that $M^n(k) \ge 1$.  Then for $1 \le r \le M^n(k)$, let $g_k^n + s^{n}_{k,r}$ be the rescaled time at which the $r$th back-edge is discovered in the $k$th component and let $n^{1/(\alpha+1)} x^{n}_{k, r} \ge 0$ be its position on the stack.  Let $g_k^n + t^n_{k,r}$ be the rescaled time at which the corresponding marked leaf in the stack is killed, thus closing the $r$th back-edge. This is the first time that the stack size goes below the height of that leaf after its discovery (the stack being a LIFO queue), so that we have
\[
t^{n}_{k,r} = \inf\{t \ge s^{n}_{k,r}: \varepsilon^{n}_k(t) \le x^{n}_{k,r} - n^{-1/(\alpha+1)}\}.
\]
Finally, if $M^n(k) \ge 1$, let $\mathcal{P}^n_{k} = \sum_{r =1}^{M^n(k)} \delta_{(s^{n}_{k,r}, t^{n}_{k,r})}$ define a point measure on $[0,\zeta(\varepsilon^n_{k})]^2$.  If $M^n(k) = 0$, let $\mathcal{P}^n_{k}$ be the null measure.  Let
\[
\mathcal{Q}^n = \sum_{k\ge1} \sum_{r = 1}^{M^n(k)} \delta_{(g_k^n+s^n_{k,r}, g_k^n+t^n_{k,r})},
\]
the point measure encompassing all of the pairs of rescaled times at which a back-edge is opened and closed.

Turning now to the limiting process, recall that $(\widetilde{\varepsilon}_i, i \ge 1)$ are the excursions of $R$ listed in decreasing order of length, and that the sequence $(\zeta(\widetilde{\varepsilon}_i), i \ge 1)$ lies in $\ell_{\downarrow}^2$. Let $M_i$ be the number of marks falling in the excursion $\widetilde{\varepsilon}_i$, and if $M_i \ge 1$, write $s_{i,1}, \ldots, s_{i,M_i}$ for the times and $x_{i,1}, \ldots, x_{i,M_i}$ for the positions of the marks, respectively.  For $1 \le r \le M_i$, let
\[
t_{i,r} = \inf\{t \ge s_{i,r}: \widetilde{\varepsilon}_i(t) \le x_{i,r}\}.
\]
If $M_i \ge 1$, write $\mathcal{P}_i = \sum_{r = 1}^{M_i} \delta_{(s_{i,r}, t_{i,r})}$, and if $M_i = 0$, let $\mathcal{P}_i$ be the null measure.  Finally, let
\[
\mathcal{Q} = \sum_{i \ge 1} \sum_{r=1}^{M_i} \delta_{(g_i +s_{i,r}, g_i + t_{i,r})}.
\]

Recall that
\[
\widetilde{\sigma}^n(i) = \min\{k: \widetilde{S}^n(k) \le -i\}
\]
and let $\widetilde{\zeta}^n(i) = \widetilde{\sigma}^n(i) - \widetilde{\sigma}^n(i-1)$ be the length of the $i$th excursion of $\widetilde{S}^n$ above its running minimum.  Since the components of $\widetilde{\mathbf{F}}_n(\nu)$ again appear in size-biased order, an argument completely analogous to that above gives that $n^{-\alpha/(\alpha+1)} \mathrm{ord}(\widetilde{\zeta}^n(k), k \ge 1)$ converges in distribution to $(\zeta(\widetilde{\varepsilon}_i), i \ge 1)$ in $\ell^2_{\downarrow}$.  In particular,
\[
n^{-\alpha/(\alpha+1)} \max_{i \ge 1} \widetilde{\zeta}^n(i) \convdist \zeta(\widetilde{\varepsilon}_1),
\]
where we recall that $\zeta(\widetilde{\varepsilon}_1)$ is the length of the longest excursion of $R$ above 0; in particular, by Proposition~\ref{prop:transient} and Lemma~\ref{lem:Josephprops} we have $\zeta(\widetilde{\varepsilon}_1)< \infty$ a.s.  

For $i \ge 1$, let $\widetilde{M}^n(i)$ be the number of back-edges falling among the vertices corresponding to the $(2i-1)$th and $2i$th components of $\widetilde{\mathbf{F}}_n(\nu)$, and let $\widetilde{N}^n_{\max} = \max_{k \ge 1} \widetilde{M}^n(k)$.  On the pair formed of the $(2i-1)$th and $2i$th components of $\widetilde{\mathbf{F}}_n(\nu)$, at corresponding vertices, we have that the size of the stack in $\mathbf{F}_n(\nu)$ is bounded above by the size of the stack in $\widetilde{\mathbf{F}}_n(\nu)$ plus 1.

\begin{prop} \label{prop:orderedexmarks}
We have
\[
\mathrm{ord} \left(\zeta(\varepsilon^n_k), \varepsilon^n_{k},  \mathcal{P}^n_{k}, k \ge 1 \right) \convdist \left(\zeta(\widetilde{\varepsilon}_i), \widetilde{\varepsilon}_i, \mathcal{P}_i, i \ge 1\right)
\]
as $n \to \infty$. Here, for each $k \ge 1$, the convergence in the second co-ordinate is for the Skorokhod topology and in the third for the Hausdorff distance on $\R_+^2$; then we take the product topology over the different indices.
\end{prop}

\begin{proof}
Observe that $\{(s_{i,r}, x_{i,r}): 1 \le r \le M_i\}$ are the times and marks of the marked Cox process (see Theorem~\ref{thm:Coxprocess}) which fall into the excursion $\widetilde{\varepsilon}_i$, for each $i \ge 1$.  Equivalently, they are picked uniformly from the Lebesgue measure under the excursion $\widetilde{\varepsilon}_i$. Such an excursion has only countably many discontinuities (all of which are up-jumps), and so the conditions of Lemma~\ref{lem:beposition} are fulfilled almost surely.  Combining this with Theorem~\ref{thm:Coxprocess}, we may deduce the joint convergence in distribution as $n \to \infty$ of
\[
(n^{-1/(\alpha+1)} R^n(\fl{t n^{\alpha/(\alpha+1)}}), t \ge 0) \to (R_t, t \ge 0)
\]
in the Skorokhod topology and $\mathcal{Q}^n \to \mathcal{Q}$ in the topology of vague convergence on $\R_+^2$.
 
By Skorokhod's representation theorem, we may work on a probability space such that this joint convergence holds almost surely.  Fix $K \in \N$.  We have already shown that, given $\delta > 0$, there exist $M > 0$ and $n_0$ sufficiently large such that, with probability at least $1 - \delta$, the $K$ longest excursions of both $n^{-1/(\alpha+1)} R^n(\fl{\cdot \ n^{\alpha/(\alpha+1)}})$ and $R$ occur in the time-interval $[0,M]$ for all $n \ge n_0$.  Since $\mathcal{Q}^n \to \mathcal{Q}$ and $\mathcal{Q}$ has only finitely many points in $[0,M]^2$, we may deduce that the $K$ smaller point processes obtained by restricting $\mathcal{Q}^n$ to each of the $K$ longest excursion-intervals converge in the sense of the Hausdorff distance to $\mathcal{P}_1, \ldots, \mathcal{P}_K$ respectively. The result follows.
\end{proof}

We conclude this section with some technical bounds on the number of back-edges in a given component, of which we will make use later. 

\begin{lem} \label{lem:nmax}
Almost surely, we have $\#\{i \ge 1: M_i \ge 2\} < \infty$ and $N_{\max} := \sup_{i \ge 1} M_i < \infty$.
\end{lem}

\begin{proof}
We will bound
\[
\E{\#\{i \ge 1: M_i \ge 2\} \Big| (\zeta(\widetilde{\varepsilon}_i))_{i \ge 1}} = \sum_{i \ge 1} \Prob{M_i \ge 2 \Big| \zeta(\widetilde{\varepsilon}_i)},
\]
where we have used the independence of the excursions given their lengths.
We use the crude bound $\Prob{\mathrm{Po}(\lambda) \ge 2} \le \lambda^2$.  For any $i \ge 1$, we have
\begin{align*}
\Prob{M_i \ge 2 \Big| \zeta(\widetilde{\varepsilon}_i) = x} &= \Prob{\mathrm{Po} \left(\frac{1}{\mu} \int_0^x \widetilde{\varepsilon}_i(u) du \right) \ge 2 \ \Bigg| \ \zeta(\widetilde{\varepsilon}_i) = x} \\
&  \le \frac{\mathbb{E}_{\N^{(x)}}  \left[ \left(\frac{1}{\mu} \int_0^x \mathbbm{e}(u) du \right)^2 \exp \left(\frac{1}{\mu} \int_0^x \mathbbm{e} (u) du \right) \right]}{\mathbb{E}_{\N^{(x)}} \left[ \exp \left(\frac{1}{\mu} \int_0^x \mathbbm{e} (u) du \right) \right]} \\
& \le \mathbb{E}_{\N^{(x)}}  \left[ \left(\frac{1}{\mu} \int_0^x \mathbbm{e} (u) du \right)^4 \right]^{1/2} \mathbb{E}_{\N^{(x)}}  \left[\exp \left(\frac{2}{\mu} \int_0^x \mathbbm{e} (u) du \right) \right]^{1/2},
\end{align*}
by the Cauchy--Schwarz inequality and the fact that the denominator is bounded below by 1.

By the scaling property of the excursion $\mathbbm{e}$, 
\[
\mathbb{E}_{\N^{(x)}}  \left[ \left(\frac{1}{\mu} \int_0^x \mathbbm{e} (u) du \right)^4 \right]^{1/2}\le C x^{2(1+1/\alpha)}
\]
for some constant $C > 0$. Define
\[
f(x):=\mathbb{E}_{\N^{(x)}}  \left[\exp \left(\frac{2}{\mu} \int_0^x \mathbbm{e} (u) du \right) \right]^{1/2}= \mathbb{E}_{\N^{(1)}} \left[\exp \left(\frac{2x^{1+1/\alpha}}{\mu} \int_0^1 \mathbbm{e} (u) du \right) \right]^{1/2}.
\] 
This is clearly an increasing function of $x$, so that for $x\le \zeta(\widetilde{\varepsilon}_1)$ we have $f(x)\le f(\zeta(\widetilde{\varepsilon}_1))$, which is almost surely finite by Lemma~\ref{lem:expbound}, since $\zeta(\widetilde{\varepsilon_1}) < \infty$ a.s. By Proposition~\ref{prop:ordered}, we have $\sum_{i \ge 1} \zeta(\widetilde{\varepsilon}_i)^2 < \infty$ a.s.\ and so
\[
\E{\#\{i \ge 1: M_i \ge 2\} \Big| (\zeta(\widetilde{\varepsilon}_i))_{i \ge 1}} \le f(\zeta(\widetilde{\varepsilon}_1)) \sum_{i \ge 1} \zeta(\widetilde{\varepsilon}_i)^{2(1+1/\alpha)} < \infty \quad \text{a.s.}
\]
It follows that
\[
\#\{i \ge 1: M_i \ge 2\} < \infty \quad \text{a.s.}
\]
Since the area of any individual excursion $\widetilde{\varepsilon}_i$ is finite, it contains an almost surely finite number of points and this, together with the fact that only finitely many excursions contain more than 2 points, gives that $N_{\max} < \infty$ a.s.
\end{proof}

\subsection{Height process}\label{subsec:heightp}
In order to deal with the metric structure, we also need to know that the height process associated with $X^n$ converges. 
Let 
\[
H^n(k) = \#\left\{j \in \{0,1,\ldots,k-1\}: X^n(j) = \inf_{j \le \ell \le k} X^n(\ell) \right\}
\]
so that $H^n(k)$ is the distance from the root of the component being explored to the current vertex at step $k$.  See Figure~\ref{fig:components} for an illustration.  (Recall that $\mathbf{M}_n(\nu)$ is obtained from $\mathbf{F}_n(\nu)$ by replacing pairs of marked leaves by back-edges.)

\begin{figure}
\begin{center}
\includegraphics[width=6cm]{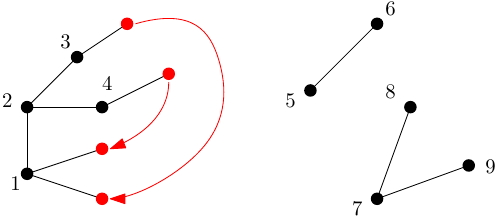} 
\vspace{5mm}

\includegraphics[width=14cm]{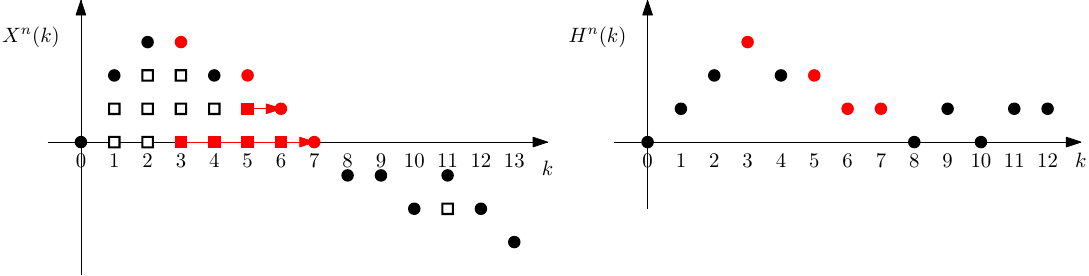}
\end{center}
\caption{Top: three components of $\mathbf{F}_n(\nu)$. Left: $X^n(k)$ drawn with the vertices on the stack indicated. Empty squares represent half-edges which are available to be connected to as back-edges. Filled squares are marked vertices. Right: the corresponding height process (with the extra vertices in red).}
\label{fig:components}
\end{figure}

Our aim in this section is to prove that $H^n$ has $\widetilde{H}$ as its scaling limit, and that we can extract its marked excursions as for the exploration process.  Recall that 
\[
\widetilde{G}^n(k) = \#\left\{j \in \{0,1,\ldots,k-1\}: \widetilde{S}^n(j) = \inf_{j \le \ell \le k} \widetilde{S}^n(\ell) \right\},
\]
which is the height process corresponding to the forest $\widetilde{\mathbf{F}}_n(\nu)$.  We will compare $H^n$ and $\widetilde{G}^n$.  It will be sufficient to do this for pairs of components of $\widetilde{\mathbf{F}}_n(\nu)$.  To this end, suppose that for $a \ge 0$ and $m \ge 1$, vertices $v_{a+1}, v_{a+2}, \ldots, v_{a+m}$ form a pair of components of $\widetilde{\mathbf{F}}_n(\nu)$, and that there are $b$ back-edges on these vertices in $\mathbf{M}_n(\nu)$.  Then the corresponding collection of components in $\mathbf{F}_n(\nu)$ together have $m+2b$ vertices, which are visited at times $c+a, c+a+1, \ldots, c+a+m+2b-1$ in the depth-first exploration, where $c \ge 0$ is such that $a = \tau_n(c+a)$.  We compare $H^n$ with $\widetilde{G}^n$ at the times $\tau_n(c+a), \tau_n(c+a+1),\ldots, \tau_n(c+a+m+2b-1)$.

\begin{lem} \label{lem:heightprocdiff}
We have
\[
\max_{0 \le i \le m+2b-1} |H^n(a+c+i) - \widetilde{G}^n(\tau_n(a+c+i)) | \le 1 + b + 2b \max_{1 \le i \le m} |\widetilde{G}^n(a+i) - \widetilde{G}^n(a+i-1)|
\]
\end{lem}

\begin{proof}
Suppose first that $b=0$.  Until we come to the end of the first component of $\widetilde{\mathbf{F}}_n(\nu)$, we have
\[
H^n(a+i) = \widetilde{G}^n(\tau_n(a+i)).
\]
Thereafter, we have $H^n(a+i) = 1 + \widetilde{G}^n(\tau_n(a+i))$, since the second component of $\widetilde{\mathbf{F}}_n(\nu)$ becomes a subtree attached to the root of the first component of $\mathbf{F}_n(\nu)$.

For $b > 0$, we encourage the reader to refer back to Figure~\ref{fig:unplugging}.
 
Write $\Delta = \max_{1 \le i \le m} |\widetilde{G}^n(a+i) - \widetilde{G}^n(a+i-1)|$. The occurrence of a head or tail of a back-edge at time $k$ implies $\tau_n(k+1) = \tau_n(k)$.  To get from $\widetilde{\mathbf{F}}_n(\nu)$ to $\mathbf{F}_n(\nu)$, we unplug a sequence of subtrees and plug them back in further along in the depth-first order. Within subtrees containing no back-edges, the \emph{increments} of the height process remain the same as they are in the corresponding subtrees of $\widetilde{\mathbf{F}}_n(\nu)$. Finally, every time we start a new subtree there is an extra difference of 1. It follows that the most by which the height process can be altered in going from $\widetilde{\mathbf{F}}_n(\nu)$ to $\mathbf{F}_n(\nu)$ is an additive factor of $1+b + 2b\Delta$.
\end{proof}

We are now ready to state and prove the main result of this section.  

\begin{prop} \label{prop:heightprocconv}
Jointly with the convergence in Theorem~\ref{thm:Coxprocess}, we have that as $n \to \infty$,
\[
\left( n^{-\frac{\alpha-1}{\alpha+1}}H^n(\fl{tn^{\alpha/(\alpha+1)}}), t \ge 0 \right) \convdist \left( \widetilde{H}_t, t \ge 0 \right)
\]
in $\mathbb{D}(\R_+,\R_+)$.
\end{prop} 

\begin{proof}
By Theorem~\ref{thm:jointdfwheight}, we have
\[
\left(n^{-\frac{\alpha-1}{\alpha+1}} \widetilde{G}^n( \fl{n^{\frac{\alpha}{\alpha+1}} u}), u \ge 0 \right) \convdist (\widetilde{H}_u, u \ge 0).
\]
By Theorem 1.4.3 of Duquesne and Le Gall~\cite{DuquesneLeGall}, $(H_u, u \ge 0)$ is almost surely continuous.  By the absolute continuity in Proposition~\ref{prop:stablemeasch}, the same is true of $(\widetilde{H}_u, 0 \le u \le t)$.  It follows that for any $t > 0$, 
\begin{equation} \label{eqn:heightcont}
n^{-\frac{\alpha-1}{\alpha+1}} \sup_{1 \le j \le \fl{t n^{\alpha/(\alpha+1)}}} |\widetilde{G}^n(j) - \widetilde{G}^n(j-1)| \convdist 0
\end{equation}
as $n \to \infty$.  

By Lemma~\ref{lem:nmax}, if $\delta > 0$, there exists $K_{\delta}  < \infty$ such that
\[
\Prob{\zeta(\widetilde{\varepsilon}_1) > K_{\delta}} < \delta \qquad \text{and} \qquad \Prob{N_{\max} > K_{\delta}} < \delta.
\]
In particular, starting from any time $k$, the number of steps until we next reach the beginning of an odd-numbered component is bounded above by $2K_{\delta} n^{\alpha/(\alpha+1)}$ with probability at least $1 - \delta+o(1)$.

Now fix $t  > 0$ and $\epsilon > 0$.  Then by Lemma~\ref{lem:heightprocdiff}, we have
\begin{align*}
& \Prob{n^{-\frac{\alpha-1}{\alpha+1}} \sup_{0 \le k \le \fl{t n^{\alpha/(\alpha+1)}}} |H^n(k) - \widetilde{G}^n(\tau_n(k))| > \epsilon}\\
& \qquad \le \Prob{n^{-\alpha/(\alpha+1)} \max_{i \ge 1} \widetilde{\zeta}^n(i) > K_{\delta}} 
+ \Prob{N^n_{\max} > K_{\delta}} \\
& \qquad + \Prob{1 + K_{\delta}+2K_{\delta} \sup_{1 \le j \le \fl{(t + 2K_{\delta}) n^{\alpha/(\alpha+1)}}} |\widetilde{G}^n(j) - \widetilde{G}^n(j-1)| > \epsilon n^{\frac{\alpha-1}{\alpha+1}}}.
\end{align*}
Using (\ref{eqn:heightcont}), we obtain that
\[
\Prob{1 + K_{\delta} + 2K_{\delta} \sup_{1 \le j \le \fl{(t + 2K_{\delta}) n^{\alpha/(\alpha+1)}}} |\widetilde{G}^n(j) - \widetilde{G}^n(j-1)| > \epsilon n^{\frac{\alpha-1}{\alpha+1}}} \to 0
\]
as $n \to \infty$.  It follows that
\[
\limsup_{n \to \infty} \Prob{n^{-\frac{\alpha-1}{\alpha+1}} \sup_{0 \le k \le \fl{t n^{\alpha/(\alpha+1)}}} |H^n(k) - \widetilde{G}^n(\tau_n(k))| > \epsilon}
 < 2 \delta.
\]
But $\delta > 0$ was arbitrary and so by Lemma~\ref{lem:back} we may deduce that
\[
\left( n^{-\frac{\alpha-1}{\alpha+1}}H^n(\fl{un^{\alpha/(\alpha+1)}}), 0 \le u \le t \right) \convdist \left( \widetilde{H}_u,  0 
\le u \le t \right). \qedhere
\]
\end{proof}

Now let
\[
h^n_{k}(t) =  n^{-(\alpha-1)/(\alpha+1)} H^n(\sigma^n(k-1) + \fl{t n^{\alpha/(\alpha+1)}}), \quad 0 \le t \le n^{-\alpha/(\alpha+1)} \zeta^n(k).
\]

\begin{prop} \label{prop:joint}
We have
\[
\mathrm{ord} \left(\varepsilon^n_{k}, h^n_{k},  \mathcal{P}^n_{k}, k \ge 1 \right) \convdist \left(\widetilde{\varepsilon}_i, \widetilde{h}_i, \mathcal{P}_i, i \ge 1\right)
\]
as $n \to \infty$. Here, for each $k \ge 1$, the convergence in the first co-ordinate is for the Skorokhod topology and in the second for the topology of vague convergence of counting measures on $[0,\infty)^2$; then we take the product topology over different $k$.
\end{prop}

\begin{proof}
We derive this from Proposition~\ref{prop:heightprocconv} by applying the same reasoning as in the proof of Proposition~\ref{prop:orderedexmarks}, using the fact that $R^n$ and $H^n$ have the same excursion-intervals.
\end{proof}

\subsection{The convergence of the metric structure}\label{subsec:metricstructure}

We have now assembled all of the ingredients needed in order to prove Theorem~\ref{thm:main}. Recall that $M_1^n, M_2^n, \ldots$ are the components of the random multigraph $\mathbf{M}_n(\nu)$, listed in decreasing order of size.  We will make the distance and measure explicit in each by writing $(M_i^n, d_i^n, \mu_i^n)_{i \ge 1}$.  Recall also that $\mathbf{F}_n(\nu)$ is the forest encoded by $H^n$.  In order to recover $\mathbf{M}_n(\nu)$ from $\mathbf{F}_n(\nu)$ we remove pairs of marked leaves and replace them by back-edges as described above.  

\begin{proof}[Proof of Theorem~\ref{thm:main}]
Write $(h^n_{(i)}, \mathcal{P}^n_{(i)})$ for the $i$th element of the sequence \linebreak $\mathrm{ord} \left( h^n_{k},  \mathcal{P}^n_{k}, k \ge 1 \right)$.  Write $(T^n_{i}, d^n_{i}, \mu^n_{i})$ for the tree with rescaled height process $h^n_{(i)}$ and with mass $n^{-\alpha/(\alpha+1)}$ on each vertex.  (For $i \ge 1$, these are the trees of the forest $\mathbf{F}_n(\nu)$ listed in decreasing order of mass.)  By Skorokhod's representation theorem, we may work on a probability space where the convergence
\[
\left(h^n_{(i)}, \mathcal{P}^n_{(i)}, i \ge 1 \right) \to (\widetilde{h}_i, \mathcal{P}_i, i \ge 1)
\]
occurs almost surely. By (\ref{eqn:GHbound}), this entails that
\[
\mathrm{d_{GHP}}((T^n_{i}, d^n_{i}, \mu^n_{i}), (\widetilde{\mathcal{T}}_i, \widetilde{d}_i, \widetilde{\mu}_i)) \to 0\quad \text{a.s.}
\]
as $n \to \infty$.  In order to obtain $(M^n_i, d_i^n, \mu_i^n)$ from $(T^n_{i}, d^n_{i}, \mu^n_{i})$ if $\mathcal{P}_{(i)}^n = \sum_{r=1}^{m^n_{(i)}} \delta_{s^n_{(i),r}, t^n_{(i),r}}$ we must remove the vertices encoded by $s^{n}_{(i),r}$ and $t^n_{(i),r}$ and replace them by a single edge between their parents, for each $1 \le r \le m_{(i)}^n$.  Since edges have rescaled length $n^{-(\alpha-1)/(\alpha+1)} \to 0$ it is straightforward to see (in the same manner as in the Proof of Theorem 22 in \cite{ABBrGo1}) that we get
\begin{equation} \label{eqn:GHPconv}
\mathrm{d_{GHP}}((M_i^n,d_i^n, \mu_i^n), (\mathcal{G}_i, d_i, \mu_i)) \to 0 \quad \text{a.s.}
\end{equation}
as $n \to \infty$. The conclusion follows.
\end{proof}

\begin{appendix}
\section*{} \label{appn}
\subsection{A change of measure for spectrally positive L\'evy processes} \label{sec:measurechangeLevy}

Let $L$ be a spectrally positive L\'evy process with L\'evy measure $\pi$ satisfying
\begin{equation} \label{eqn:integrability}
\int_0^{\infty} (x \wedge x^2) \pi(dx) < \infty.
\end{equation}
Then we may write the Laplace transform of $L_t$ as
\[
\E{\exp(-\lambda L_t)} = \exp(t \Psi(\lambda)),
\]
where
\[
\Psi(\lambda) = \gamma \lambda + \frac{\delta^2 \lambda^2}{2}  + \int_0^{\infty} \pi(dx) (e^{-\lambda x} - 1 + \lambda x).
\]
We impose also that
\begin{equation} \label{eqn:params}
\gamma \ge 0, \quad \delta \ge 0
\end{equation}
and that at least one of the two following conditions holds:
\begin{equation} \label{eqn:infinitevar}
\delta > 0 \quad \text{or} \quad \int_0^{\infty} x \pi(dx) = \infty.
\end{equation}
As observed by Duquesne \& Le Gall \cite{DuquesneLeGall}, assumptions (\ref{eqn:integrability}), (\ref{eqn:params}) and (\ref{eqn:infinitevar}) together ensure that $L$ does not drift to $+\infty$ and has paths of infinite variation.

We note that $\int_0^t \Psi(\theta s) ds < \infty$ for all $\theta >0$ and all $t > 0$.

\begin{lem} \label{lem:laplint}
For any $\theta > 0$, we have
\[
\E{\exp \left(- \theta \int_0^t s dL_s \right) } = \exp\left(\int_0^t \Psi(\theta s) ds \right)= \E{\exp\left(\theta \int_0^t (L_s - L_t) ds\right)}.
\]
In consequence, the process
\[
\left(\exp \left(- \theta \int_0^t s dL_s  - \int_0^t \Psi(\theta s) ds \right), t \ge 0 \right)
\]
is a martingale.
\end{lem}

\begin{proof}
Let $M(ds, dx)$ be a Poisson random measure on $\R_+$ of intensity $ds \otimes \pi(dx)$, and let $\widetilde{M}(ds,dx)$ be its compensated version.  Then
\[
 \E{\exp \left(-\theta \int_0^t \int_0^{\infty} sx \widetilde{M}(ds,dx) \right)} = \exp\left( \int_0^t ds \int_0^{\infty} \pi(dx) (e^{-\theta sx} - 1 + \theta s x) \right).
\]
If $B$ is a standard Brownian motion, we obtain
\[
\E{\exp \left(- \theta \int_0^t s dB_s \right)} = \exp \left(\frac{1}{2} \theta^2 \int_0^t s^2 ds \right).
\]
Since we may, in general, realise $L$ as
\[
L_t = -\gamma t + \delta B_t + \int_0^t\int_0^{\infty} x \widetilde{M}(ds,dx),
\]
where $B$ and $\widetilde{M}$ are independent, we obtain
\begin{align*}
& \E{\exp \left(-\theta \int_0^t s dL_s \right)} \\
& \qquad = \E{\exp \left(\theta  \int_0^t \gamma s ds - \theta \int_0^t \delta s dB_s - \theta \int_0^t \int_0^{\infty} xs \widetilde{M}(ds,dx) \right)} \\
& \qquad = \exp \left(\gamma  \int_0^t \theta s ds + \frac{1}{2} \delta^2\int_0^t \theta^2  s^2 ds + \int_0^t ds \int_0^{\infty} \pi(dx) (e^{-\theta sx} - 1 + \theta s x) \right) \\
& \qquad = \exp \left( \int_0^t \Psi(\theta s) ds \right).
\end{align*}
The second equality in the statement of the lemma follows on integrating by parts, and the martingale property follows since $L$ has independent increments.
\end{proof}

This martingale plays an important role as a Radon--Nikodym derivative.  Fix $\theta > 0$ and consider the process $X$ with independent (but non-stationary) increments and Laplace transform
\[
\E{\exp(-\lambda X_t)} = \exp\left( \int_0^t ds \int_0^{\infty}  \pi(dx) (e^{-\lambda x} - 1 + \lambda x) e^{-\theta xs} \right). 
\]
(The process $X$ may again be realised as a stochastic integral with respect to a compensated Poisson random measure on $\R_+ \times \R_+$, but this time with intensity $\exp(-xs) ds \pi(dx)$.)
Let
\[
A_t = -\frac{1}{\theta} \Psi(\theta t)  = - \gamma t - \frac{1}{2}\theta \delta^2 t^2 - \frac{1}{\theta} \int_0^{\infty} \pi(dx) (e^{-\theta t x} - 1 + \theta t x) 
\]
and $\widetilde{L}_t = \delta B_t + X_t + A_t$. (Note that this is expressed as the Doob--Meyer decomposition of $\widetilde{L}$, with $\delta B_t + X_t$ the martingale part.)

\begin{prop} \label{prop:Levymeasch}
For any $\theta > 0$ and every $t \ge 0$, we have the following absolute continuity relation: for every non-negative integrable functional $F$,
\[
\E{F(\widetilde{L}_s, 0 \le s \le t)} = \E{\exp\left( - \theta \int_0^t s dL_s - \int_0^t \Psi(\theta s) ds \right) F(L_s, 0 \le s \le t)}.
\]
\end{prop}

\begin{proof}
Observe first that Lemma~\ref{lem:laplint} entails that the change of measure is well-defined for each $t \ge 0$. 

Let us first deal with the case where $\gamma = \delta = 0$.  We use a  decomposition of the L\'evy measure similar to that in Bertoin~\cite{BertoinFactorize} or the proof of Proposition 1 in Miermont~\cite{Mier05}:
\begin{align*}
& \int_0^t ds \int_0^{\infty} \pi(dx) (e^{-\lambda x} - 1 + \lambda x) e^{-\theta x s} \\
& = \int_0^t ds \int_0^{\infty} \pi(dx) (e^{-(\lambda + \theta s) x} - 1 + (\lambda + \theta s) x)
- \int_0^t ds \int_0^{\infty} \pi(dx) (e^{-\theta xs} - 1 + \theta xs) \\
& \qquad - \int_0^t ds \int_0^{\infty} \pi(dx) \lambda x (1 - e^{-\theta x s}) \\
& = \int_0^t \Psi(\lambda + \theta s)ds - \int_0^t \Psi(\theta s) ds - \int_0^t ds \int_0^{\infty} \pi(dx) \lambda x (1 - e^{-\theta x s}).
\end{align*}
The last integral on the right-hand side makes sense because of the integrability condition  (\ref{eqn:integrability}).
Indeed, it may be calculated as follows:
\begin{align*}
\int_0^t ds \int_0^{\infty} \pi(dx) \lambda x (1 - e^{-\theta xs}) & = \lambda \int_0^{\infty} x \pi(dx) \int_0^t (1 - e^{-\theta xs}) ds \\
& = \frac{\lambda}{\theta} \int_0^{\infty} \pi(dx) (e^{-\theta tx} - 1 + \theta tx) = \frac{\lambda}{\theta} \Psi(\theta t).
\end{align*}
Hence,
\begin{align*}
\E{\exp(-\lambda  X_t)} & = \exp \left(  \int_0^t ds \int_0^{\infty} \pi(dx) (e^{-\lambda x} - 1 + \lambda x) e^{-\theta x s}\right) \\
& = \exp\left( - \frac{\lambda}{\theta} \Psi(\theta t) +  \int_0^t \Psi(\lambda +  \theta s) ds -   \int_0^t \Psi(\theta s) ds\right)
\end{align*}
and we obtain
\[
\E{\exp(-\lambda (X_t + A_t))} = \exp\left( \int_0^t \Psi(\lambda+\theta s) ds -  \int_0^t \Psi(\theta s) ds  \right).
\]

Consider the stochastic integral
\[
\int_0^t  (\lambda + \theta s) dL_s = \lambda L_t + \theta \int_0^t s dL_s.
\]
We have
\[
\E{\exp \left(-  \int_0^t  (\lambda + \theta s) dL_s \right)} = \exp \left( \int_0^t \Psi(\lambda + \theta s)) ds \right)
\]
and so
\[
\E{\exp(-\lambda (X_t+A_t))} = \E{\exp \left(-\lambda L_t - \theta \int_0^t s dL_s - \int_0^t \Psi(\theta s) ds  \right)}.
\]

Suppose now that $0=t_0 < t_1 < \dots < t_m=t$.  Let $\lambda_1, \ldots, \lambda_m \in \R_+$. Then, by the fact that $X$ has independent increments,
\[
\E{\exp \left(-\sum_{i=1}^m \lambda_i (\widetilde{L}_{t_i} - \widetilde{L}_{t_{i-1}}) \right)}
= \prod_{i=1}^m \E{\exp(-\lambda_i (\widetilde{L}_{t_i} - \widetilde{L}_{t_{i-1}}))}.
\]
By the same argument as above, we then have
\begin{align*}
\E{\exp \left(-\sum_{i=1}^m \lambda_i (\widetilde{L}_{t_i} - \widetilde{L}_{t_{i-1}}) \right)}
& = \prod_{i=1}^m \E{\exp \left(- \int_{t_{i-1}}^{t_i} (\lambda_i + \theta s) dL_s - \int_{t_{i-1}}^{t_i} \Psi(\theta s) ds \right)} \\
& = \E{ \exp \left( -\sum_{i=1}^m \int_{t_{i-1}}^{t_i} (\lambda_i + \theta s)dL_s -  \int_{0}^{t} \Psi(\theta s) ds  \right)},
\end{align*}
since $L$ also has independent increments. Again by integration by parts, we then get that the right-hand side is equal to
\[
\E{ \exp \left(-\sum_{i=1}^m \lambda_i(L_{t_i} - L_{t_{i-1}}) -\theta \int_0^t sdL_s -  \int_{0}^{t} \Psi(\theta s)\right)}.
\]
This yields the claimed result for $\gamma = \delta = 0$.

Now let us instead suppose that $\gamma \ge 0$, $\delta > 0$ and there is no jump component i.e.\ $L_t = -\gamma t + \delta B_t$ and $\widetilde{L}_t = -\gamma t + \delta B_t - \delta \theta t^2/2$.  Then by the Cameron--Martin--Girsanov formula (see, for example, Section 5.6 of Le Gall~\cite{LeGallBook}),
\begin{align*}
\E{f(\widetilde{L}_s, 0 \le s \le t)} & = \E{ \exp \left( - \delta \theta \int_0^t s dB_s - \delta^2 \theta^2 \int_0^t s^2 ds \right) f(L_s, 0 \le s \le t)} \\
& = \E{ \exp \left( - \theta \int_0^t s dL_s -  \int_0^t \Psi(\theta s) ds \right) f(L_s, 0 \le s \le t)}.
\end{align*}

The result for general $\gamma, \delta$ and $\pi$ now follows using the independence of $X$ and $B$.
\end{proof}

The $\alpha$-stable case stated in Proposition~\ref{prop:stablemeasch} is obtained by setting $\gamma = \delta = 0$, $\pi(dx) = \frac{c}{\mu} x^{-(\alpha+1)}dx$ and $\theta = 1/\mu$.  The Brownian case is obtained by taking $\gamma = 0$, $\delta = \sqrt{\beta/\mu}$, $\theta = 1/\mu$ and no L\'evy measure $\pi$.

\subsection{Size-biased reordering}
In this section, we prove some elementary results about the size-biased reordering $(\hat{D}_1^n, \hat{D}_n^2, \ldots, \hat{D}_n^n)$ of the degrees.  First, we prove Proposition~\ref{prop:changemeas}.

\begin{proof}[Proof of Proposition~\ref{prop:changemeas}]
Denote the set of permutations of $\{1,2,\ldots,n\}$ by $\mathfrak{S}_n$.  By definition,
\begin{align*}
& \Prob{\hat{D}_1^n = k_1, \hat{D}_2^n = k_2, \ldots, \hat{D}_n^n = k_n} \\
& \qquad = \Prob{D_{\Sigma(1)} = k_1, D_{\Sigma(2)} = k_2, \ldots, D_{\Sigma(n)} = k_n} \\
& \qquad = \sum_{\sigma \in \mathfrak{S}_n} \Prob{D_{\sigma(1)} = k_1, D_{\sigma(2)}=k_2, \ldots, D_{\sigma(n)} = k_n, \Sigma = \sigma} \\
& \qquad = \sum_{\sigma \in \mathfrak{S}_n} \Prob{D_{\sigma(1)} = k_1, D_{\sigma(2)}=k_2, \ldots, D_{\sigma(n)} = k_n} \frac{k_1}{\sum_{j=1}^n k_j} \frac{k_2}{\sum_{j=2}^n k_j} \ldots \frac{k_n}{k_n} \\
& \qquad = n! \ \nu_{k_1} \nu_{k_2} \ldots \nu_{k_n} \frac{k_1}{\sum_{j=1}^n k_j} \frac{k_2}{\sum_{j=2}^n k_j} \ldots \frac{k_n}{k_n},
\end{align*}
since $D_1, \ldots, D_n$ are i.i.d.\ with law $\nu$.  Rearrangement of this expression yields
\begin{align*}
\Prob{\hat{D}_1^n = k_1, \hat{D}_2^n = k_2, \ldots, \hat{D}_n^n = k_n}
& = k_1 \nu_{k_1} k_2 \nu_{k_2} \ldots k_n \nu_{k_n} \prod_{i=1}^n \frac{(n-i+1)}{\sum_{j=i}^n k_j} \\
& = \frac{k_1 \nu_{k_1}}{\mu} \frac{k_2 \nu_{k_2}}{\mu} \cdots \frac{k_n \nu_{k_n}}{\mu} \prod_{i=1}^n \frac{(n-i+1)\mu }{\sum_{j=i}^n k_j}.
\end{align*}
Now
\begin{align*}
& \Prob{\hat{D}_1^n = k_1, \hat{D}_2^n = k_2, \ldots, \hat{D}_m^n = k_m} \\
& = \sum_{k_{m+1}, \ldots, k_n \ge 1} \Prob{\hat{D}_1^n = k_1, \hat{D}_2^n = k_2, \ldots, \hat{D}_n^n = k_n} \\
& = \frac{k_1 \nu_{k_1}}{\mu} \frac{k_2 \nu_{k_2}}{\mu} \cdots \frac{k_m \nu_{k_m}}{\mu}  \mu^m n! \sum_{k_{m+1}, \ldots, k_n \ge 1}  k_{m+1} \nu_{k_{m+1}} \cdots k_n \nu_{k_n} \prod_{i=1}^n \frac{1}{\sum_{j=i}^n k_j} \\
& = \frac{k_1 \nu_{k_1}}{\mu} \frac{k_2 \nu_{k_2}}{\mu} \cdots \frac{k_m \nu_{k_m}}{\mu}  \mu^m n! \\
& \qquad \times \sum_{k_{m+1}, \ldots, k_n \ge 1}    \prod_{i=1}^m \frac{1}{\sum_{j=i}^m k_j + \sum_{j=m+1}^n k_j} \nu_{k_{m+1}} \ldots \nu_{k_n} \prod_{\ell=m+1}^n \frac{k_{\ell}}{\sum_{j=\ell}^n k_j} \\
& = \frac{k_1 \nu_{k_1}}{\mu} \frac{k_2 \nu_{k_2}}{\mu} \cdots \frac{k_m \nu_{k_m}}{\mu}  \mu^m \frac{n!}{(n-m)!} \\
& \qquad \times \sum_{k_{m+1}, \ldots, k_n \ge 1}    \prod_{i=1}^m \frac{1}{\sum_{j=i}^m k_j + \sum_{j=m+1}^n k_j} \Prob{\hat{D}_1^{n-m} = k_{m+1}, \ldots, \hat{D}^{n-m}_{n-m} = k_n}.
\end{align*}
We have that $\sum_{j=1}^{n-m} \hat{D}_j^{n-m} \equidist \sum_{j=m+1}^{n} D_j = \Xi_{n-m}$.  It follows that the last expression is equal to
\[
\frac{k_1 \nu_{k_1}}{\mu} \frac{k_2 \nu_{k_2}}{\mu} \cdots \frac{k_m \nu_{k_m}}{\mu}   \frac{n! \mu^m}{(n-m)!} \E{\prod_{i=1}^m \frac{1}{\sum_{j=i}^m k_j + \Xi_{n-m}}},
\]
and the claimed result follows.
\end{proof}

A simple consequence of Proposition~\ref{prop:changemeas} is the following stochastic domination.

\begin{lem} \label{lem:stochdom}
We have
\[
(\hat{D}^n_1, \hat{D}_2^n, \ldots, \hat{D}_n^n) \le_{\mathrm{st}} (Z_1, Z_2, \ldots, Z_n).
\]
\end{lem}

\begin{proof}
By Proposition~\ref{prop:changemeas} we have
\[
\Prob{\hat{D}^n_1 \ge d_1, \hat{D}_2^n \ge d_2, \ldots, \hat{D}_n^n \ge d_n} = \E{ \prod_{i=1}^n \frac{(n-i+1)\mu}{\sum_{j=i}^n Z_j} \I{Z_1 \ge d_1, Z_2 \ge d_2, \ldots, Z_n \ge d_n}}.
\]
Let
\[
f(k_1, k_2, \ldots, k_n) = \prod_{i=1}^n \frac{(n-i+1)\mu}{\sum_{j=i}^n k_j}
\]
and
\[
g(k_1, k_2, \ldots, k_n) = \I{k_1 \ge d_1, k_2 \ge d_2, \ldots, k_n \ge d_n}.
\]
Then $f$ is a decreasing function of its arguments and $g$ is an increasing function of its arguments.  It follows from the FKG inequality that
\[
\E{f(Z_1, Z_2, \ldots, Z_n) g(Z_1, Z_2, \ldots, Z_n)} \le \E{f(Z_1, Z_2, \ldots, Z_n)} \E{g(Z_1, Z_2, \ldots, Z_n)}.
\]
But $\E{f(Z_1, Z_2, \ldots, Z_n)} = 1$ and so
\[
\Prob{\hat{D}^n_1 \ge d_1, \hat{D}_2^n \ge d_2, \ldots, \hat{D}_n^n \ge d_n} \le \Prob{Z_1 \ge d_1, Z_2 \ge d_2, \ldots, Z_n \ge d_n},
\]
as required.
\end{proof}

\begin{lem} \label{lem:sumofsquares}
Fix $\alpha \in (1,2)$ and suppose $m = O(n^{\beta})$ for some $\beta < \alpha/2$.  Then as $n \to \infty$,
\[
\frac{1}{n} \sum_{i=1}^{m} (\hat{D}_i^n)^2 \convprob 0.
\]
In particular, the above holds for $m = \fl{t n^{\alpha/(\alpha+1)}}$.
\end{lem}

\begin{proof}
By Lemma~\ref{lem:stochdom}, it is sufficient to prove that
\[
\frac{1}{n} \sum_{i=1}^{m} Z_i^2 \convprob 0.
\]
By Theorem 2.5.9 of Durrett~\cite{Durrett}, we have
\[
\limsup_{m \to \infty} \frac{1}{m^{1/\beta}} \sum_{i=1}^{m} Z_i^2 = 0 \text{ a.s.}
\]
if and only if
\[
\sum_{m=1}^{\infty} \Prob{Z_1^2 > m^{1/\beta}} < \infty.
\]
But $\Prob{Z_1^2 > m^{1/\beta}} = \Prob{Z_1 > m^{1/2\beta}} = O(m^{-\alpha/2\beta})$, which is summable since $\alpha > 2\beta$.
\end{proof}

\begin{lem} \label{lem:lln}
As $n \to \infty$,
\[
\frac{1}{n} \sum_{i=\fl{t n^{\alpha/(\alpha+1)}}+1}^{n} \hat{D}_i^n \convprob \mu.
\]
\end{lem}

\begin{proof}
We have
\[
\sum_{i=\fl{t n^{\alpha/(\alpha+1)}}+1}^{n} \hat{D}_i^n \quad =  \quad \sum_{i=1}^n \hat{D}_i^n  - \sum_{i=1}^{\fl{t n^{\alpha/(\alpha+1)}}} \hat{D}_i^n \quad = \quad \sum_{i=1}^n D_i - \sum_{i=1}^{\fl{t n^{\alpha/(\alpha+1)}}} \hat{D}_i^n.
\]
By the weak law of large numbers,
\[
\frac{1}{n} \sum_{i=1}^n D_i \convprob \E{D_1} = \mu
\]
and, since $\hat{D}_i^n \ge 1$ for $1 \le i \le n$, by Lemma~\ref{lem:sumofsquares} we have that
\[
\frac{1}{n} \sum_{i=1}^{\fl{t n^{\alpha/(\alpha+1)}}} \hat{D}_i^n \convprob 0.
\]
The result follows.
\end{proof}

\subsection{Convergence of the measure-change in the Brownian case}\label{subsec:changemeasbrown}

Recall that we have $\mu := \E{D_1}$, $\E{D_1^2} = 2 \mu$ and $\beta := \E{D_1(D_1-1)(D_1 - 2)}$, so that $\E{D_1^3} = \beta + 4\mu$.

\begin{lem}
Let $\mathcal{L}(\lambda) := \E{\exp(-\lambda D_1)}$. Then as $\lambda \to 0$,
\begin{equation} \label{eqn:LaplaceBrownian}
\mathcal{L}(\lambda) = \exp \left(-\lambda \mu + \frac{\lambda^2 \mu(2-\mu)}{2} - \frac{\lambda^3}{6}(\beta + 4 \mu - 6\mu^2 + 2 \mu^3) + o(\lambda^3) \right).
\end{equation}
\end{lem}

\begin{proof}
The first three cumulants of $D_1$ are 
\[
\E{D_1} = \mu, \quad \var{D_1} = \mu(2-\mu), \quad \E{(D_1 - \mu)^3} = \beta + 4\mu - 6\mu^2 + 2\mu^3,
\]
and the result follows immediately.
\end{proof}

Recall from Proposition~\ref{prop:changemeas} that 
\[
\phi_m^n(k_1, k_2, \ldots, k_m) = \E{\prod_{i=1}^m \frac{(n-i+1)\mu}{\sum_{j=i}^m k_j + \Xi_{n-m}}}.
\]
We prove the following lemma.

\begin{lem}
Let $s(0) = 0$ and $s(i) = \sum_{j= 1}^i (k_j-2)$ for $i \ge 1$. Suppose that $|s(i)- s(m)| \le n^{1/3} \log n$ for all $0 \le i \le m$.  Then if $m = \Theta(n^{2/3})$, we have
\[
\phi_n^m(k_1,k_2, \ldots,k_m) \ge \exp\left(\frac{1}{n \mu}\sum_{i=0}^{m} (s(i) - s(m)) - \frac{\beta m^3}{6 \mu^3 n^2}  \right) (1 + o(1)),
\]
where the $o(1)$ term is independent of $k_1, \ldots, k_m \ge 1$ satisfying the conditions.
\end{lem}

\begin{proof}  The method of proof is similar in spirit to, but somewhat more involved than, that of Lemma~\ref{lem:lowerbound}.  Let us first introduce some useful notation. Let $D'_i = D_i - \mu$, the centred degree random variables, and let $\Delta_{n-m}:=\Xi_{n-m}-\mu(n-m)$ be their sum. Let $\psi$ be the log-Laplace transform of $D_1'$,
\[
\psi(\lambda) = \log \E{\exp(-\lambda D_1')},
\]
so that as $\lambda \to 0$, we have
\begin{equation} \label{eqn:psiexp}
\psi(\lambda) = \frac{\lambda^2 \mu(2-\mu)}{2} - \frac{\lambda^3}{6}(\beta + 4 \mu - 6\mu^2 + 2 \mu^3) + o(\lambda^3).
\end{equation}

Now,
\begin{align*}
& \phi_n^m(k_1, \ldots, k_m) \\
& = \prod_{i=1}^{m}\left(1-\frac{i-1}{n}\right) \\
& \qquad \quad \times \E{\exp\left(-\sum_{i=1}^{m}\log\left(1+\frac{\Delta_{n-m}+s(m)-s(i-1)+2(m-i+1)-\mu m}{\mu n}\right)\right)}.
\end{align*}
We use Taylor expansion in order to approximate the exponent:
\begin{align} \label{eq:Taylor}
& \sum_{i=1}^m \left( \log \left(1 - \frac{i-1}{n} \right) - \log\left(1+\frac{\Delta_{n-m}+ s(m)-s(i-1)+ (2 - \mu)m-2(i-1)}{\mu n}\right) \right) \\
& \notag = - \frac{m^2}{2n} - \frac{m^3}{6n^2} + o(1) \\
& \notag \qquad  - \frac{m\Delta_{n-m}}{n \mu} + \frac{1}{n \mu} \sum_{i=0}^m (s(i) - s(m)) - \frac{(2-\mu)m^2}{n \mu} + \frac{m^2}{n \mu} + o(1) \\
& \notag \qquad + \frac{m(\Delta_{n-m})^2}{2n^2 \mu^2} + \frac{1}{2\mu^2n^2} \sum_{i=0}^m(s(i) - s(m))^2 + \frac{(2-\mu)^2 m^3}{2\mu^2n^2} + \frac{2 m^3}{3 \mu^2 n^2} - \frac{(2-\mu) m^3}{\mu^2n^2}  \\
& \notag \qquad + o(1)\\
& \notag \qquad + \frac{\Delta_{n-m}}{\mu^2n^2} \sum_{i=0}^m(s(m) - s(i)) - \frac{(\mu-1)m^2 \Delta_{n-m}}{\mu^2 n^2} + \frac{(2-\mu)m}{\mu^2 n^2} \sum_{i=0}^m (s(m) - s(i)) \\
& \notag \qquad - \frac{2}{\mu^2n^2} \sum_{i=0}^m i (s(m) - s(i)) + \cdots 
\end{align}
As $\Delta_{n-m}$ is a centred sum of i.i.d.\ random variables with finite variance, the central limit theorem applies and we have that $n^{-1/2} \Delta_{n-m} \convdist \mathrm{N}(0,\sqrt{\mu(2-\mu)})$ as $n \to \infty$. The desired lower bound will, however, be obtained by restricting to the moderate deviation event
\[
\mathcal{E}_n = \left\{ -(2-\mu)m - n^{7/12} \le \Delta_{n-m} \le -(2-\mu)m + n^{7/12}\right\}.
\]
On this event, for any $0 \le i \le m$, we have
\begin{equation*}
|\Delta_{n-m}| = O(n^{2/3}), \quad |s(m) - s(i)| = O(n^{1/3} \log n) \quad \text{and} \quad |(2-\mu)m - 2(i-1)| = O(n^{2/3}).
\end{equation*}
So we have
\begin{align*}
& \frac{1}{2\mu^2n^2} \sum_{i=0}^m (s(i) - s(m))^2 = o(1), \qquad \frac{(2-\mu)m}{\mu^2 n^2} \sum_{i=0}^m (s(m) - s(i)) = o(1), \\
& \frac{\Delta_{n-m}}{\mu^2 n^2} \sum_{i=0}^m (s(m)-s(i)) = o(1), \quad \text{and} \quad - \frac{2}{\mu^2n^2} \sum_{i=0}^m i (s(m) - s(i)) = o(1),
\end{align*}
and that the remainder term (hidden in the ellipsis)  in the expansion of (\ref{eq:Taylor}) is $o(1)$. Using these facts we see that, on $\mathcal{E}_n$, the exponent (\ref{eq:Taylor}) is equal to $F_n + o(1)$, where 
\begin{align*}
F_n & :=\frac{1}{n\mu}\sum_{i=0}^m(s(i)-s(m)) -\frac{(2-\mu) m^2}{2\mu n}-\frac{(2-\mu)(\mu-1)m^3}{3\mu^2 n^2} \\
& \qquad + \frac{m(\Delta_{n-m})^2}{2 n^2 \mu^2} -\left(\frac{m}{\mu n} + \frac{(\mu -1)m^2}{\mu^2n^2} \right)\Delta_{n-m}.
\end{align*}
In order to find a lower bound on $\E{\exp(F_n) \mathbbm{1}_{\mathcal{E}_n}}$, we first consider the expectation of the stochastic part,
\[
\E{\exp \left(\frac{m(\Delta_{n-m})^2}{2 n^2 \mu^2} -\left(\frac{m}{\mu n} + \frac{(\mu -1)m^2}{\mu^2n^2} \right)\Delta_{n-m}\right) \mathbbm{1}_{\mathcal{E}_n}}.
\]
Let $\theta > 0$ (we shall choose a specific value for $\theta$ shortly) and define an equivalent measure $\Q$ via
\[
\frac{\mathrm d \Q}{\mathrm d \mathbb{P}} = \exp \left( -\theta \Delta_{n-m} - (n-m) \psi(\theta)  \right).
\]
Because the Radon--Nikodym derivative has a product form, under $\Q$ the random variables $D_1', D_2', \ldots, D_{n-m}'$ are still i.i.d.\ and each have mean
\[
\mathbb{E}_{\Q} \left[ D_1' \right] = \frac{\E{D_1' \exp(-\theta D_1' )}}{\E{\exp(-\theta D_1')}} = - \psi'(\theta)
\]
and variance $\mathrm{var}_{\Q}(D_1') = \psi''(\theta)$.
Now fix
\[
\theta = \frac{m}{\mu n} + \frac{m^2}{\mu n^2},
\]
so that
\begin{align*}
\mathbb{E}_{\Q}[\Delta_{n-m}] & = - (n-m)\psi'(\theta) \\
& = -(n-m) \left[ \frac{(2-\mu)m}{n} + \frac{(2-\mu)m^2}{n^2} - \frac{(\beta +4\mu - 6 \mu^2 + 2 \mu^3) m^2}{2 \mu^2 n^2} \right. \\ 
& \qquad \qquad  \qquad \quad \left. - \frac{(\beta +4\mu - 6 \mu^2 + 2 \mu^3) m^3}{2 \mu^2 n^3} + o\left(\frac{m^3}{n^3} \right) \right] \\
& = -(2-\mu)m + O(n^{1/3}) \\
\intertext{and}
\mathrm{var}_{\Q}(\Delta_{n-m}) & = (n-m)\psi''(\theta) =  \mu(2-\mu)(n-m) + O(n^{1/3}).
\end{align*}
Using Chebyshev's inequality and the fact that $n^{1/3} \ll n^{7/12}$, it follows that 
\[
\Q(\mathcal{E}_n^c) \le \Q(|\Delta_{n-m} + (2-\mu)m| > \tfrac{1}{2}n^{7/12}) \le \frac{5\mathrm{var}_{\Q}(\Delta_{n-m})}{n^{7/6}} = O(n^{-1/6}).
\]
So
\begin{align*}
&  \E{\exp \left(\frac{m(\Delta_{n-m})^2}{2 n^2 \mu^2} -\left(\frac{m}{\mu n} + \frac{(\mu -1)m^2}{\mu^2n^2} \right)\Delta_{n-m}\right) \mathbbm{1}_{\mathcal{E}_n} } \\
& \qquad = \exp \left( (n-m) \psi \left( \frac{m}{n \mu} + \frac{m^2}{\mu n^2} \right) \right) \mathbb{E}_{\Q} \left[ \exp \left(\frac{m(\Delta_{n-m})^2}{2 n^2 \mu^2} + \frac{m^2\Delta_{n-m}}{\mu^2 n^2}  \right) \mathbbm{1}_{\mathcal{E}_n} \right]  \\
& \qquad \ge \exp \left( (n-m) \psi \left( \frac{m}{n \mu} + \frac{m^2}{\mu n^2} \right) \right)  \Q(\mathcal{E}_n)  \\
& \qquad \qquad \times \exp \left( \frac{m((2-\mu)m - n^{7/12})^2}{2n^2\mu^2} - \frac{m^2((2-\mu)m + n^{7/12})}{\mu^2n^2} \right)\\
& \qquad = (1 + o(1)) \exp \left( (n-m) \psi \left( \frac{m}{n \mu} + \frac{m^2}{\mu n^2} \right) + \frac{(2-\mu)^2 m^3}{2 n^2 \mu^2} - \frac{(2-\mu)m^3}{\mu^2 n^2} \right).
\end{align*}
Now,
\begin{align*}
& (n-m) \psi \left( \frac{m}{n\mu} + \frac{m^2}{\mu n^2} \right) \\
& = (n-m) \left[ \frac{(2-\mu)m^2}{2\mu n^2} + \frac{(2-\mu)m^3}{\mu n^3} - \frac{(\beta + 4\mu - 6 \mu^2 + 2 \mu^3) m^3}{6 \mu^3 n^3} \right] + o(1) \\
& = \frac{(2-\mu) m^2}{2\mu n} + \frac{(2-\mu) m^3}{2 \mu n^2} - \frac{(\beta + 4\mu - 6 \mu^2 + 2 \mu^3) m^3}{6 \mu^3 n^2} + o(1).
\end{align*}
It follows that
\begin{align*}
& \phi_n^m(k_1, \ldots, k_m) \\
& \qquad \ge (1+o(1)) \E{\exp(F_n) \mathbbm{1}_{\mathcal{E}_n}} \\
& \qquad \ge (1+o(1)) \exp \left( \frac{(2-\mu) m^2}{2\mu n} + \frac{(2-\mu) m^3}{2 \mu n^2} - \frac{(\beta + 4\mu - 6 \mu^2 + 2 \mu^3) m^3}{6 \mu^3 n^2} \right. \\ 
& \qquad  \qquad \qquad \qquad \qquad \quad + \frac{(2-\mu)^2 m^3}{2 n^2 \mu^2} - \frac{(2-\mu)m^3}{\mu^2 n^2}  -\frac{(2-\mu) m^2}{2\mu n}-\frac{(2-\mu)(\mu-1)m^3}{3\mu^2 n^2} \\
& \qquad \qquad \qquad \qquad \qquad \quad  \left. + \frac{1}{n\mu}\sum_{i=0}^m(s(i)-s(m)) \right) \\
& \qquad = (1 + o(1)) \exp \left( \frac{1}{n\mu}\sum_{i=0}^m(s(i)-s(m))  - \frac{\beta m^3}{6 \mu^3 n^2} \right),
\end{align*}
as claimed.
\end{proof}

The event $\{|S(i)- S(m)| \le n^{1/3} \log n \text{ for $1 \le i \le m$}\}$ has probability tending to 1 as $n \to \infty$, and so the analogue of Proposition~\ref{prop:discretechangeofmeas} now follows exactly as in the $\alpha \in (1,2)$ case.  

\subsection{Convergence of a single large component for $\alpha \in (1,2)$}\label{sec:singlecompo}

In this section, we consider a large component of the graph conditioned to have size $\fl{x n^{\alpha/(\alpha+1)}}$, for $\alpha \in (1,2)$ only, and do the main technical work necessary to prove that it converges in distribution to a single component of the stable graph conditioned to have size $x$. By arguments analogous to those in Section~\ref{sec:graph}, it is essentially sufficient to consider a single tree in the forest $\widetilde{\mathbf{F}}_n(\nu)$ of size $\fl{xn^{\alpha/(\alpha +1)}}$, described by an excursion of the corresponding coding functions $\widetilde{S}^n$ and $\widetilde{G}^n$. The main result of this section, Theorem~\ref{thm:convofexcursion}, is a conditioned version of Theorem~\ref{thm:jointdfwheight}, which says that these excursions converge jointly in distribution to normalised excursions of $\widetilde{L}$ and $\widetilde{H}$ of length $x$.  (This is precisely the analogue of Theorem~\ref{thm:singletree} in the measure-changed setting.)  At the end of the section, we sketch how to obtain the metric space scaling limit of a single large component of the graph.

For simplicity, we will make the assumption that the support of the law of $D_1$ is $\Z_+$ so that excursions of any strictly positive length occur with positive probability.  This assumption is not necessary, since the condition $\Prob{D_1 = k} \sim c k^{-(\alpha + 2)}$ implies that the greatest common divisor of $\{k \ge 2: \Prob{D_1-1 = k} > 0\}$ is $1$, so that the claimed results all hold for $n$ sufficiently large.

Recall that $(S(k), k \ge 0)$ is a random walk which is skip-free to the left and in the domain of attraction of an $\alpha$-stable L\'evy process.  Let
\[
\mathcal{E}_m = \Big\{S(k) \ge 0 \text{ for $0 < k < m$, } S(m) = -1 \Big\},
\]
the event that the first $m$ steps form an excursion above the running minimum. If $m = \fl{n^{\alpha/(\alpha+1)}x}$ then, by Theorem~\ref{thm:singletree},
\begin{align*}
& \E{f\left(n^{-1/(\alpha+1)} S(\fl{n^{\alpha/(\alpha+1)} t}), n^{-(\alpha - 1)/(\alpha+1)} G(\fl{n^{\alpha/(\alpha+1)} t}),0 \le t \le x \right) \Big| \mathcal{E}_m} \\
& \qquad \to \N^{(x)} \left[ f(\mathbbm{e}, \mathbbm{h}) \right].
\end{align*}
More generally, write
\[
\mathcal{E}_{m_1, m_2} = \Big\{S(m_1) = \min_{0 \le k \le m_2-1} S(k) = S(m_2) +1 \Big\}
\]
(so that $\mathcal{E}_{m} = \mathcal{E}_{0,m}$) and, similarly,
\[
\widetilde{\mathcal{E}}^n_{m_1, m_2} = \Big\{\widetilde{S}^n(m_1) = \min_{0 \le k \le m_2-1} \widetilde{S}^n(k) = \widetilde{S}^n(m_2) +1 \Big\},
\]
so that $\widetilde{\mathcal{E}}^n_{m_1,m_2}$ is the event that there is an excursion of $\widetilde{S}^n$ above its running minimum between times $m_1$ and $m_2$ (recall that $\widetilde{S}^n$ and $\widetilde{G}^n$ have the same excursion intervals). This, of course, corresponds to a component of size $m_2 - m_1$.  Observe that the corresponding excursion of the height process starts and ends at 0. We will prove the following result.

\begin{thm} \label{thm:convofexcursion}
For any bounded continuous test function $f$, $0 \le t_1 < t_2$ such that $t_2 - t_1 = x$, and $m_1 = \fl{t_1 n^{\alpha/(\alpha+1)}}$, $m_2 = \fl{t_2 n^{\alpha/(\alpha+1)}}$, $m=m_2-m_1$ then
\begin{align*}
& \mathbb{E} \Big[ f \Big( n^{-1/(\alpha+1)} [\widetilde{S}^n(\fl{(t_1 + t) n^{\alpha/(\alpha+1)}}) - \widetilde{S}^n(\fl{t_1 n^{\alpha/(\alpha+1)}})], \\
& \hspace{2cm} n^{-(\alpha-1)/(\alpha+1)} \widetilde{G}^n(\fl{(t_1 + t) n^{\alpha/(\alpha+1)}}), \quad  0 \le t \le n^{-\alpha/(\alpha+1)} m \Big) \Big| \widetilde{\mathcal{E}}^n_{m_1, m_2} \Big] \\
& \hspace{1cm} \to \frac{\mathbb{N}^{(x)} \left[ \exp \left( \frac{1}{\mu} \int_0^x \mathbbm{e}(t)dt \right) f(\mathbbm{e}, \mathbbm{h}) \right]}{\mathbb{N}^{(x)} \left[ \exp \left( \frac{1}{\mu} \int_0^x \mathbbm{e}(t)dt \right) \right]}
\end{align*}
as $n \to \infty$.
\end{thm}

We need to prove a refinement of Lemma~\ref{lem:lowerbound}, to show that the change of measure is well-behaved at times when the process attains a new minimum.  

\begin{prop} \label{prop:upperbound}
Fix $T > 0$.  For $n \ge 1$ and $m \le T n^{\frac{\alpha}{\alpha+1}}$, let $k^{(n)}_1, k^{(n)}_2, \ldots, k^{(n)}_m \ge 1$ and let $s^{(n)}(i) = \sum_{j=1}^i(k^{(n)}_j - 2)$ be such that $s^{(n)}(0) = 0$ and $s^{(n)}(i) > s^{(n)}(m)$ for $1 \le i \le m-1$.  Then
\[
\phi_n^m(k^{(n)}_1, k^{(n)}_2, \ldots, k^{(n)}_m) = (1 + \delta_{n}) \exp \left(\frac{1}{n \mu} \sum_{i=0}^{m} (s^{(n)}(i)-s^{(n)}(m)) - \frac{C_{\alpha} m^{\alpha+1}}{(\alpha+1) \mu^{\alpha+1}n^{\alpha}} \right) ,
\]
where $\delta_{n}$ depends only on $n$ and $T$, and $\delta_{n} \to 0$ as $n \to \infty$.
\end{prop}

We start with a technical lemma.

\begin{lem} \label{lem:exponentialmgf}  Suppose that $m = O(n^{\alpha/(\alpha+1)})$.
Let $E_1, E_2, \ldots$ be i.i.d.\ standard exponential random variables.  Suppose that for each $n$ we have a sequence $a^{(n)}_1, a^{(n)}_2, \ldots, a^{(n)}_m$ such that $a^{(n)}_i \in (0,Km/n)$ for all $1 \le i \le m$ for some constant $K$.
\begin{enumerate}
\item[(a)] We have
\[
\sum_{i=1}^m a^{(n)}_i (E_i - 1) \to 0
\]
in $L^2$.
\item[(b)] For any $p > 1$, there exists a constant $C > 0$ such that 
\[
\E{\exp\left(p \sum_{i=1}^m a^{(n)}_i (E_i - 1) \right)} \le C \exp \left( \frac{2p K^2 m^3}{n^2} \right).
\]
\end{enumerate}
Both the convergence in $(a)$ and the bound in $(b)$ are uniform in sequences $(a^{(n)}_i)$ satisfying the above conditions.
\end{lem}

\begin{proof}
(a) 
Since the sum is centred, we have
\[
\E{\left( \sum_{i=1}^m a_i^{(n)} (E_i - 1) \right)^2} = \var{\sum_{i=1}^m a_i^{(n)} (E_i - 1)} = \sum_{i=1}^m (a_i^{(n)})^2 \le \frac{K^2m^3}{n^2} \to 0
\]
as $n \to \infty$.

(b) For $0<a < 1/2$ we have
\[
\E{\exp(aE_1 - a)} = \frac{e^{-a}}{1-a} = e^{-a} \left(1 + \frac{a}{1-a} \right) \le \exp\left(-a + \frac{a}{1-a} \right) \le \exp(2a^2).
\]
So for sufficiently large $n$ we have
\[
 \E{\exp \left( p \sum_{i=1}^m a^{(n)}_i (E_i - 1) \right)}  \le \exp \left(2p \sum_{i=1}^m (a_i^{(n)})^2 \right)  \le \exp \left(\frac{2pK^2m^3}{n^2} \right). \qedhere
\]
\end{proof}
\begin{proof}[Proof of Proposition~\ref{prop:upperbound}]
The lower bound does not rely on $s^{(n)}$ attaining a new minimum at time $m$, and has already been proved in Lemma~\ref{lem:lowerbound}; we need a matching upper bound. To ease readability, we will suppress the superscripts on $k_i^{(n)}$ and $s^{(n)}(i)$.  Now,
\begin{align*}
& \E{\prod_{i=1}^m \frac{(n-i+1) \mu}{\sum_{j=i}^m k_j + \Xi_{n-m}} } \\
& = \prod_{i=1}^{m-1} \left(1 - \frac{i}{n} \right) \E{\prod_{i=1}^m \frac{n \mu}{s(m) - s(i-1) + 2(m-i+1) + \Xi_{n-m}} } \\
& = \prod_{i=1}^{m-1} \left(1 - \frac{i}{n} \right) \E{ \exp \left( - \sum_{i=1}^m \left\{ \frac{(s(m) - s(i-1) + 2(m-i+1) + \Xi_{n-m}}{n \mu} - 1 \right\} E_i  \right) },
\end{align*}
where $E_1, E_2, \ldots$ are i.i.d.\ standard exponential random variables, independent of $\Xi_{n-m}$.   We shall first consider the expectation conditionally on $E_1, E_2, \ldots, E_m$.  Write $A_m$ for the quantity
\begin{align*}
& \prod_{i=1}^{m-1} \left(1 - \frac{i}{n} \right) \\
& \quad \times \E{ \exp \left( - \sum_{i=1}^m \left\{ \frac{(s(m) - s(i-1) + 2(m-i+1) + \Xi_{n-m}}{n \mu} - 1 \right\} E_i  \right) \Bigg| E_1, \ldots, E_m}.
\end{align*}
Let $C > 0$ be a constant to be chosen later.  We will split $\E{A_m}$ into two parts, so that
\[
 \E{\prod_{i=1}^m \frac{(n-i+1) \mu}{\sum_{j=i}^m k_j + \Xi_{n-m}} } = \E{A_m \I{\sum_{i=1}^m E_i > C n^{\alpha/(\alpha+1)}}} + \E{A_m \I{\sum_{i=1}^m E_i \le C n^{\alpha/(\alpha+1)}}}.
\]
We deal with the first term on the right-hand side first.  Since $k_j \ge 1$ for all $j$ and $\Xi_{n-m} \ge n-m$ a.s., we have the crude bound
\[
s(m) - s(i-1) + 2(m-i+1) + \Xi_{n-m} \ge n-i+1 > n-m > n/2
\]
for $1 \le i \le m$ and all $n$ sufficiently large that $m/n < 1/2$.  Then
\begin{align*}
\E{A_m \I{\sum_{i=1}^m E_i > C n^{\alpha/(\alpha+1)}}} 
& \le \E{\exp \left(\left(1 - \frac{1}{2\mu} \right) \sum_{i=1}^m E_i \right) \I{\sum_{i=1}^m E_i > C n^{\alpha/(\alpha+1)}} } \\
& = \int_{C n^{\alpha/(\alpha+1)}}^{\infty} \exp\left(x -\frac{x}{2\mu}\right) \frac{e^{-x} x^{m-1}}{\Gamma(m)} dx \\
& =  \int_{C n^{\alpha/(\alpha+1)}}^{\infty} \exp\left(- \frac{x}{2\mu}\right) \frac{x^{m-1}}{\Gamma(m)} dx \\
& =(2\mu)^m \Prob{\frac{1}{2\mu}  \sum_{i=1}^m E_i >  C n^{\alpha/(\alpha+1)}}.
\end{align*}
By Markov's inequality, this last quantity is bounded above by
\[
(2\mu)^m  \E{\exp\left(\frac{1}{2\mu} \sum_{i=1}^m E_i \right)} \exp \left( -Cn^{\alpha/(\alpha+1)}\right)  = \left( \frac{(2\mu)^2}{2\mu-1} \right)^m \exp(-Cn^{\alpha/(\alpha+1)}) \to 0
\]
as $n \to \infty$, as long as we take $C > T (2 \log (2\mu) -\log(2\mu-1))$, which we henceforth assume.

Let us now turn to the expectation of $A_m$ on the event $\{\sum_{i=1}^m E_i \le C n^{\alpha/(\alpha+1)}\}$.  Since 
\[
\prod_{i=1}^{m-1} \left(1 - \frac{i}{n} \right) \le \exp\left(-\frac{m(m-1)}{2n} \right),
\]
we have
\begin{align*}
& A_m \I{\sum_{i=1}^m E_i \le C n^{\alpha/(\alpha+1)}}  \\
& \qquad \le \exp \left( - \frac{m(m-1)}{2n} +  \frac{1}{n\mu} \sum_{i=1}^m(s(i-1) - s(m)) E_i - \sum_{i=1}^m \left( \frac{2(m-i+1)}{n \mu} - 1\right) E_i \right) \\
& \qquad \qquad  \times \E{ \exp \left(- \left(\frac{1}{n \mu}\sum_{i=1}^m E_i \right) \Xi_{n-m} \right) \Bigg| E_1, \ldots, E_m} \I{\sum_{i=1}^m E_i \le Cn^{\alpha/(\alpha+1)}}.
\end{align*}
On the event $\left\{ \sum_{i=1}^m E_i \le Cn^{\alpha/(\alpha+1)} \right\}$, we have $\frac{1}{n \mu} \sum_{i=1}^m E_i = o(1)$. Hence, we may apply the asymptotic formula (\ref{eqn:Lasymp}) for the Laplace transform of $D_1$ to obtain that on the event $\left\{ \sum_{i=1}^m E_i \le Cn^{\alpha/(\alpha+1)} \right\}$ we have
\begin{align*}
& \E{ \exp \left(- \left(\frac{1}{n \mu}\sum_{i=1}^m E_i \right) \Xi_{n-m} \right) \Bigg| E_1, \ldots, E_m}  \\
& = \exp \left( - \frac{(n-m)}{n}\sum_{i=1}^m E_i + \frac{(2-\mu)}{2\mu n} \left(\sum_{i=1}^m E_i \right)^2  \! -  \frac{C_{\alpha}}{(\alpha+1) n^{\alpha} \mu^{\alpha+1}}\left( \sum_{i=1}^m E_i \right)^{\alpha+1} \! \! \! + o(1) \right) .
\end{align*}
It follows that 
\begin{align*}
& A_m \I{\sum_{i=1}^m E_i \le C n^{\alpha/(\alpha+1)}} \\
& \le \exp \left( \frac{1}{n\mu} \sum_{i=1}^m (s(i-1) - s(m)) E_i - \frac{C_{\alpha}}{(\alpha+1)n^{\alpha} \mu^{\alpha+1}}\left( \sum_{i=1}^m E_i \right)^{\alpha+1} + o(1) \right) \\
&  \quad  \times \exp \left( \frac{(2-\mu)}{2\mu n} \left( \sum_{i=1}^m E_i \right)^2 - \frac{2}{n\mu} \sum_{i=1}^m (m-i+1) E_i  + \frac{m}{n} \sum_{i=1}^m E_i - \frac{m^2}{2n} \right) \\
& \quad \times \I{\sum_{i=1}^m E_i \le Cn^{\alpha/(\alpha+1)}}.
\end{align*}
Observe that
\begin{align*}
& \frac{(2-\mu)}{2\mu n} \left( \sum_{i=1}^m E_i \right)^2 - \frac{2}{n\mu} \sum_{i=1}^m (m-i+1) E_i  + \frac{m}{n} \sum_{i=1}^m E_i - \frac{m^2}{2n} \\
& =  \frac{(2-\mu)}{2\mu n}  \left(m+ \sum_{i=1}^m E_i \right) \sum_{i=1}^m (E_i - 1)  - \frac{2}{n\mu} \sum_{i=1}^m (m-i+1)(E_i -1) \\
& \qquad  + \frac{m}{n} \sum_{i=1}^m (E_i-1) + \frac{m}{n \mu}.
\end{align*}
So
\begin{align}
& A_m \I{\sum_{i=1}^m E_i \le Cn^{\alpha/(\alpha+1)}} \notag \\
&  \le \exp \left( \frac{1}{n\mu} \sum_{i=0}^m (s(i) - s(m))  - \frac{C_{\alpha} m^{\alpha+1}}{(\alpha+1) \mu^{\alpha+1} n^{\alpha}} + o(1) \right) \exp(\chi_n) \I{\sum_{i=1}^m E_i \le Cn^{\alpha/(\alpha+1)}},  \label{eqn:smallerexpression}
\end{align}
where
\begin{align*}
\chi_n & =  \frac{1}{n\mu} \sum_{i=1}^m (s(i-1) - s(m)) (E_i-1) - \frac{2}{n\mu} \sum_{i=1}^m (m-i+1)(E_i -1) + \frac{m}{n} \sum_{i=1}^m (E_i-1) \\
& \quad + \frac{(2-\mu)}{2\mu n}  \left(m+ \sum_{i=1}^m E_i \right) \sum_{i=1}^m (E_i - 1) -\frac{C_{\alpha}m^{\alpha+1}}{(\alpha+1)n^{\alpha} \mu^{\alpha+1}}\left(\left( \frac{1}{m} \sum_{i=1}^m E_i \right)^{\alpha+1} - 1 \right).
\end{align*}

We need to understand the asymptotics of the expectation of the right-hand side of (\ref{eqn:smallerexpression}). 
Recall that $s$ has steps down of magnitude at most 1, so that we have the crude bound $s(i) - s(m) \le m-i+1 \le m$ for all $0 \le i \le m$. So by Lemma~\ref{lem:exponentialmgf}(a), we get
\[
 \frac{1}{n \mu} \sum_{i=1}^m (s(i-1) - s(m)) (E_i-1) - \frac{2}{n \mu} \sum_{i=1}^m (m-i+1)(E_i-1) +  \frac{m}{n} \sum_{i=1}^m (E_i - 1) \to 0
\]
 in $L^2$, as $n \to \infty$. We have
\begin{align*}
& \E{ \left( \frac{(2-\mu)}{2\mu n}  \left(m+ \sum_{i=1}^m E_i \right)  \sum_{i=1}^m (E_i - 1)  \right)^2\I{\sum_{i=1}^m E_i \le Cn^{\alpha/(\alpha+1)}} } \\
& \quad \le \frac{(2-\mu)^2(T+C)^2 n^{2\alpha/(\alpha+1)}}{4 \mu^2 n^2} \E{\left(\sum_{j=1}^m (E_j-1) \right)^2} \le \frac{(T+C)^2T n^{3\alpha/(\alpha+1)}}{\mu^2 n^2} \to 0,
\end{align*}
 since $n^{\alpha/(\alpha+1)}/n^2 = n^{(\alpha-2)/(\alpha-1)} = o(1)$. If $m \to \infty$ as $n \to \infty$, it follows straightforwardly from the weak law of large numbers that
\[
\left(\frac{1}{m}  \sum_{i=1}^m E_i \right)^{\alpha+1} \convprob 1
\]
and so, for any $m  \le Tn^{\alpha/(\alpha+1)}$, we have
\[
\frac{C_{\alpha}m^{\alpha+1}}{(\alpha+1)n^{\alpha} \mu^{\alpha+1}} \left( \left( \frac{1}{m}\sum_{i=1}^m E_i \right)^{\alpha+1} - 1 \right) \convprob 0.
\]
These results imply that $\chi_n \convprob 0$ on $\{\sum_{i=1}^m E_i \le Cn^{\alpha/(\alpha+1)}\}$, as $n \to \infty$ and so
\[
\exp(\chi_n) \I{\sum_{i=1}^m E_i \le Cn^{\alpha/(\alpha+1)}} \convprob 1.
\]
It remains to show that
\[
\E{\exp(\chi_n) \I{\sum_{i=1}^m E_i \le Cn^{\alpha/(\alpha+1)}}} \to 1
\]
as $n \to \infty$, for which we require uniform integrability.  Now, we have
\begin{align*}
& \exp \left( \frac{(2-\mu)}{2\mu n}  \left(m+ \sum_{i=1}^m E_i \right) \sum_{i=1}^m (E_i - 1)  \right) \I{\sum_{i=1}^m E_i \le Cn^{\alpha/(\alpha+1)}} \\
& \qquad \qquad \le 1 + \exp \left( \frac{(2-\mu)}{2\mu n} (C+1)m  \sum_{i=1}^m (E_i-1)  \right).
\end{align*}
Hence,
\begin{align*}
& \exp(\chi_n ) \I{\sum_{i=1}^m E_i \le Cn^{\alpha/(\alpha+1)}} \\
&  \le \exp \left( \frac{C_{\alpha}m^{\alpha+1}}{(\alpha+1)\mu^{\alpha+1}n^{\alpha}} \right) \\
&  \times \left[\exp\left( \sum_{i=1}^m \left\{ \frac{(s(i-1) - s(m))}{n\mu} -\frac{2(m-i+1)}{n\mu} + \frac{m}{n} \right\}(E_i-1)  \right) \right.\\
&  \left. + \exp \left( \sum_{i=1}^m  \left\{ \frac{(s(i-1) - s(m))}{n\mu} -\frac{2(m-i+1)}{n\mu} + \frac{m}{n}  + \frac{(2-\mu)(C+1)m}{2\mu n} \right\}(E_i-1) \right) \right].
\end{align*}
Applying Lemma~\ref{lem:exponentialmgf}(b), we see that both terms are bounded in $L^p$ for $p > 1$.  Hence, the sequence $(\exp(\chi_n) \I{\sum_{i=1}^m E_i \le Cn^{\alpha/(\alpha+1)}}, \ n \ge 1)$
is uniformly integrable and we may deduce that
\[
\E{\exp(\chi_n) \I{\sum_{i=1}^m E_i \le Cn^{\alpha/(\alpha+1)}}} \to 1,
\]
which concludes the proof.
\end{proof}

We will also need the following lemma.

\begin{lem} \label{lem:excUI}
Fix $\theta > 0$ and let $m = \fl{T n^{\alpha/(\alpha+1)}}$.  Then we have the following uniform integrability: for $K > 0$,
\[
\limsup_{K \to \infty} \sup_{n \ge 1} \E{ \exp \left( \frac{\theta}{n} \sum_{i=0}^m \left( S(i) - S(m) \right) \right) \I{\frac{1}{n} \sum_{i=0}^m \left( S(i) - S(m) \right) > K} \Bigg| \mathcal{E}_m} =0.
\]
\end{lem}

\begin{proof}
The proof uses similar ingredients to the proof of Lemma~\ref{lem:expbound}.  On the event $\mathcal{E}_m$ we have
\[
\frac{\theta}{n} \sum_{i=0}^m (S(i) - S(m)) \le \frac{\theta m}{n} \left(1+ \max_{0 \le i \le m} S(i) \right) \le \theta T^{(\alpha+1)/\alpha} m^{-1/\alpha} \left(1+ \max_{0 \le i \le m} S(i) \right).
\]
So it will be sufficient to show that we have
\[
\limsup_{K \to \infty} \sup_{m \ge 1} \E{ \exp \left(\theta m^{-1/\alpha} \max_{0 \le i \le m} S(i) \right) \I{m^{-1/\alpha} \max_{0 \le i \le m} S(i)  > K} \Bigg| \mathcal{E}_m} = 0.
\]
We have
\begin{align*}
& \E{ \exp \left(\theta m^{-1/\alpha} \max_{0 \le i \le m} S(i) \right) \I{m^{-1/\alpha} \max_{0 \le i \le m} S(i)  > K} \Bigg| \mathcal{E}_m}  \\
& \quad = \sum_{k = \fl{K m^{1/\alpha}}+1}^{\infty} e^{\theta m^{-1/\alpha} k} \Prob{\max_{0 \le i \le m} S(i) = k \Big| \mathcal{E}_m} \\
& \quad \le e^{(K+1) \theta} \Prob{m^{-1/\alpha} \max_{0 \le i \le m} S(i) > K \Big| \mathcal{E}_m} \\
& \quad \qquad + \sum_{k = \fl{K m^{1/\alpha}}+2}^{\infty} \theta m^{-1/\alpha} e^{\theta m^{-1/\alpha} k} \Prob{\max_{0 \le i \le m} S(i) \ge k \Big| \mathcal{E}_m},
\end{align*}
by summation by parts and the fact that $e^{\theta m^{-1/\alpha} k} - e^{\theta m^{-1/\alpha}(k-1)} \le m^{-1/\alpha} \theta e^{\theta m^{-1/\alpha} k} $.
Theorem 9 of Kortchemski~\cite{KortchemskiSubExpTailBounds} gives that for any $\delta \in (0, \alpha/(\alpha-1))$, there exist universal constants $C_1, C_2 > 0$ such that
\[
\Prob{m^{-1/\alpha} \max_{0 \le i \le m} S(i) \ge u \Big| \mathcal{E}_m} \le C_1 \exp(-C_2 u^{\delta}).
\]
We take $\delta \in (1,\alpha/(\alpha-1))$.  So then
\begin{align*}
& \E{ \exp \left(\theta m^{-1/\alpha} \max_{0 \le i \le m} S(i) \right) \I{m^{-1/\alpha} \max_{0 \le i \le m} S(i)  > K} \Bigg| \mathcal{E}_m}  \\
& \qquad \le e^{(K+1) \theta} \Prob{m^{-1/\alpha} \max_{0 \le i \le m} S(i) > K \Big| \mathcal{E}_m} \\
& \qquad \qquad + \int_K^{\infty}  \theta e^{\theta x} \Prob{m^{-1/\alpha} \max_{0 \le i \le m} S(i) \ge x - 1 \Big| \mathcal{E}_m}  dx \\
& \qquad \le C_1 \exp((K+1) \theta - C_2 K^{\delta}) + \int_K^{\infty}  C_1 \theta \exp(\theta x - C_2 (x-1)^{\delta}) dx,
\end{align*}
which clearly tends to 0 as $K \to \infty$ since $\delta > 1$.  The result follows.
\end{proof}

\begin{proof}[Proof of Theorem~\ref{thm:convofexcursion}]
Recall that $S(k) = \sum_{i=1}^k (Z_i - 2)$, where $Z_1, Z_2, \ldots$ are i.i.d.\ with the size-biased degree distribution. 
Then we have
\begin{align*}
& \E{f \left(n^{-\frac{1}{\alpha+1}} [\widetilde{S}^n(\fl{(t_1 + \cdot) n^{\frac{\alpha}{\alpha+1}}} )- \widetilde{S}^n(\fl{t_1 n^{\frac{\alpha}{\alpha+1}}})],  \widetilde{G}^n(\fl{(t_1 + \cdot) n^{\frac{\alpha}{\alpha+1}}})\right) \Big| \widetilde{\mathcal{E}}^n_{m_1, m_2} } \\
&= \frac{\E{f \left(n^{-\frac{1}{\alpha+1}} [\widetilde{S}^n(\fl{(t_1 + \cdot) n^{\frac{\alpha}{\alpha+1}}} )- \widetilde{S}^n(\fl{t_1 n^{\frac{\alpha}{\alpha+1}}})], n^{-\frac{\alpha-1}{\alpha+1}} \widetilde{G}^n(\fl{(t_1 + \cdot) n^{\frac{\alpha}{\alpha+1}}})\right)\mathbbm{1}_{\widetilde{\mathcal{E}}^n_{m_1, m_2} }}}{\E{\mathbbm{1}_{\widetilde{\mathcal{E}}^n_{m_1, m_2} }}}. 
\end{align*}
Using the change of measure, this is equal to
\begin{align}
& \tfrac{\E{\Phi(n,\fl{t_2 n^{\frac{\alpha}{\alpha+1}}})  f \left(n^{-\frac{1}{\alpha+1}} [S(\fl{(t_1 + \cdot) n^{\frac{\alpha}{\alpha+1}}})- S(\fl{t_1 n^{\frac{\alpha}{\alpha+1}}})], n^{-\frac{\alpha-1}{\alpha+1}} G(\fl{(t_1 + \cdot) n^{\frac{\alpha}{\alpha+1}}})\right)\mathbbm{1}_{\mathcal{E}_{m_1,m_2}}}}{\E{\Phi(n, \fl{t_2 n^{\frac{\alpha}{\alpha+1}}})\mathbbm{1}_{\mathcal{E}_{m_1,m_2}}}} \notag \\
& = \tfrac{ \E{\Phi(n,\fl{t_2 n^{\frac{\alpha}{\alpha+1}}}) f \left(n^{-\frac{1}{\alpha+1}} [S(\fl{(t_1 + \cdot) n^{\frac{\alpha}{\alpha+1}}}) - S(\fl{t_1 n^{\frac{\alpha}{\alpha+1}}})], n^{-\frac{\alpha-1}{\alpha+1}} G(\fl{(t_1+\cdot) n^{\frac{\alpha}{\alpha+1}}}) \right) \Big| \mathcal{E}_{m_1,m_2}} }{\E{\Phi(n, \fl{t_2 n^{\frac{\alpha}{\alpha+1}}}) \Big| \mathcal{E}_{m_1,m_2} }} \label{eqn:increments}.
\end{align}
By Proposition~\ref{prop:upperbound}, we have that on the event $\mathcal{E}_{m_1,m_2}$, 
\begin{equation} \label{ApproximateMeasureChange} 
 \Phi(n,\fl{t_2 n^{\frac{\alpha}{\alpha+1}}})  = (1 +o(1)) \exp \left(\frac{1}{\mu n} \sum_{i=0}^{m_2} (S(i)-S(m_2)) - \frac{C_\alpha t_2^{\alpha+1}}{(\alpha+1)\mu^{\alpha+1}} \right),
\end{equation}
where the $o(1)$ is uniform on $\mathcal{E}_{m_1,m_2}$.  But using that $S(m_2) = S(m_1) - 1$, we get
\begin{align*}
\Phi(n,\fl{t_2 n^{\frac{\alpha}{\alpha+1}}}) & = (1+o(1)) \exp \left(\frac{1}{\mu n} \sum_{i=0}^{m_1-1} (S(i)-S(m_1)) + \frac{m_1}{\mu n}- \frac{C_\alpha t_2^{\alpha+1}}{(\alpha+1)\mu^{\alpha+1}} \right)  \\
 & \hspace{2cm} \times \exp \left(\frac{1}{\mu n} \sum_{i=m_1}^{m_2} ([S(i) - S(m_1)] - [S(m_2)-S(m_1)]) \right).
\end{align*}
The increments of the random walk $S$ are independent, and so the first and second terms in this product are independent. The first term is also independent of the argument of the function $f$.  So in both the numerator and denominator of the fraction (\ref{eqn:increments}), we may cancel a factor of
\[
\E{\exp \left(\frac{1}{\mu n} \sum_{i=0}^{m_1-1} (S(i)-S(m_1)) + \frac{m_1}{\mu n}- \frac{C_\alpha t_2^{\alpha+1}}{(\alpha+1)\mu^{\alpha+1}} \right) \Bigg| \mathcal{E}_{m_1,m_2}}.
\]
Using also the stationarity of the increments of $S$ and the fact that $m=m_2-m_1$, we then obtain that (\ref{eqn:increments}) is equal to $(1+o(1))$ times
\[
\tfrac{\E{ \exp \left(\frac{1}{\mu n} \sum_{i=0}^{m} (S(i) - S(m)) \right)  f \left(n^{-\frac{1}{\alpha+1}} S(\fl{t n^{\frac{\alpha}{\alpha+1}}}) , n^{-\frac{\alpha-1}{\alpha+1}} G(\fl{t n^{\frac{\alpha}{\alpha+1}}}), 0 \le t \le  n^{-\frac{\alpha}{\alpha+1}} m \right) \Big| \mathcal{E}_{m}} }{\E{\exp \left(\frac{1}{\mu n} \sum_{i=0}^{m} (S(i) - S(m)) \right) \Big| \mathcal{E}_{m} }}.
\]
Lemma~\ref{lem:excUI} gives us the requisite uniform integrability in order to now deduce the result from Theorem~\ref{thm:singletree} and the continuous mapping theorem.
\end{proof}

Let us briefly sketch how this result gives a scaling limit for a single component of $\mathbf{M}_n(\nu)$ or $\mathbf{G}_n(\nu)$ conditioned to have size $\fl{xn^{\alpha/(\alpha+1)}}$.  First note that for any $\epsilon > 0$ there exists a time $T > 0$ such that any component of size $\fl{x n^{\alpha/(\alpha+1)}}$ is discovered before time $\fl{Tn^{\alpha/(\alpha+1)}}$ with probability exceeding $1-\epsilon$, uniformly in $n$ sufficiently large.  Any such component discovered before time $\fl{Tn^{\alpha/(\alpha+1)}}$ corresponds to a tree of size $\approx x n^{\alpha/(\alpha+1)}$ in the forest encoded by $\widetilde{S}^n$ and $\widetilde{G}^n$ and, indeed, this tree is asymptotically indistinguishable in the Gromov--Hausdorff--Prokhorov sense from a spanning tree of the graph component.  The locations of the back-edges can then be handled in exactly the same way as in the unconditioned setting.
\end{appendix}
 \section*{Acknowledgements}
G.C. would like to thank Thomas Duquesne and Igor Kortchemski for useful discussions.  C.G.\ is very grateful to James Martin and Jon Warren for very helpful discussions, to Serte Donderwinkel for her careful reading of the paper which resulted in many improvements, and to Zheneng Xie for an improved proof of Lemma~\ref{lem:James}.  She would also like to thank Robin Stephenson, Jean Bertoin, Juan Carlos Pardo Mill\'an, Andreas Kyprianou and V\'ictor Rivero Mercado for advice and discussions.  Work on this paper was considerably facilitated by professeur invit\'e positions at LIX, \'Ecole polytechnique in November 2016 and November 2017 and at LAGA, Universit\'e Paris 13 in September 2018.  C.G.\ would like to thank Marie Albenque and LIX, and B\'en\'edicte Haas and LAGA for their hospitality, and the ANR GRAAL for funding. Both authors would like to express their gratitude to the referee for their careful and insightful reading of the paper, and for comments which led to many improvements.

C.G.'s research is supported by EPSRC Fellowship EP/N004833/1.
 
\bibliographystyle{amsplain}
\bibliography{stable}

\end{document}